\documentclass[11pt]{article}
\usepackage{color}
\usepackage{geometry}                
\usepackage[parfill]{parskip}    
\usepackage{graphicx}
\usepackage{subfig}	
\usepackage{amssymb}			
\usepackage{epstopdf}
\usepackage{amsmath,bbm}
\usepackage{tabularx,url}
\usepackage{algorithmicx}
\usepackage[ruled]{algorithm}
\usepackage{algpseudocode,enumerate}
\usepackage[numbers]{natbib}
\usepackage{amsthm,changepage,lipsum}
\usepackage{accents}
\usepackage{tikz}
\usetikzlibrary{positioning}
\usepackage{pgfplots}
\pgfplotsset{compat=newest}
\pgfplotsset{plot coordinates/math parser=false}
\usepackage{placeins}
\usepackage{todonotes}
\usepackage{cancel} 

\providecommand{\keywords}[1]{\textbf{{Keywords }} #1}
\providecommand{\subjclass}[1]{\textbf{{Mathematics subject classification }} #1}

\newtheorem{mydef}{Definition}
\newtheorem{prop}{Proposition}
\newtheorem{mytheorem}{Theorem}
\newtheorem*{mytheorem*}{Theorem}
\newtheorem{corollary}{Corollary}
\newtheorem*{corollary*}{Corollary}
\newtheorem{lemma}{Lemma}
\newtheorem{remark}{Remark}

\newcommand{\vect}[1]{\boldsymbol{#1}}

\numberwithin{equation}{section}

\def \sup {{\rm sup}}
\def \supp {{\rm supp}}
\def \zero {{\boldsymbol{0}}}
\def \a {{\boldsymbol{a}}}

\def \b {{\boldsymbol{b}}}
\def \c {{\boldsymbol{c}}}
\def \ct {{\boldsymbol{\tilde{c}}}}
\def \rt {{\boldsymbol{\tilde{r}}}}

\def \k {{\boldsymbol{k}}}
\def \l {{\boldsymbol{l}}}
\def \e {{\boldsymbol{e}}}
\def \ep {{\boldsymbol{e'}}}

\def \q {{\boldsymbol{q}}}
\def \r {{\boldsymbol{r}}}
\def \rp {{\boldsymbol{r'}}}
\def \rpp {{\boldsymbol{r''}}}
\def \rt {{\boldsymbol{\tilde{r}}}}
\def \n {{\boldsymbol{n}}}
\def \nt {{\boldsymbol{\tilde{n}}}}
\def \m {{\boldsymbol{m}}}
\def \u {{\boldsymbol{u}}}
\def \v {{\boldsymbol{v}}}
\def \x {{\boldsymbol{x}}}
\def \xib {{\boldsymbol{\xi}}}
\def \w {{\boldsymbol{w}}}
\def \y {{\boldsymbol{y}}}
\def \z {{\boldsymbol{z}}}
\def \supp {\text{supp}} 
\def \B {\mathcal B} 
\def \S {\mathcal S} 
\def \Q {\mathcal Q} 
\def \F {\mathcal F} 
\def \T {\mathcal T} 
\def \D {\mathcal D} 
\def \N {\mathbbm N} 
\def \C {\mathbbm C} 

\algrenewcommand\algorithmicindent{1.0em}%

\allowdisplaybreaks

\title{Sparse Harmonic Transforms II:  Best $s$-Term Approximation Guarantees for Bounded Orthonormal Product Bases in Sublinear-Time}

\author{Bosu Choi\thanks{University of Texas at Austin, Oden Institute for Computational Engineering and Sciences, \texttt{choibosu@utexas.edu}.} 
\and 
Mark Iwen\thanks{Michigan State University, Department of Mathematics, and the Department of Computational Mathematics, Science and Engineering (CMSE), \texttt{markiwen@math.msu.edu}.}
\and 
Toni Volkmer\thanks{Chemnitz University of Technology, Faculty of Mathematics, \texttt{toni.volkmer@math.tu-chemnitz.de}.} 
}
\date{}
\begin{document}

\maketitle
\vspace{-12pt}
\begin{abstract}
In this paper we develop a sublinear-time compressive sensing algorithm for approximating functions of many variables which are compressible in a given Bounded Orthonormal Product Basis (BOPB).  The resulting algorithm is shown to both have an associated best $s$-term recovery guarantee in the given BOPB, and also to work well numerically for solving sparse approximation problems involving functions contained in the span of fairly general sets of as many as $\sim10^{230}$ orthonormal basis functions.  All code is made publicly available.

As part of the proof of the main recovery guarantee new variants of the well known CoSaMP algorithm are proposed which can utilize any sufficiently accurate support identification procedure satisfying a {Support Identification Property (SIP)} in order to obtain strong sparse approximation guarantees.  These new CoSaMP variants are then proven to have both runtime and recovery error behavior which are largely determined by the associated runtime and error behavior of the chosen support identification method.  The main theoretical results of the paper are then shown by developing a sublinear-time support identification algorithm for general BOPB sets which is robust to arbitrary additive errors.  Using this new support identification method to create a new CoSaMP variant then results in a new robust sublinear-time compressive sensing algorithm for BOPB-compressible functions of many variables.
\end{abstract}



\section{Introduction} \label{sec:Introduction}
\setcounter{equation}{0}

In this paper we focus on rapidly computing best $s$-term approximations in the sense of compressive sensing \cite{cohen2009compressed,foucart2013mathematical} for functions of many variables $f: \D \subseteq \mathbbm{R}^D \rightarrow \mathbbm{C}$.  More specifically, we develop a numerical method that aims to very quickly approximate any given function $f$ using a near optimal $s$-sparse linear combination of $N'$ fixed basis functions $\B = \{ \b_1, \dots, \b_{N'} \}$ chosen in advance.  The developed method has two basic components:  $(i)$ a low-cardinality grid of evaluation points, $\mathcal{G} \subset \D$, and $(ii)$ a fast deterministic algorithm $\mathcal{H}: \mathbbm{C}^{|\mathcal{G}|} \rightarrow {\rm span}(\B)$ which takes $f$ evaluated on $\mathcal{G}$, $f(\mathcal{G}) \in \mathbbm{C}^{|\mathcal{G}|}$, as input, and then outputs an accurate $s$-sparse approximation to $f$ of the form $\sum_{j=1}^s a_{\ell_j} \b_{\ell_j} \in {\rm span}(\B)$.  In particular, we require that $\mathcal{H}$ can approximate all functions $f$ near-optimally based only on the evaluations of $f$ on $\mathcal{G}$ so that 
\begin{equation}
\left\| f - \mathcal{H}\left(f(\mathcal{G}) \right) \right\| ~\lesssim~ C~ \inf_{\small \z \in \mathbbm{C}^{N'},~\| \z \|_0 \leq s} \left\| f - \sum^{N'}_{j=0} z_j \b_j \right\|
\label{equ:GenericSparseApprox}
\end{equation}
holds with respect to suitable norms for all functions $f: \D \rightarrow \mathbbm{C}$ in a sufficiently general function class.

Note that we are requiring several strong properties of both $\mathcal{G}$ and $\mathcal{H}$ above.  First, we want the approximation algorithm $\mathcal{H}$ to succeed {\it for all} functions $f$ in a suitably large class when only given access to function evaluations of each $f$ on the same fixed and nonadaptive grid $\mathcal{G}$.  Second, we require that $\mathcal{H}$ is {\it fast}, which will mean in this paper that we require it to use a total number of scalar arithmetic and read/write operations that scales {\it sublinearly} with respect to the basis size $N'$ (e.g., herein we will focus on methods with runtimes that scale like $\mathcal{O}(\log^c (N'))$ for all sufficiently small sparsities $s$).  This second requirement has several other beneficial repercussions beyond computational speed.  Principally among them is the fact that the deterministic procedure $\mathcal{H}$ can use at most $o(N')$ function evaluations since its fast runtime constrains the number of function evaluations $\mathcal{H}$ can use.  This effectively constrains the size of the nonadaptive grid $\mathcal{G} \subset \D$ that it makes sense to use in the first place.  Similarly, any such $\mathcal{H}$ must also have low $o(N')$ memory requirements given that it only has time to perform $o(N')$-total scalar operations involving memory accesses.  
  
The first work on sublinear-time algorithms $\mathcal{H}$ of this kind focused almost exclusively on the one-dimensional Fourier basis where, e.g., $\B = \{ \mathbbm{e}^{2 \pi \mathbbm{i} \omega x} ~\big|~ \omega \in \left(\lceil N'/2 \rceil, \lfloor N'/2 \rfloor \right] \cap \mathbbm{Z}\}$.  The first of these \cite{Mansour:1992:RIA:646246.684842,gilbert2005improved,iwen2007empirical,hassanieh2012simple} were randomized algorithms which used grids $\mathcal{G}$ that varied from function to function and that failed with some nonzero probability for each given $f$.  All of these methods have runtimes that scale like $\mathcal{O}(s \log^c N')$ and achieve approximation errors along the lines of \eqref{equ:GenericSparseApprox} with high probability (w.h.p.) for each given $f: [0,1] \rightarrow \mathbbm{C}$.  Later on, entirely deterministic and explicit $\mathcal{O}(s^2 \log^c N')$-time methods $\mathcal{H}$ were then devised which use one fixed and nonadaptive grid $\mathcal{G}$ in order to guarantee approximation errors of the form \eqref{equ:GenericSparseApprox} for all sufficiently smooth and periodic functions $f: [0,1] \rightarrow \mathbbm{C}$ (see \cite{iwen2008deterministic,iwen2010combinatorial,bailey2012design,segal2013improved}).  These deterministic methods were then randomized in \cite{merhi2017new} to achieve highly efficient $\mathcal{O}(s \log^c N')$-time randomized discrete Fourier transform methods (generally known as ``sparse Fourier transforms") with high probability best $s$-term approximation guarantees \eqref{equ:GenericSparseApprox} along the lines of the first methods mentioned above, as well as sped up to produce entirely deterministic methods that are significantly faster than the generic $\mathcal{O}(s^2 \log^c N')$-time deterministic algorithms for periodic functions $f: [0,1] \rightarrow \mathbbm{C}$ which exhibit structured sparsity in the Fourier domain \cite{bittens2019deterministic}.  Code for many of these methods is publicly available\footnote{The code for an implementation of \cite{merhi2017new} is available at \url{https://sourceforge.net/projects/aafftannarborfa/}.  The code for an implementation of \cite{bittens2019deterministic} is available at \url{https://www.math.msu.edu/~markiwen/Code/FAST_block_sparse.zip}.}, and a nice survey article covering the standard techniques used to construct many of these first sublinear-time Fourier methods appeared in 2014 \cite{gilbert2014recent}.

As sublinear-time methods for the one-dimensional Fourier basis started to mature, similar algorithms began to be developed for other one-dimensional bases $\B$ as well, including for the cosine, Chebyshev, and Legendre polynomial bases \cite{hu2017rapidly,bittens2018sparse} (see also \cite{R12} for traditional compressive sensing methods which focus on the Legendre polynomial basis).  Recently these ideas have been extended yet further to produce sublinear-time algorithms with reconstruction guarantees for restricted classes of signals exhibiting approximate sparsity in any given one-dimensional Jacobi polynomial basis \cite{gilbert2019sparse}.  Another direction of research has focused on extending the types of sparse approximation algorithms discussed above to higher dimensional settings in order to approximate, e.g., functions $f:[0,1]^D \rightarrow \mathbbm{C}$ with respect to either  multidimensional Fourier \cite{R07,iwen2013improved,potts2016sparse,choi2016multi,kapralov2016sparse,morotti2017explicit,kammerer2017high,2019arXiv190703692C,2019arXiv190210633K} or Chebyshev \cite{potts2017multivariate} bases $\B$ of cardinality $N' = N^D$.  In these cases achieving fast algorithms $\mathcal{H}$ that run in $o(N')$-time becomes increasing important as $D$ grows.

As in the one-dimensional setting, sublinear-time methods $\mathcal{H}$ for approximating functions of $D$ variables are most well developed in the case of the multidimensional Fourier basis where, e.g., $\B = \{ \mathbbm{e}^{2 \pi \mathbbm{i} \boldsymbol{\omega} \cdot {\bf x}} ~\big|~ \boldsymbol{\omega} \in \left(\lceil N'/2 \rceil, \lfloor N'/2 \rfloor \right]^D \cap \mathbbm{Z}^D\}$.  For example, see Theorem 8 in \cite{iwen2013improved} and Theorems 10 and 12 in \cite{morotti2017explicit} for explicit and deterministic $\mathcal{O}(s^2 \log^c(N'))$-time methods that use function evaluations on a single fixed and nonadaptive grid $\mathcal{G} \subset [0,1]^D$ in order to guarantee approximation errors of the form \eqref{equ:GenericSparseApprox} for all sufficiently smooth and periodic functions $f:[0,1]^D \rightarrow \mathbbm{C}$.  When it comes to approximating functions of many variables with respect to non-Fourier bases $\B$ in $o(N')$-time, however, very little is currently known.  The first result in this direction \cite{choi2018sparse} provided sublinear-time recovery guarantees for all functions that are {\it exactly $s$-sparse}\footnote{A function is exactly $s$-sparse in $\B$ if it is a linear combination of $\leq s$ unknown elements of $\B$.} in any tensor product basis $\B$ of one-dimensional bounded orthonormal bases.  The aim of this paper is to augment this first general result with best $s$-term approximation guarantees along the lines of \eqref{equ:GenericSparseApprox} while maintaining its fast runtime and small fixed and nonadaptive grid size.  In doing so the authors aim to complement existing compressive sensing approaches for uncertainty quantification and function approximation \cite{schwab2006karhunen,chkifa2016polynomial,adcock2017infinite,2017arXiv170101671B,Adcock2017} with a new class of methods whose runtimes scale sublinearly in the basis size used for approximation.  These new methods will then hopefully allow for the extension of such techniques to, e.g., functions of hundreds or even thousands of variables in a more computationally feasible fashion.

\subsection{Setup and Main Results}
\label{sec:Setup}

Let $L^2(\D,\mu)$ for $\D := \times_{j\in [D]} \D_j \subset \mathbbm{R}^D$ denote all functions $f: \D \rightarrow \mathbbm{C}$ that are square-integrable with respect to a given product of probability measures $\mu := \times_{j\in [D]} \mu_j$ over $\D$, and suppose that you are given a countable orthonormal (with respect to $\mu$) basis, 
\begin{equation}
\B' := \left\{ T_{\n}: \D \rightarrow \C ~\big|~ \n \in \mathbbm{N}^D \right\},
\label{def:B'}
\end{equation}
of $L^2(\D,\mu)$ so that
\begin{equation*}
\langle T_{\k}, T_{\l}  \rangle_{(\D,\mu)} := \int_{\D} T_{\k}(\xib) \overline{T_{\l}(\xib)} d\mu(\xib) = \delta_{\k, \l} = 
\begin{cases}
1 & \textrm{if}~\k = \l \\
0 & \textrm{if}~\k \neq \l
\end{cases}.
\end{equation*}
Furthermore, suppose that $\B'$ is a tensor product basis so that
\begin{equation}
T_{\n}(\xib):=\prod _{j\in[D]} T_{j;n_j}(\xi_{j})
\label{def:T_n}
\end{equation}
holds for all $\n \in \mathbbm{N}^D$ and $\xib \in \D$, where each set 
$\B'_j := \left\{ T_{j;n_j}: \D_j \rightarrow \C ~\big|~ n_j \in \mathbbm{N} \right\}$
 with $j \in [D] := \{ 0, \dots, D-1 \}$ is itself an orthonormal (with respect to the probability measure $\mu_j$ over $\D_j \subset \mathbbm{R}$) basis of $L^2(\D_j,\mu_j)$.  We will call any such basis $\B'$ an Orthonormal Product Basis (OPB) with respect to $\mu$.

Our objective in this paper is to approximate smooth functions $f \in L^2(\D,\mu)$ as rapidly as possible using just a few point evaluations.  Toward this end we will take the traditional approach of considering only a finite subset $\B_{N,d}$ of $\B'$, and then approximating $f$ by approximating its projection $\tilde{f}$ onto the span of $\B_{N,d}$ (consider, e.g., hyperbolic cross/sparse grid methods for approximating functions of several variables \cite{shen2010sparse,dung2016hyperbolic,bungartz_griebel_2004}).  The potential improvement that the sublinear-time methods considered herein will then potentially provide over such standard methods in some cases will come from the fact that the finite basis $\B_{N,d}$ can be chosen to be extremely large herein (e.g., experiments were performed for Section~\ref{sec:Numerics} on a standard workstation using bases of cardinality $200^{100}$).  More specifically, herein we will consider two different types of bases $\B'$, each of which will allow us to demonstrate that the finite basis $\B_{N,d} \subset \B'$ we select below for approximation purposes also promote computational efficiency.  

We will characterize $\B'$ below based on the behavior of its lowest order elements 
\begin{equation*}
\B_N := \left\{ T_{\n} ~\big|~ \| \n \|_{\infty} < N \right\} \subset \B'
\end{equation*}
which we will assume throughout this paper is a finite Bounded Orthonormal System (BOS) with respect to the probability measure $\mu$ over $\D$ with a finite BOS constant 
$$K' :=\max_{\n \in [N]^D} \|T_{\n}\|_\infty := \max_{\n \in [N]^D} \sup_{\xib \in \D} \left| T_{\n} (\xib) \right| \in [1,\infty).$$  
Note that this implies that each set $\B_{j;N} := \left\{ T_{j;n_j}: \D_j \rightarrow \C ~\big|~ n_j \in [N] \right\} \subset \B'_j$ with $j \in [D]$ is itself also a BOS with respect to the probability measure $\mu_j$ over $\D_j \subset \mathbbm{R}$ with BOS constant 
\begin{equation}
K_j :=\max_{n_j \in [N]} \|T_{j;n_j}\|_\infty \in [1,\infty).
\label{def:KjBOSC}
\end{equation}
Finally, we will further define $K^0_j$ to be
\begin{equation} 
1 \leq K^0_j := \|T_{j;0}\|_\infty  \leq K_j
\label{def:K0jBOSC}
\end{equation}
for each $j \in [D]$.  Note that $K^0_j$ is strictly smaller than $K_j$ for many BOSs of interest (e.g., the cosine and Chebyshev polynomial bases as $\B'_j$ both have $K_j^0 = 1 < K_j = \sqrt{2}$). 
From these two definitions we can also see, e.g., that $\prod _{j\in[D]}  {K}^0_j \leq K' = \prod _{j\in[D]}  {K}_j$ always holds.  Due to the boundedness of $K'$ assumed throughout the remainder of this paper we will always refer to $\B_N$ (as well as $\B'$ with slight abuse) as a Bounded Orthonormal Product Basis (BOPB) going forward.

We will approximate any given smooth $f \in L^2(\D,\mu)$ by approximating its projection $\tilde{f}$ onto the span of the finite BOS set 
\begin{equation}
\B_{N,d} := \left \{ T_{\n}  ~\big|~ \n \in [N]^D ~{\rm and}~\| \n \|_0 \leq d \right \} \subseteq \B_N \subset \B'
\label{defBNd}
\end{equation}
for some $d \in [D+1]$, where $d$ is used to the help constrain the BOS constant.
The BOS constant $K \leq K'$ of $\B_{N,d}$ will be referred to as the {\it effective BOS constant} below.
As is usually the case in compressive sensing scenarios involving BOSs, its size will be a significant consideration with respect to sampling and computational efficiency.  In order to limit $K$'s size we will concentrate on the following two types of BOPBs going forward:
\begin{itemize}
\item \underline{\bf BOPBs of TYPE I:}  We will say a BOPB is of type I if the BOS constants $K_j$ are $1$ for all but at most $\tilde{d} \in \mathbbm{Z} \cap [0, D]$ BOS basis sets $\B_{j;N}$.  In this case we let $K_{\infty} := \max_{j\in [D]} K_j$ and note that $1 \leq K \leq K' \leq K_{\infty}^{\tilde{d}}$ so that $K$ will scale sub-exponentially in $D$ when $\tilde{d} \ll D$ independently of our choice of $d$ in \eqref{defBNd}.  We note that this type of BOPB includes several interesting examples of bases including the multidimensional Fourier basis (for which $\tilde{d} = 0$), and mixed BOPBs $\B'$ that have one-dimensional Fourier bases used for all but $\tilde{d}$ of their $\B'_j$ component bases.
\item \underline{\bf BOPBs of TYPE II:}  We will say a BOPB is of type II if $K_0 := \max_{j\in [D]} K^0_j = 1$.  This type of BOPB includes many bases where having a small number of interacting dimensions, $d$, helps to limit the effective BOS constant $K$ involved in the underlying sparse approximation problem.  Examples include the multivariate cosine, Chebyschev, and Legendre polynomial bases, as well as mixed polynomial bases where each one-dimensional component basis $\B'_j$ is, e.g., a potentially different Jacobi polynomial basis.
\end{itemize}

In either case above one can see that $K \leq K_{\infty}^d K^{D-d}_0$ will always hold.  In particular, $K = 2^{d/2}$ always holds if $\B'$ is either the multivariate cosine or Chebyshev basis in the type II case.  In the type I case we note that $1 \leq K \leq K_{\infty}^{\min(d,\tilde{d})} K^{\tilde{d}-\min(d,\tilde{d})}_0 \leq K_{\infty}^{\tilde{d}}$ will always hold so that $d$ can be set to $D$ without causing $K$ to become too large if, e.g., $\tilde{d} \ll D$.  This is certainly the case if $\B'$ is the multidimensional Fourier basis where $\tilde{d} = 0$.

Let $f \in L^2(\D,\mu)$ be smooth enough\footnote{Given that we will be recovering $f$ based on point samples we will require at least enough smoothness to guarantee that any particular point sample we might possibly utilize actually contains information about the given function's basis coefficients $\left\{ c_{\n} \right\}_{\n \in \mathbbm{N}^D}$.  Of course, the details regarding this smoothness requirement will vary with the choice of basis $\B'$.} that there exists a sequence $\left\{ c_{\n} \right\}_{\n \in \mathbbm{N}^D}$ such that
\begin{equation}
f(\xib) = \sum_{\n \in \mathbbm{N}^D} c_{\n} T_{\n}(\xib)
\label{equ:FunctoApprox}
\end{equation}
holds pointwise for all $\xib \in \D$.  Given such an $f: \D \rightarrow \mathbbm{C}$, we will denote its orthogonal projection onto the span of $\B_{N,d}$ by $\tilde{f}: \D \rightarrow \mathbbm{C}$.  Let 
$$\mathcal{I}_{N,d} := \left\{ \n \in [N]^D ~\big|~ \| \n \|_0 \leq d \leq D \right\} \subset \mathbbm{N}^D$$ 
be the set of indices corresponding to the basis elements in $\B_{N,d}$.  We then have that
\begin{equation}
\tilde{f}(\xib) := \sum_{\n \in \mathcal{I}_{N,d}} \tilde{c}_{\n} T_{\n}(\xib)
\label{equ:FinApproxf}
\end{equation}
for all $\xib \in \D$, where $\ct$ will be considered to be a vector in $\mathbbm{C}^{|\mathcal{I}_{N,d}|}$ indexed by $\mathcal{I}_{N,d}$.  Note further that the entries of $\ct$ will satisfy $\tilde{c}_{\n} = c_{\n}$ for all $\n \in \mathcal{I}_{N,d}$.

As mentioned above, we will ultimately approximate $f$ by producing a sparse approximation in $\B_{N,d}$ to $\tilde{f}$.  The best possible $s$-term approximation to $\tilde{f}$ in $\B_{N,d}$ will be denoted by $\tilde{f}_s^{\rm opt}: \D \rightarrow \mathbbm{C}$, and will be defined as follows:  Order the basis coefficients $\ct \in \mathbbm{C}^{|\mathcal{I}_{N,d}|}$ of $\tilde{f}$ by their magnitudes so that 
$$\left| \tilde{c}_{\n_1} \right| \geq \left| \tilde{c}_{\n_2} \right| \geq \left| \tilde{c}_{\n_3} \right| \geq \dots \geq  \left| \tilde{c}_{\n_{|\mathcal{I}_{N,d}|}} \right|,$$
where ties are broken lexicographically using the entries' indices in $\mathcal{I}_{N,d}$.  Then $\tilde{f}_s^{\rm opt}$ will be defined to be
$$\tilde{f}_s^{\rm opt} (\xib) := \sum_{j = 1}^s \tilde{c}_{\n_j} T_{\n_j}(\xib)$$  
for all $\xib \in \D$, and its (potentially) nonzero coefficients' indices will be denoted by 
$${\Omega^{\rm opt}_{\tilde{f},s}} := \left\{ \n_1, \dots, \n_s \right\} \subset \mathcal{I}_{N,d}.$$   

Note that $\tilde{f}_s^{\rm opt}$ will indeed have the property that
$$\left\|  \tilde{f} - \tilde{f}_s^{\rm opt} \right\|_{{L^2(\mathcal{D},\mu)}}~=~ \inf_{\small \z \in \mathbbm{C}^{|\mathcal{I}_{N,d}|},~\| \z \|_0 \leq s} \left\| \tilde{f} - \sum_{\n \in \mathcal{I}_{N,d}} z_{\n} T_{\n} \right\|_{{L^2(\mathcal{D},\mu)}}.$$
Furthermore, if we let $\ct_{\Omega^{\rm opt}_{\tilde{f},s}} \in \mathbbm{C}^{|\mathcal{I}_{N,d}|}$ denote the $\B_{N,d}$ basis coefficients of $\tilde{f}_s^{\rm opt}$ then we can see that both $\left\| \ct_{\Omega^{\rm opt}_{\tilde{f},s}} \right\|_0 \leq s$ and $\left\|  \ct - \ct_{\Omega^{\rm opt}_{\tilde{f},s}} \right\|_{2} ~=~\left\|  \tilde{f} - \tilde{f}_s^{\rm opt} \right\|_{{L^2(\mathcal{D},\mu)}}$ will hold.  As a result, norms involving the vector $\ct - \ct_{\Omega^{\rm opt}_{\tilde{f},s}} \in \mathbbm{C}^{|\mathcal{I}_{N,d}|}$ can be interpreted as best $s$-term approximation errors of $\tilde{f}$ in a natural way.

Finally, to prove our main result below we will effectively be considering the point samples we take from $f$ to instead be point samples taken from $\tilde{f}$ that are contaminated with evaluation errors of size $\left(f - \tilde{f} \right) (\xib)$ at each evaluation point $\xib \in \mathcal{G}$.  To bound all of these errors in a uniform fashion we will define 
\begin{equation}
\gamma := \left\| f - \tilde{f} \right\|_\infty ~=~ \sup_{\xib \in \D} \left|  \left(f - \tilde{f} \right) (\xib) \right|.
\label{equ:gammaDef}
\end{equation}
The following theorem is proven in Section~\ref{sec:CoSaMPguarantee}.\\

\begin{mytheorem}{(Main Result).}
Let $\eta \in (0, \infty)$ and $s,d,N \in \mathbbm{N} \setminus \{1 \}$ with $d \leq D$ and $s < |\mathcal{I}_{N,d}| / 2$.  There exists a finite set of grid points $\mathcal{G} \subset \D$, an algorithm $\mathcal{H}: \mathbbm{C}^{\left| \mathcal{G} \right|} \rightarrow \left( \mathcal{I}_{N,d} \times \mathbbm{C} \right)^s$, and an absolute universal constant $C' \in \mathbbm{R}^+$ such that the function $a: \D \rightarrow \mathbbm{C}$ defined by $a(\xib) := \sum_{(\n,a_{\n}) \in \mathcal{H}(f(\mathcal{G}))} a_\n T_\n(\xib)$ satisfies
\begin{equation}
\| f - a \|_{L^2(\mathcal{D},\mu)} \leq \left\| f - \tilde{f} \right\|_{L^2(\mathcal{D},\mu)} + C' \left( \sqrt{s} \left\| \ct - \ct_{\Omega^{\rm opt}_{\tilde{f},s}} \right\|_2 + \left\| \ct - \ct_{\Omega^{\rm opt}_{\tilde{f},s}}  \right\|_1 + \gamma \sqrt{s} \right) + \eta
\label{eqn:MRintroerrorg}
\end{equation}
for all $f = \sum_{\n \in \mathbbm{N}^D} c_{\n} T_{\n} \in L^2(\mathcal{D},\mu)$ with $\gamma := \left\| f - \tilde{f} \right\|_\infty$ $=~ \sup_{\xib \in \D} \left|  \left(f - \tilde{f} \right) (\xib) \right| ~<~ \infty$, where $\tilde{f}: \D \rightarrow \mathbbm{C}$ is the finite dimensional approximation to $f$ defined as per \eqref{equ:FinApproxf}.

	If the BOPB $\B_{N,d}$ is of type I so that the BOS constants $K_j$ are $1$ for all but at most $\tilde{d} \in \mathbbm{Z} \cap [0, D]$ BOS basis sets $\B_{j;N}$, then
	$$\left| \mathcal{G} \right| = \mathcal{O} \left( s^3 D K^{4\tilde{d}}_\infty d^4 \cdot  \log^4 \left( \frac{DN}{d} \right) \log^2 (s) \log^2 (D)  \right),$$
	and the algorithm $\mathcal{H}$ will have runtime complexity 
	$$\mathcal{O}\left( \left( s^5 + s^3 N \right) D^2 K^{4\tilde{d}}_\infty d^4 \cdot  \log^4 \left( \frac{DN}{d} \right) \log^2 (s) \log^2 (D) \log \left( \left\| \ct_{\Omega^{\rm opt}_{\tilde{f},s}} \right \|_2 / \eta\right) \right).$$ 	
	If the BOPB $\B_{N,d}$ is of type II so that $K_0 = 1$, then 
	$$\left| \mathcal{G} \right| = \mathcal{O} \left( s^3 D K^{4d}_\infty d^4 \cdot  \log^4 \left( \frac{DN}{d} \right) \log^2 (s) \log^2 (D) \right),$$
	and the algorithm $\mathcal{H}$ will have runtime complexity 
	$$\mathcal{O} \left( \left( s^5 + s^3 N \right) D^2 K^{4d}_\infty d^4 \cdot  \log^4 \left( \frac{DN}{d} \right) \log^2 (s) \log^2 (D) \log \left( \left\| \ct_{\Omega^{\rm opt}_{\tilde{f},s}} \right \|_2 / \eta\right) \right).$$
	Here we have assumed that any desired basis function $T_{\n} \in \B_{N,d}$ can be evaluated at any desired point in $\D$ in $\mathcal{O}(N D)$-time (which will be the case, e.g., for polynomial product bases of degree $\leq ND$).
	\label{thm:MRintro}
\end{mytheorem}

\begin{proof} This is a restatement of Corollary~\ref{cor:MainRes}.
\end{proof}

Looking at Theorem~\ref{thm:MRintro} we can see that it effectively subsumes the theoretical recovery results of \cite{choi2018sparse}.  Consider, for example, the case where $f = \tilde{f} = \tilde{f}_s^{\rm opt}$ (so that $f$ is exactly $s$-sparse in $\B_{N,d}$).  In this setting we will have both $\left\| f - \tilde{f} \right\|_{L^2(\mathcal{D},\mu)} = \gamma = 0$ and $\ct - \ct_{\Omega^{\rm opt}_{\tilde{f},s}} = {\bf 0}$ hold true so that \eqref{eqn:MRintroerrorg} implies that $f$ is recovered exactly (up to any chosen tolerance $\eta$).  Unlike the results in \cite{choi2018sparse}, however, Theorem~\ref{thm:MRintro} also guarantees that the method $\mathcal{H}$ will work well for functions $f$ which have $\| f - \tilde{f}_s^{\rm opt} \|_{L^2(\mathcal{D},\mu)}$ relatively small, but nonzero.

The authors would also like to emphasize the generality of Theorem~\ref{thm:MRintro}, which is unique to the best of their knowledge in the literature related to sublinear-time sparse approximation methods.  If, for example, one chooses $\B_{N,d}$ to be the multidimensional Fourier basis with $d = D$ and $\tilde{d} = 0$ then one immediately obtains a new sparse Fourier transform result for functions of many variables whose sampling and runtime requirements scale only polylogarithmically in the total basis size $|\B_{N,d}| = N^D$.  Though this new Fourier result does not compare favorably to the best deterministic multidimensional Fourier results of this kind \cite{iwen2013improved,morotti2017explicit} with respect to achievable runtimes or error guarantees, it also does not use any of the specific algebraic structure of the Fourier basis.  This allows Theorem~\ref{thm:MRintro} to be significantly more flexible than these older Fourier results in that it can apply to situations which they don't cover.  For example, it can generate entirely discrete Fourier results where $\mathcal{G} \subset \{ \frac{j}{N}~\big|~ j \in [N] \}^D$ (unlike \cite{iwen2013improved}) by using a discrete and finite multidimensional Fourier BOBP with $f = \tilde{f}$ which will work for any choice of $N$ (unlike \cite{morotti2017explicit}, which requires $N$ to be prime).

Finally, the astute reader has likely noticed that Theorem~\ref{thm:MRintro} is phrased in the form of an existence result, which may be troubling to the practical numerical analyst who actually wants to know how to compute an accurate solution.  Let us allay any anxieties that this choice of theoretical statement may have birthed -- the algorithm $\mathcal{H}$ referred to above is a modified version of the well known CoSaMP algorithm \cite{needell2009cosamp} (see Algorithm~\ref{alg:main} in Section~\ref{sec:CoSaMPguarantee}).  It has been implemented and evaluated in Section~\ref{sec:Numerics}, and the code made publicly available.\footnote{See ``SHT II: Best s-Term Approximation Guarantees for Bounded Orthonormal Product Bases in Sublinear-Time'' on Mark Iwen's code page \url{https://www.math.msu.edu/~markiwen/Code.html}.}  In short, the result is entirely explicit and constructive with respect to the algorithm $\mathcal{H}$.  The grid $\mathcal{G} \subset \D$, which ultimately responsible for the form of Theorem~\ref{thm:MRintro} as an existence result, on the other hand, is a bit more nuanced with respect to its practical construction.

As we shall see below, the grid $\mathcal{G} \subset \D$ is constructed by randomly selecting points from $\D$ according to several prescribed probability distributions that are ultimately derived from the orthogonality measure $\mu$ (see Theorems~\ref{thm:iterInvariant}~and~\ref{thm:SuppIDWorks} and their proofs for details).  It is then proven that this randomly constructed grid will allow $\mathcal{H}$ to satisfy the error guarantee \eqref{eqn:MRintroerrorg} for all functions $f$ as per \eqref{equ:FunctoApprox} with high probability while simultaneously satisfying the stated upper bounds on its cardinality.  The runtime complexity of $\mathcal{H}$ follows from the boundedness of $|\mathcal{G}|$.  Hence, the existence result is proven by randomly constructing a grid $\mathcal{G}$ which is guaranteed to satisfy the conclusions of Theorem~\ref{thm:MRintro} with high probability.  

In fact, this is completely analogous to the role of random sampling matrices in standard compressive sensing results involving the Restricted Isometry Property (RIP).  Many compressive sensing methods are guaranteed to be accurate if they are used in combination with a random sampling matrix that has the RIP, a condition which can only be achieved near-optimally with high probability.  Herein, the conclusions of Theorem~\ref{thm:MRintro} will hold for any grid $\mathcal{G}$ that can be used to form two associated random sampling matrices: one with the RIP, and another with a property known as the {\it Support Identification Property (SIP)} which will be defined in Section~\ref{subsec:SIP}.  The conclusions of Theorem~\ref{thm:MRintro} will hold whenever these two conditions are satisfied by $\mathcal{G}$, and it will be shown that a randomly constructed grid $\mathcal{G}$ will satisfy both conditions with high probability.

\subsection{An Outline of the Paper and of the Proof of Theorem~\ref{thm:MRintro}}

After reviewing some relevant compressive sensing results and establishing necessary notation in Section~\ref{sec:Notation}, we will begin proving Theorem~\ref{thm:MRintro} in Section~\ref{sec:CoSaMPguarantee}.  The first step in that process will be to prove a compressive sensing recovery guarantee for a generalized version of the well known CoSaMP method \cite{needell2009cosamp}.  This new theorem, Theorem~\ref{thm:iterInvariant}, will establish a best $s$-term recovery guarantee for the CoSaMP algorithm where the support identification step is performed by any algorithm $\mathcal{A}$ and grid $\mathcal{G}$ pair which has the SIP (see Section~\ref{subsec:SIP} below for details on the SIP).  With Theorem~\ref{thm:iterInvariant} in hand we will then turn our attention to constructing a sublinear-time algorithm $\mathcal{A}$ and grid $\mathcal{G}$ pair that have the SIP, an effort whose results are summarized by Theorem~\ref{thm:suppIdforsec3} (see also Proposition~\ref{coro:SIPrevealed}).  Combining Theorems~\ref{thm:iterInvariant}~and~\ref{thm:suppIdforsec3} then quickly establishes our main result above which appears in the form of Theorem~\ref{thm:NewCoSaMP} and Corollary~\ref{cor:MainRes} in Section~\ref{sec:CoSaMPguarantee}.

The vast majority of the effort in the paper will be focussed on proving Theorem~\ref{thm:suppIdforsec3} in Section~\ref{sec:SupportID}.  That is, to demonstrate that Algorithm~\ref{alg:suppid:impl} therein can be used together with a randomly constructed grid $\mathcal{G}$ in order to effectively achieve the SIP with high probability.  This is done by Theorem~\ref{thm:SuppIDWorks} (a specialized version of Theorem~\ref{thm:GenAlgSuppIDWorks} in Section~\ref{sec:SuppIDHeavyEls}) which formalizes the random sampling strategy one must use in order to construct $\mathcal{G}$ so that the SIP is achieved with high probability, and by Theorem~\ref{thm:suppId} which translates the conclusions of Theorem~\ref{thm:SuppIDWorks} into a SIP-type statement.  Theorem~\ref{thm:GenAlgSuppIDWorks}, in turn, follows from Theorem~\ref{thm:ExistFastFuncs} which is proven in Section~\ref{sec:proofBigSuppIDThm}.  

Finally, the authors would like to note that the reader who is interested in seeing the proof of Theorem~\ref{thm:MRintro} unfold from basic compressive sensing principals in a more direct fashion (though without the benefit of waypoints explaining the relevance of each result to the final goal) might consider the following alternate reading order for the sections below:  Such readers can begin with Section~\ref{sec:proofBigSuppIDThm} after reviewing Section~\ref{sec:Notation}, followed by the first 4 paragraphs of Section~\ref{sec:SupportID}, then Section~\ref{sec:SuppIDHeavyEls}, and finally the remainder of Section~\ref{sec:SupportID} after which Theorem~\ref{thm:suppIdforsec3} will have been proven.  Reading Section~\ref{sec:SupportID} in this bottom up fashion first will then allow Section~\ref{sec:CoSaMPguarantee} to be read without having to temporarily take any of the theoretical statements therein for granted along the way.  For readers who are mostly interested in the numerical ramifications of the methods developed herein, we suggest skipping down to conduct a careful review of Algorithms~\ref{alg:main}~and~\ref{alg:suppid:impl} (together with the equations referred to therein) after reading Section~\ref{sec:Notation}, after which the careful numerical evaluation conducted in Section~\ref{sec:Numerics} should be understandable.

Before moving on to establish some additional required notation, however, we will first discuss the SIP in the next subsection.  This is crucial as the notion of the SIP will allow for easier sublinear-time methods to be developed in the future.  To emphasize this last point:  Any basis for which the SIP can be established via a sublinear-time algorithm can be combined with Theorem~\ref{thm:iterInvariant} below in order to produce a new sublinear-time compressive sensing method for that basis.  We expect that this new pathway for developing future sublinear-time algorithms will help to stimulate the further improvement and generalization of sparse Fourier transform techniques to other bases of interest going forward.

\subsection{The Support Identification Property (SIP)}
\label{subsec:SIP}

As above, let $[N] := \{ 0, \dots, N-1 \}$ for all $N \in \N$ and further define $\mathcal{P}([N])$ to be the power set of any such set $[N]$.  In Section~\ref{sec:CoSaMPguarantee} we will prove that CoSaMP will still produce accurate sparse approximations as long as its support identification step employs a triple with the {\it support identification property}.\\

\begin{mydef}[The Support Identification Property (SIP)]
Let $s \in [N]$, $\beta \in (0,1)$, $\Phi \in \mathbbm{C}^{m \times N}$, $\mathcal{A}: \mathbbm{C}^m \rightarrow \mathcal{P}([N])$, and $\Gamma: \mathbbm{C}^m \rightarrow [0,\infty)$ with $\Gamma(\zero) = 0$.  The triple $(\Phi,\mathcal{A},\Gamma)$ is said to have the Support Identification Property (SIP) of order $(s,\beta )$ if 
$$\left\| \v_{\mathcal{A}(\Phi \v + \e)^c} \right\|_2 \leq \beta \| \v \|_2$$
holds for all $\e \in \mathbbm{C}^m$ and $\v \in \mathbbm{C}^N$ with $\| \v \|_0 \leq s$ that also satisfy $\| \v \|_2 > \Gamma(\e)$.
\label{def:SIP}
\end{mydef}

Note that many triples with the SIP exist.  One completely trivial example is the triple consisting of the $N \times N$ identify matrix $I$, the function $\mathcal{A}$ which always outputs $[N]$, and the zero function $\Gamma$.  Of course this example is extremely unsatisfying -- generally for compressive sensing applications we prefer that the any SIP triple $(\Phi \in \mathbbm{C}^{m \times N},\mathcal{A}: \mathbbm{C}^m \rightarrow \mathcal{P}([N]),\Gamma)$ has $m \ll N$ and an efficient computational complexity for $\mathcal{A}$ (preferably sublinear-in-$N$ herein).  Thankfully these types of SIP triples also exist -- in fact it is easy to see that any fast and error-robust compressive sensing algorithm $\mathcal{A}$ must in fact be a member of such a triple.\\

\begin{lemma}
Let $\Gamma': \mathbbm{C}^m \rightarrow \mathbbm{R}^+$ be such that $\Gamma'({\bf 0}) = 0$, and let $\mathcal{A}: \mathbbm{C}^m \rightarrow \mathbbm{C}^N$ be a compressive sensing algorithm with an associated measurement matrix $\Phi \in \mathbbm{C}^{m \times N}$ that satisfies 
$$\| \x - \mathcal{A} \left( \Phi \x + \e \right) \|_2 \leq \Gamma'(\e)$$
for all $\e \in \mathbbm{C}^m$ and $\x \in \mathbbm{C}^N$ with $\| \x \|_0 \leq s$.  Furthermore, let ${\rm supp}: \mathbbm{C}^N \rightarrow \mathcal{P}([N])$ output the indices of the nonzero entries of any given input vector.
Then, the triple $(\Phi,{\rm supp} \circ \mathcal{A},(1/ \beta') \Gamma')$ will have the SIP of order $(s,\beta')$ for all $\beta' < 1$.
\label{lem:EveryCSisSIP}
\end{lemma}

\begin{proof}
Let $\e \in \mathbbm{C}^m$ and note that
\begin{align*}
\left\|  \x_{{\rm supp}\left( \mathcal{A} \left( \Phi \x + \e \right) \right)^c} \right\|^2_2 ~&\leq~ \left\| \x_{{\rm supp}\left( \mathcal{A} \left( \Phi \x + \e \right) \right)^c} \right\|^2_2 + \left\| \left(\x - \mathcal{A} \left( \Phi \x + \e \right) \right)_{{\rm supp}\left( \mathcal{A} \left( \Phi \x + \e \right) \right)}  \right\|^2_2\\
~&=~ \left\|  \x - \mathcal{A} \left( \Phi \x + \e \right) \right\|^2_2 ~\leq~ \left( \Gamma'(\e) \right)^2.
\end{align*}
Thus, if $\| \x \|_2 > (1/\beta') \Gamma'(\e)$ then $\left\|  \x_{{\rm supp}\left( \mathcal{A} \left( \Phi \x + \e \right) \right)^c} \right\|_2 \leq \beta' \left((1/\beta')\Gamma'(\e) \right) \leq \beta' \| \x \|_2.$
\end{proof}

Lemma~\ref{lem:EveryCSisSIP} demonstrates that many nontrivial SIP triples of the type we are interested in exist.  Of course, using a compressive sensing method in order to create a SIP triple seems slightly nonsensical given that one would generally want to create a SIP triple in order to develop a new compressive sensing method in the first place.  This immediately raises the question of whether nontrivial SIP triples exist which do not in themselves already effectively serve as a compressive sensing method.  The answer to that question is ``yes'', and the easiest example is the SIP triple which the original CoSaMP method is itself already implicitly utilizes.  Given $\x \in \mathbbm{C}^N$ and $s \in  [N]$ let $\x \big|_{s} \in \mathbbm{C}^N$ be the vector obtained from $\x$ by setting all but its $s$-largest magnitude entries to $0$.  The following lemma explicitly demonstrates the SIP triple on which the original CoSaMP algorithm \cite{needell2009cosamp} is implicitly based.\\

\begin{lemma}[The CoSaMP SIP Triple]
Let $\Phi \in \mathbbm{C}^{m \times N}$ have the RIP of order $(2s, 0.1)$ (so that its Restricted Isometry Constants (RIC)s satisfy $\delta_{s} \leq \delta_{2s} \leq 0.1$), $s \in [N]$, $\beta \in (0.2223, 1)$, and define $\mathcal{A}: \mathbbm{C}^m \rightarrow \mathcal{P}([N])$ by $\mathcal{A}(\y) := {\rm supp} \left( (\Phi^* \y)\big|_{s}~ \right)$, and $\Gamma: \mathbbm{C}^m \rightarrow [0,\infty)$ by $\Gamma(\e) := \left( \frac{2.34}{\beta - 0.2223} \right) \| \e \|_2$.  Then, the triple $\left( \Phi,\mathcal{A}, \Gamma \right)$ has the SIP of order $(s,\beta )$.
\end{lemma}

\begin{proof}
Lemma 4.2 of \cite{needell2009cosamp} implies that $\left\|  \x_{\mathcal{A} \left( \Phi \x + \e \right)^c} \right\|_2 \leq 0.2223 \| \x \|_2 + 2.34 \| \e \|_2$ holds for all $\e \in \mathbbm{C}^m$ and $\x \in \mathbbm{C}^N$ with $\| \x \|_0 \leq s$.  Suppose, furthermore, that $\| \x \|_2 > \Gamma(\e) = \left( \frac{2.34}{\beta - 0.2223} \right) \| \e \|_2$. Then,
\begin{align*}
\left\|  \x_{\mathcal{A} \left( \Phi \x + \e \right)^c} \right\|_2 &~\leq~0.2223 \| \x \|_2 + 2.34 \| \e \|_2 ~<~ 0.2223 \| \x \|_2 + 2.34 \left( \frac{\beta - 0.2223}{2.34} \right) \| \x \|_2\\
&~= \beta \| \x \|_2 
\end{align*}
holds.  Also, $\Gamma({\bf 0}) = 0$.
\end{proof}

In Section~\ref{sec:CoSaMPguarantee} we will demonstrate that the original SIP triple implicitly used by the  CoSaMP algorithm can be replaced with any other SIP triple of similar quality without substantively changing the performance of the resulting CoSaMP variant as a compressive sensing algorithm.  Before we can do this, however, we will require some additional notation and preliminary infrastructural results.

\section{Notation and Preliminaries} \label{sec:Notation}

Recall that $D \in \N$ is the number of variables in the function of interest $f:  \times_{j\in [D]} \D_j \rightarrow \C$ (where $\D_j \subset \mathbbm{R}$ for all $j \in [D]$, and $\D := \times_{j\in [D]} \D_j$).  
Vectors $\n \in [N]^D$ with $\| \n \|_0 \leq d \leq D$ will always index a basis function in 
\begin{equation}
\B := \B_{N,d} = \left\{ T_{\n}: \D \rightarrow \C ~\big|~ \n \in [N]^D ~{\rm with}~\| \n \|_0 \leq d \right\},
\label{def:BOS_B}
\end{equation} 
where we have suppressed the basis subscripts for ease of discussion.
In addition, we further assume that the BOS $\B$ is a product basis so that $T_{\n}(\xib)$ satisfies \eqref{def:T_n} as above.

\subsection{Restrictions and Partial Evaluations}

The following notation will be utilized heavily during the analysis of the proposed support identification procedure.  Let $\S \subset [D]$, $\w \in \times_{j\in\S} \D_j$ with $w_j \in \D_j$, and $\n \in [N]^D$.  The function $T_{\S;\n}: \times_{j\in\S} \D_j \rightarrow \C$ is defined to be
\begin{equation}
T_{\S;\n}(\w) := \prod_{j \in \S} T_{j;n_j} (w_j).
\label{equ:RestictedBOS}
\end{equation}
Then, the set 
\begin{equation}
\B_{\S} := \left\{ T_{\S;\n} ~\big|~\n \in [N]^D ~{\rm with}~\| \n \|_0 \leq d \right\} 
\label{def:B_S}
\end{equation}
is a BOS with respect to the probability measure $\mu_{\S} := \otimes_{j \in \S} \mu_j$ over $\displaystyle \D_{\S} := \times_{j\in\S} \D_j \subset \mathbbm{R}^{|\S|}$ with BOS constant $K_{\S} \leq \min \left\{ \prod_{j \in \S} K_j, K_{\infty}^d K^{\max\{|\S|-d,0\}}_0 \right\}$.  
For any set $E$, let $\mathcal{P}(E)$ denote the power set of $E$ containing all possible subsets of $E$.
Given any vector $\v \in \C^p$ and $\T \subset [p]$ we will let $\v_{\T} \in \C^p$ have entries
\begin{equation*}
\left(v_{\T} \right)_j = 
\begin{cases}
v_j & \textrm{if}~j \in \T\\
0 & \textrm{if}~j \notin \T
\end{cases}.
\end{equation*}
For $t \in [p]$, we let $\v_t \in \C^p$ a vector restricting $\v$ to its $t$ largest-magnitude entries.
Let $\S^c := [D] \setminus \S$ for all $\S \subset [D]$.  We will then construct $f_{\S;\w}:  \D_{\S^c} \rightarrow \C$ from $f: \D \rightarrow \C$ by defining
\begin{equation*}
f_{\S;\w}(\z) := f(\xib)
\end{equation*}
where $\xib \in \D \subset \mathbbm{R}^D$ is the unique vector with $\xib_{\S} = \w$ and $\xib_{\S^c} = \z$.
In this context, we define the permutation function
$\varrho_\S\colon \D_{\S}\times\D_{\S^c}\rightarrow\D$ given by
\begin{equation}\label{def:rho}
\varrho_\S(\w, \z)=\xib \text{ such that } \xib_\S=\w \text{ and } \xib_{\S^c}=\z.
\end{equation}
This yields the alternative characterization $f_{\S;\w}(\z) = f\left(\varrho_\S(\w, \z)\right)$.

The restricted vectors of the input vector $\xib$ such as $\xib_\S$ and $\xib_{\S^c}$ have the reduced dimensions. However, the coefficient vectors such as $\ct$ and $\rt$ will maintain the full dimension even though they are restricted to some subset of indices.

If $\n, \m \in \mathcal{I}_{N,d}$ and $\S \subset [D]$ then we will define $(\n,\m)_\S \in \mathcal{I}_{N,d}$ to be the vector $\n_\S + \m_{\S^c}$.  Furthermore, for a given $\vect{v}\in \C^{|\mathcal{I}_{N,d}|}$, $\n \in \mathcal{I}_{N,d}$, and $\S \subset [D]$, we will let the vector $\vect{v}_{\S;\n} \in \C^{|\mathcal{I}_{N,d}|}$ indexed by $\vect{k} \in \mathcal{I}_{N,d}$ have entries given by
\begin{align}
\left({v}_{\S;\n}\right)_{\vect{k}}=
   \begin{cases}
       v_{\vect{k}},&\text{ if } \k_{\S} = \n_{\S} \\
       0 &\text{ otherwise}\\
   \end{cases}. \label{def:vecprefixfixed}
\end{align}
Note that $\vect{v}_{\S;\n} \in \C^{|\mathcal{I}_{N,d}|}$ will only have at most $N^{|\S^c|}$ nonzero entries corresponding to the entries of $\v \in \C^{|\mathcal{I}_{N,d}|}$, $v_{\m} \in \C$, whose indices $\m$ match those of $\n$ on $\S$ (i.e., so that $\m_{\S} = \n_{\S}$).

The following calculation will be repeated sufficiently often that it merits being referred to as a lemma.  It concerns the partial sum approximation to $f$ from \eqref{equ:FunctoApprox} in $\B$ given by \eqref{equ:FinApproxf}.
Recall that $\ct$ contains only the entries of the sequence $\c$ corresponding to the indices in $\mathcal{I}_{N,d}$.

Moreover, the next lemma also demonstrates the usage of the newly introduced notation. Its statement will be used later in the proofs of Lemmas ~\ref{lem:EntryID_ExactInnerProd} and~\ref{lem:SupportLem2}.
\begin{lemma}
Let $\S \subset [D]$, $\w \in \D_{\S} = \times_{j\in\S} \D_j$ with $w_j \in \D_j$, and $\n \in \mathcal{I}_{N,d}$.  Then 
$$\left \langle \tilde{f}_{\S;\w}, T_{\S^c;\n} \right \rangle_{\left( \D_{\S^c}, \mu_{\S^c}\right)} = \left \langle~ \ct_{\S^c;\n},  \overline{\Phi_{\S;\n;\w} } ~\right \rangle$$
where $\tilde{f}$ is as in \eqref{equ:FinApproxf}, and $\Phi_{\S;\n;\w} \in \C^{|\mathcal{I}_{N,d}|}$ is a vector indexed by $\vect{k} \in \mathcal{I}_{N,d}$ with entries
\begin{equation}
\left(\Phi_{\S;\n;\w}\right)_\k := 
\begin{cases}
T_{\S;\k}(\w) & \textrm{if}~\k_{\S^c} = \n_{\S^c}\\
0 & \textrm{otherwise}
\end{cases}.
\label{equ:RowRIPlem1}
\end{equation}
\label{lem:PartEvalinnerP}
\end{lemma}

\begin{proof}
Computing the inner product one quickly sees that
\begin{align*}
\left \langle \tilde{f}_{\S;\w}, T_{\S^c;\n} \right \rangle_{\left( \D_{\S^c}, \mu_{\S^c}\right)} &= \int_{\D_{\S^c}} \tilde{f}_{\S;\w}(\z) ~\overline{T_{\S^c;\n}(\z)}~d \mu_{\S^c}(\z)\\
&= \int_{\D_{\S^c}} \left( \sum_{\k \in \mathcal{I}_{N,d}} \tilde{c}_{\k} ~T_{\S;\k}(\w)~T_{\S^c;\k}(\z) \right) \overline{T_{\S^c;\n}(\z)}~d \mu_{\S^c}(\z)\\
&= \sum_{\k \in \mathcal{I}_{N,d}} \left(\tilde{c}_{\S^c;\n}\right)_{\vect{k}}~T_{\S;\k}(\w).
\end{align*}
The stated result follows.
\end{proof}

Let $m\in \mathbbm{N}$ and $\n \in \mathcal{I}_{N,d}$. For any matrix $A \in \mathbbm{C}^{m \times |\mathcal{I}_{N,d}|}$, we define $(A)_{\n}$ be the column of $A$ corresponding to the index $\n$. Also, we can choose multiple columns, e.g., for $\n_1, \n_2 \in \mathcal{I}_{N,d}$, $(A)_{\{\n_1,\n_2\}}$ refers the columns of $A$ corresponding to the indices $\n_1$ and $\n_2$.  More generally, for any $\S \subset \mathcal{I}_{N,d}$ the matrix $(A)_S = A_S \in \mathbbm{C}^{m \times |\S|}$ will consist of the columns of $A$ indexed by $\S$.

\subsection{Sampling Matrices associated to a BOS and Restricted Isometry Constants}

Given a BOS as in \eqref{def:BOS_B}, let $\{\xib_\ell\}_{\ell \in [m]} \subset \D$ be sampling points drawn independently at random according to $\mu$ with corresponding samples $\{\tilde{\y}_{\ell} :=\tilde{f}(\xib_\ell)\}_{\ell \in [m]}$ from $\tilde{f}$ in \eqref{equ:FinApproxf}.  The {\it random sampling matrix} $\Phi \in\C^{m\times |\mathcal{I}_{N,d}|}$ associated with the points $\{\xib_\ell\}_{\ell \in [m]}$ and the BOS has entries given by 
\begin{equation}
\Phi_{\ell,\n}=T_{\vect{n}}(\xib_\ell)
\label{def:A}
\end{equation}
with indices $\ell \in [m]$ and $\n \in \mathcal{I}_{N,d}$.  One can see that, e.g., $\tilde{\y} = \Phi \ct$ will hold in this case.  Furthermore, results from the compressive sensing literature guarantee that $\frac{1}{\sqrt{m}} \Phi$ will also have well-behaved {\it restricted isometry constants} as soon as $m$ is sufficiently large.\\

\begin{mydef}[See Definition 6.1 in \cite{foucart2013mathematical}]
The $s$-th restricted isometry constant $\delta_s$ of a matrix $A \in \C^{m \times N}$ is the smallest $\delta \geq 0$ such that 
\begin{equation*}
(1-\delta)\|\x\|_2^2\leq \|A\x\|_2^2 \leq (1+\delta)\|\x\|_2^2
\end{equation*}
holds for all $s$-sparse vectors $\x \in \C^N$.  The matrix $A$ is said to satisfy the restricted isometry property (RIP) of order $(s,\delta)$ if $1 > \delta \geq \delta_s \geq 0$.
\end{mydef}

\begin{mytheorem}[See Theorem 12.32 and Remark 12.33 in \cite{foucart2013mathematical}]
\label{thm:BOS_RIP}
Let $\Phi\in \C^{m\times \left|\mathcal{I}_{N,d} \right|}$ be the random sampling matrix associated to a {BOS} with constant $K\geq 1$ for $m,\left|\mathcal{I}_{N,d} \right|,s \in \mathbbm{Z}^+ \setminus \{1\}$. If, for $\delta,p \in (0,1)$, 
\begin{equation*}
m\geq a K^2 \delta^{-2}s \cdot \max \left\{ \ln^2(s) \ln \left( \left|\mathcal{I}_{N,d} \right| \right) \ln(m), \ln(p^{-1}) \right\},
\end{equation*}
then with probability at least $1-p$ the restricted isometry constant $\delta_s$ of $\frac{1}{\sqrt{m}}\Phi$ satisfies $\delta_s \leq \delta$ so that $\Phi$ has the RIP of order $(s,\delta)$. Here the constant $a \in \mathbbm{R}^+$ is universal.
\end{mytheorem}

In addition, one can also, e.g., bound the $\ell_2$ operator norm of the random sampling matrix $\Phi$ in the event that it has the RIP.  We have the following consequence of Theorem~\ref{thm:BOS_RIP}.\\

\begin{lemma}[See Proposition 3.5 in \cite{needell2009cosamp}]
\label{lem:OpBoundRSM}
Suppose $A \in \C^{m \times \left|\mathcal{I}_{N,d} \right| }$ has the restricted isometry property (RIP) of order $(s,\delta)$.  Then, 
$$\| A \x \|_2 \leq \sqrt{1 + \delta} \left( \frac{\| \x \|_1}{\sqrt{s}} + \| \x \|_2 \right)$$
holds for all $\x \in \C^{\left|\mathcal{I}_{N,d} \right|}$. 
\end{lemma}

We are now prepared to develop the new CoSaMP variants on which our new sublinear-time algorithms will be based.

\section{Robust Sublinear-Time Sparse Approximation via CoSaMP with Fast Support Identification}
\label{sec:CoSaMPguarantee}

\begin{algorithm}[h]
	\caption{CoSaMP with the new support identification}
	\label{alg:main}
	\begin{algorithmic}[1]
		\Procedure{$\mathbf{CoSaMPnewSupportID}$}{}\\
		{\textbf{Input: }}{${\y}_{\rm SID} = \Phi_{\rm SID}\x_s+\e_{\rm SID}$, $\Phi_{\rm SID}$, $\mathcal{A}$, ${\y}_{\rm CE} = \Phi_{\rm CE}\x_s+\e_{\rm CE}$, ${\Phi}_{\rm CE}$, $\kappa$, $s$, $d$, $\tilde{d}$}\\
		{\textbf{Output: }}{$s$-sparse approximation $\a$ of $\x$}
		\State ${\a}^0=\zero$ \hfill \{Initial approximation\}
		\State ${\v}_{\rm SID} \gets {\y}_{\rm SID}$
		\State $k\gets 0$
		\Repeat
		\State $k\gets k+1$
		\State $\widetilde{\Omega} \gets \mathcal{A}({\v}_{\rm SID})$ \label{suppIDcalled} \hfill \{$|\widetilde{\Omega}|\leq 2s$, New support identification step (e.g., Algorithm \ref{alg:suppid:impl})\} 
		\State $\Omega \gets \widetilde{\Omega} \cup {\rm supp}({\a}^{k-1})$ \hfill \{Merge supports\}
		\State $\Phi' \gets \frac{1}{\sqrt{m_{\rm CE}}}{\Phi}_{\rm CE} \large|_{\Omega}$
		\State $\b\large|_{\Omega} \gets ({\Phi'})^{\dagger} \frac{{\y}_{\rm CE}}{\sqrt{m_{\rm CE}}}$ \hfill \{approximated using 3 LS iterations (Richardson's or CG)\}
		\State $\a^k\gets (\b\large|_{\Omega})_s$ \hfill \{Prune to obtain next approximation\}
		\State $\v_{\rm SID}\gets {\y}_{\rm SID} - \Phi_{\rm SID} \a^{k}$ \label{sampleUpdateI} \hfill \{Update current samples I\}  
		\State $\v_{\rm CEold} \gets \v_{\rm CE}$, $\v_{\rm CE}\gets {\y}_{\rm CE} - \Phi_{\rm CE} \a^{k}$ \label{sampleUpdateII} \hfill \{Update current samples II\} 
		\Until{
			$\|\v_{\rm CE} \|_2^2 > \|\v_{\rm CEold}\|_2^2 $,
			or $k \geq \kappa$} 
		\hfill \{Halting criteria\}
		\State If $\|\v_{\rm CE} \|_2^2 > \|\v_{\rm CEold}\|_2^2 $ then $\a \gets \a^{k-1}$ else  $\a \gets \a^{k}$ 
		\EndProcedure
	\end{algorithmic}
\end{algorithm}

In this section we analyze a generalized CoSaMP variant which uses any support identification method satisfying the SIP introduced in Definition~\ref{def:SIP} above (see Algorithm~\ref{alg:main}). In Theorem~\ref{thm:iterInvariant} we provide error guarantees as well as the general sampling and runtime complexities that one can obtain for such CoSaMP variants with a particular choice of halting criteria.  Later, in Section~\ref{sec:SupportID}, we then propose a new admissible support identification method which runs in sublinear time for BOPBs with sufficiently small BOS constants (see Algorithm~\ref{alg:suppid:impl}).
This method is proven to satisfy the SIP as stated in Theorem~\ref{thm:suppIdforsec3} of this section. Finally, combining Algorithms~\ref{alg:main} and~\ref{alg:suppid:impl}, we obtain Theorem~\ref{thm:NewCoSaMP} which combines the error guarantees from Theorem~\ref{thm:iterInvariant} due to the SIP with the specific sampling and runtime complexities of the support identification algorithm presented in Section~\ref{sec:SupportID} for BOPBs.
We hasten to point out that the modularity of this proof approach makes it easier to improve upon than prior works have been.
If a better (e.g., faster) support identification method satisfying the SIP is developed for a particular basis in the future it can immediately replace the one from Section \ref{sec:SupportID} and produce an improved CoSaMP type algorithm with a better performance for that particular basis.

We assume herein 
that the function~$f$ in~\eqref{equ:FunctoApprox} can be written as
\begin{equation}
f := \tilde{f} + e'
\label{equ:fDefined}
\end{equation}
where $\tilde{f}: \D \rightarrow \C$ is as per \eqref{equ:FinApproxf} with the coefficient vector $\ct \in \C^{\mathcal{I}_{N,d}}$ in $\B$, and where $e': \D \rightarrow \C$ is bounded so that $\|e'\|_{\infty} \leq \gamma$. Now, we rewrite $f$ as
\begin{equation}
 f=\tilde{f}+e'=\tilde{f}^{\rm opt}_{s}+ 
 \underbrace{\left( \tilde{f} -  \tilde{f}^{\rm opt}_{s} +e'\right)}_{ \text{\normalsize $=:e$} }.
 \label{equ:fEquivalnce}
\end{equation}
Our goal is to approximate the best $s$-term approximation $\ct_{\Omega^{\rm opt}_{\tilde{f},s}}$ of $\ct$, which is the coefficient vector of $\tilde{f}^{\rm opt}_{s}$. Since CoSaMP from~\cite{needell2009cosamp} approximates the best $s$-term of a given vector efficiently while allowing mild noise on the samples, we modify the CoSaMP algorithm in order to make it handle our high-dimensional problem more efficiently.  Since the analysis of our CoSaMP type algorithm will be based on~\cite{needell2009cosamp},
it is helpful to introduce the connection between our notation and the notation from \cite{needell2009cosamp}.  Toward that end, going forward  
we will set $\x_s := \ct_{\Omega^{\rm opt}_{\tilde{f},s}}$, $\x := \ct$, and $\Phi := \frac{1}{\sqrt{m_{\rm CE}}}\Phi_{\rm CE}$ in the notation of \cite{needell2009cosamp}.
The samples $ \u = \Phi \x + \e $ in \cite{needell2009cosamp} can then be viewed as containing renormalized function evaluations of $f$, and accordingly, $\e$ contains renormalized function evaluations of the $e$ defined in \eqref{equ:fEquivalnce}. In particular, $\e = \frac{\e_{\rm CE}}{\sqrt{m_{\rm CE}}}$.

In the pseudocode of Algorithm~\ref{alg:main}, most of the steps are identical to the original CoSaMP except the ``New support identification step", ``Update current samples I \& II", and ``Halting criteria" lines. The inputs $\y_{\rm SID} \in \C^{m_{\rm SID}}$ and $\y_{\rm CE}\in \C^{m_{\rm CE}}$ of Algorithm~\ref{alg:main} contain function evaluations of~$f$ which will be used for support identification and coefficient estimation (i.e., via least squares), respectively.  Accordingly, $\e_{\rm SID}\in \C^{m_{\rm SID}}$ and $\e_{\rm CE}\in \C^{m_{\rm CE}}$ appearing in Theorem \ref{thm:iterInvariant} contain the corresponding function evaluations of $e$ from \eqref{equ:fEquivalnce}, and $\Phi_{\rm SID}\in \C^{m_{\rm SID} \times |\mathcal{I}_{N,d}|}$ and  $\Phi_{\rm CE} \in \C^{m_{\rm CE} \times |\mathcal{I}_{N,d}|}$ have the function evaluations of $T_{\n}$ for $\n \in \mathcal{I}_{N,d}$ at the corresponding evaluation points.  Note that $\e_{\rm SID}$ and $\e_{\rm CE}$ do not change over the iterations of Algorithm~\ref{alg:main}. We define $\ep_{\rm SID}\in \C^{m_{\rm SID}}$ and $\ep_{\rm CE}\in \C^{m_{\rm CE}}$ as the vectors whose entries are the function evaluations of $e'$ from \eqref{equ:fDefined}. Each row number ($m_{\rm SID}$ and $m_{\rm CE}$) is, therefore, the total number of function evaluations used for support identification and the coefficient estimation, respectively. 
In the $k$-th iteration, Algorithm~\ref{alg:main} starts with an $s$-sparse approximation $\a^{k-1}$ of $\x_s$ and then tries to approximate the at most $2s$-sparse residual vector $\r^{k-1}:=\x_s-\a^{k-1}$. The support identification procedure $\mathcal{A}$ in the ``New support identification" step begins approximating $\r^{k-1}$ by finding a support set $\tilde{\Omega} \subset \mathcal{I}_{N,d}$ of cardinality at most $2s$ which contains the indices of the entries where most of the energy of $\r^{k-1}$ is located.  As noted above, any support identification method satisfying the SIP can substitute the ``New support identification step" in Algorithm \ref{alg:main} in order to accomplish this task -- the algorithm developed and analyzed in Section \ref{sec:SupportID} is a specific instance.

After the support identification, in the ``Merge supports" step, a new support set $\Omega$ of cardinality at most $3s$ is then formed from the union of $\tilde{\Omega}$ with the support of the current approximation $\a^{k-1}$.  At this stage $\Omega$ should contain the overwhelming majority of the important (i.e., energetic) index vectors for $\r^{k-1}$.  
As a result, restricting the columns of the sampling matrix $\Phi_{\rm CE}$ to those in $\Omega$ (or constructing them on the fly in a low memory setting) in order to solve for 
$\b_{\Omega} := {\rm argmin}_{\u \in \C^{|\Omega|} } \frac{1}{\sqrt{m_{\rm CE}}} \left\| \left({\Phi}_{\rm CE} \right)_{\Omega} \u-\y_{\rm CE} \right\|_2$
should yield accurate estimates for the true coefficients of $\x_s = \ct_{\Omega^{\rm opt}_{\tilde{f},s}}$ indexed by the elements of $\Omega$, $\left(\ct_{\Omega^{\rm opt}_{\tilde{f},s}}\right)_{\Omega}$.\footnote{In practice, it suffices to  approximate the least-squares solution $\b_{\Omega}$ by an iterative least-squares approach such as Richardson's iteration or conjugate gradient \cite{doi:10.1137/1.9781611971484,Dahlquist:2008:NMS:1383510} since computing the exact least squares solution can be expensive when $s$ is large. The argument of \citep{needell2009cosamp} shows that it is enough  to take three iterations for Richardson's iteration or conjugate gradient if the initial condition is set to $\a^{k-1}$, and if $\Phi_{\rm CE}$ has an RIP constant $\delta_{2s}<0.025$. In fact, both of these methods have similar runtime performance.}  
The vector $\left( \b_{\Omega} \right)_s$ then becomes the next approximation~$\a^k$ of~$\x_s = \ct_{\Omega^{\rm opt}_{\tilde{f},s}}$.
Theorem \ref{thm:iterInvariant} provides the error guarantees for $\left\| \ct_{\Omega^{\rm opt}_{\tilde{f},s}} - \a \right\|_2$, as well as the runtime complexity of Algorithm \ref{alg:main} in terms of the provided support identification algorithm's runtime.\\

\begin{mytheorem}
	Let $\bar{\Gamma}\geq 0$, $\beta \in(0,0.2228]$, $\kappa \in \N$ and $\delta \in (0,0.025]$ be fixed, and let $K$ be the BOS constant of \eqref{defBNd}. Suppose that $\x_s:=\ct_{\Omega^{\rm opt}_{\tilde{f},s}}$ is $s$-sparse with $\y_{\rm SID}=\Phi_{\rm SID}\x_s+\e_{\rm SID}$ and $\y_{\rm CE}=\Phi_{\rm CE}\x_s+\e_{\rm CE}$ where the triple $\big(\Phi_{\rm SID} \in \mathbbm{C}^{m_{\rm SID} \times |\mathcal{I}_{N,d}|},\mathcal{A}: \mathbbm{C}^{m_{\rm SID}} \rightarrow \mathcal{P}([|\mathcal{I}_{N,d}|]),\Gamma: \mathbbm{C}^{m_{\rm SID}} \rightarrow [0,\infty) \big)$ has the SIP of order $(2s,\beta )$ and $\Gamma(\e_{\rm SID}) \leq \bar{\Gamma}$, and $\frac{1}{\sqrt{m_{\rm CE}}} \Phi_{\rm CE} \in \mathbbm{C}^{m_{\rm CE} \times |\mathcal{I}_{N,d}|}$ has RIP constant $\delta_{2s} \leq \delta$ and $m_{\rm CE} = \mathcal{O}(s K^2 \log^4 |\mathcal{I}_{N,d}|)$. 
	Suppose that $\r^k :=\x_s - \a^k$ is a $2s$-sparse vector such that $\v_{\rm SID} = \Phi_{\rm SID} (\x_s - \a^k) + \e_{\rm SID} = \Phi_{\rm SID} \r^k + \e_{\rm SID}$. 
	Furthermore, suppose that the support identification procedure $\mathcal{A}$'s output always has cardinality at most $2s$, that it runs in $\mathcal{O}(\mathcal{L}_\mathcal{A})$-time, and that it uses $m_{\rm SID} = \mathcal{O}\left(\mathcal{L}'_\mathcal{A}\right)$ function evaluations.  Then, for all $k \geq 0$, the signal approximation $\a^k$ in Algorithm~\ref{alg:main} is $s$-sparse and satisfies
	\begin{equation*}		
	\left\| \x_s - \a^{k+1} \right\|_2 \leq 0.5 \left\|	\x_s - \a^k \right\|_2 + \frac{2.124}{\sqrt{m_{\rm CE}}} \| \e_{\rm CE} \|_2,
	\end{equation*}
	as long as $\left\| \r^k \right\|_2 = \left\| \x_s - \a^k \right\|_2 > \bar{\Gamma}$.
	In particular, if $\min_{j \in [k]} \left\| \r^j \right\|_2 > \bar{\Gamma}$ then
	\begin{equation}	
	\left\| \x_s - \a^k \right\|_2 \leq  2^{-k} \| \x_s \|_2 +  \frac{4.248}{\sqrt{m_{\rm CE}}} \| \e_{\rm CE} \|_2.
	\label{eqn:errbd}
	\end{equation}
	As a consequence, CoSaMP with any such support identification method $\mathcal{A}$ will produce an $s$-sparse approximation $\a$ that satisfies
	\begin{equation}
	\| \x_s - \a\|_2 \leq \max  \left\{   1.03   \bar{\Gamma} + 2.03 \frac{\| \e_{\rm CE} \|_2}{{\sqrt{m_{\rm CE}}}} , ~2^{-\kappa}\|\x_s\|_2 + 4.3 \frac{\|\e_{\rm CE} \|_2}{\sqrt{m_{\rm CE}}}, ~8.625 \frac{\|\e_{\rm CE} \|_2}{\sqrt{m_{\rm CE}}} \right\}.
	\label{equ:sparseErrBd}
	\end{equation}
The sampling complexity of Algorithm~\ref{alg:main} will be $m_{\rm SID} + m_{\rm CE}= \mathcal{O}\left(\mathcal{L}'_\mathcal{A}+s K^2 \log^4 |\mathcal{I}_{N,d}| \right)$.
	The runtime complexity of Algorithm~\ref{alg:main} will be $\mathcal{O} \big( \left (\mathcal{L}_\mathcal{A} + s^2 K^2 \mathcal{L}_{\Phi} \log^4 |\mathcal{I}_{N,d}| + s  \mathcal{L}_{\Phi} m_{\rm SID} \right)  \cdot \kappa \big)$, where $\mathcal{O}\left(\mathcal{L}_{\Phi} \right)$ is the runtime complexity of computing any desired matrix entry $\left( \Phi_{\rm CE} \right)_{j,\ell}$, or $\left( \Phi_{\rm SID} \right)_{j,\ell}$, for any valid choice of $j,\ell$.  
	\label{thm:iterInvariant}
\end{mytheorem}

\begin{proof}
	When $k < \kappa$ and $\| \r^k \|_2 > \bar{\Gamma} \geq \Gamma(\e_{\rm SID})$, we obtain
	\begin{align*}
	\left\| \x_s - \a^{k+1} \right\|_2 &\leq 2 \| \x_s - \b \|_2 \tag{Lemma 4.5 in \cite{needell2009cosamp}, $\a^{k+1} = (\b\large|_{\Omega})_s$, $\b = \b\large|_{\Omega}$}\\
	&\leq 2.224 \left\| (\x_s)_{{\Omega}^c} \right\|_2 + 0.0044 \| \r^k \|_2 + \frac{2.124}{\sqrt{m_{\rm CE}}} \| \e_{\rm CE} \|_2 \hspace{.07in} \tag{Corollary 5.3 in \cite{needell2009cosamp}}\\
	&\leq 2.224 \left\| \r^k_{\tilde{\Omega}^c} \right\|_2 + 0.0044 \| \r^k \|_2 + \frac{2.124}{{\sqrt{m_{\rm CE}}}} \| \e_{\rm CE} \|_2  \tag{Lemma 4.3 in \cite{needell2009cosamp}}\\
	&\leq 2.224 \cdot  \beta \| \r^k \|_2   + 0.0044 \| \r^k \|_2 + \frac{2.124}{{\sqrt{m_{\rm CE}}}} \| \e_{\rm CE} \|_2 \tag{The SIP assumption}\\
	&\leq 0.5 \| \r^k \|_2 + \frac{2.124}{\sqrt{m_{\rm CE}}} \| \e_{\rm CE} \|_2 \\
	&= 0.5 \left\| \x_s - \a^{k} \right\|_2 + \frac{2.124}{\sqrt{m_{\rm CE}}} \left\| \e_{\rm CE} \right\|_2.
	\end{align*}
	In order to obtain the bound in (\ref{eqn:errbd}) we may now simply solve the recursion for the final error after noting that
	\begin{equation*}
	(1+0.5 +0.25+ \cdots) \cdot \frac{2.124}{\sqrt{m_{\rm CE}}} \| \e_{\rm CE} \|_2 = \frac{4.248}{\sqrt{m_{\rm CE}}} \| \e_{\rm CE} \|_2.
	\end{equation*}
	If the last $k \geq \kappa$ in Algorithm \ref{alg:main}, and $\|\r^k\|_2 > \bar{\Gamma}\geq \Gamma \left( {\e_{\rm SID} } \right)$ for all $k < \kappa$, then 
	\begin{equation*}
	\|\x_s -\a\|_2 \leq 2^{-\kappa}\|\x_s \|_2+ \frac{4.248}{\sqrt{m_{\rm CE}}} \|\e_{\rm CE} \|_2.
	\end{equation*}
	On the other hand, if the last $k \geq \kappa$ in Algorithm \ref{alg:main}, $\left\| \v_{\rm CE}\right\|_2 \leq \left\| \v_{\rm CEold}\right\|_2$ for all $k < \kappa$, and $\|\r^k \|_2 \leq \bar{\Gamma}$ for some $k < \kappa$, then
	\begin{align*}
	\frac{1}{\sqrt{m_{\rm CE}}}\left\|\y_{\rm CE}- \Phi_{\rm CE}\a^{\kappa}\right\|_2 
	& \leq \frac{1}{\sqrt{m_{\rm CE}}}\left\|\y_{\rm CE}- \Phi_{\rm CE}\a^{k} \right\|_2 \\
	& \leq \left\| \frac{1}{\sqrt{m_{\rm CE}}} \Phi_{\rm CE}(\x_s -\a^{k})+ \frac{ \e_{\rm CE}}{\sqrt{m_{\rm CE}}} \right\|_2 \\
	& \leq \left\| \frac{1}{\sqrt{m_{\rm CE}}} \Phi_{\rm CE}(\x_s -\a^{k}) \right\|_2 + \frac{\| \e_{\rm CE} \|_2}{\sqrt{m_{\rm CE}}}  \\
	& \leq \sqrt{1+\delta}\left\|\x_s -\a^{k} \right\|_2 + \frac{\| \e_{\rm CE} \|_2}{\sqrt{m_{\rm CE}}}  \\
    &  \leq \bar{\Gamma} \sqrt{1+\delta} + \frac{\| \e_{\rm CE} \|_2}{\sqrt{m_{\rm CE}}},
	\end{align*}
	and
	\begin{align*}
	\frac{1}{\sqrt{m_{\rm CE}}} \left\|\y_{\rm CE}- \Phi_{\rm CE}\a^{\kappa}\right\|_2 
	& \geq \frac{1}{\sqrt{m_{\rm CE}}} \left( \left\| \Phi_{\rm CE}(\x_s -\a^{\kappa})\right\|_2 - \| \e_{\rm CE} \|_2 \right) \\
	& \geq \sqrt{1-\delta} \left\| \x_s -\a^{\kappa} \right\|_2 -  \frac{\| \e_{\rm CE} \|_2}{\sqrt{m_{\rm CE}}}. 
	\end{align*}
	By combining the upper and lower bounds, we obtain
	\begin{align*}
	\| \x_s - \a \|_2 \leq \left\| \x_s - \a^{\kappa} \right\|_2 &\leq \frac{ \sqrt{1+\delta}}{\sqrt{1-\delta}}   \bar{\Gamma}
	+\frac{2\| \e_{\rm CE} \|_2}{\sqrt{m_{\rm CE}}  \sqrt{1-\delta}} \\
	&\leq 1.0254  \bar{\Gamma} + \frac{2.0255}{\sqrt{m_{\rm CE}}} \| \e_{\rm CE} \|_2.
	\end{align*}
	
	Now assume that the first condition $\|\v_{\rm CE} \|_2^2 > \|\v_{\rm CEold}\|_2^2 $ of the halting criteria in line 16 of Algorithm \ref{alg:main} holds. There are two possible cases : (i)  $\|\r^{k-1} \|_2 \leq \bar{\Gamma} $ and (ii)  $\|\r^{k-1} \|_2 > \bar{\Gamma} $. The case (i) implies that $\| \x_s - \a \|_2 \leq \bar{\Gamma}$. For the case (ii), note first that $\|\x_s-\a^k\|_2 \leq 0.5 \|\x_s-\a^{k-1}\|_2 + \frac{2.124}{\sqrt{m_{\rm CE}}} \| \e_{\rm CE}\|_2$.
	Also, from the halting criterion,
	\begin{align*}
	\frac{1}{\sqrt{m_{\rm CE}}} \left\| \Phi_{\rm CE} \left(\x_s -\a^{k} \right) + \e_{\rm CE} \right\|_2 &\geq \frac{1}{\sqrt{m_{\rm CE}}} \left\| \Phi_{\rm CE}  \left(\x_s -\a^{k-1} \right) + \e_{\rm CE} \right\|_2, \\
	\left\| \frac{1}{\sqrt{m_{\rm CE}}}  \Phi_{\rm CE} \left(\x_s -\a^{k} \right) \right\|_2 +  \frac{\left\| \e_{\rm CE} \right\|_2}{\sqrt{m_{\rm CE}}} &\geq \left\| \frac{1}{\sqrt{m_{\rm CE}}}  \Phi_{\rm CE} \left(\x_s -\a^{k-1}\right) \right\|_2 - \frac{\| \e_{\rm CE} \|_2}{\sqrt{m_{\rm CE}}}, \\
	\sqrt{1+\delta} \left\| \x_s -\a^{k} \right\|_2 +  \frac{\| \e_{\rm CE} \|_2}{\sqrt{m_{\rm CE}}} &\geq \sqrt{1-\delta} \left\| \x_s -\a^{k-1} \right\|_2 - \frac{\| \e_{\rm CE} \|_2}{\sqrt{m_{\rm CE}}}, \\
	\left\| \x_s -\a^{k} \right\|_2  &\geq \sqrt{\frac{1-\delta}{1+\delta}} \left\| \x_s -\a^{k-1} \right\|_2-\frac{2}{\sqrt{1+\delta}} \frac{\|\e_{\rm CE} \|_2}{\sqrt{m_{\rm CE}}} .
	\end{align*}
	By combining the upper and lower bounds of $\left\| \x_s -\a^{k} \right\|_2$, we obtain
	\begin{align}
	\left\| \x_s -\a \right\|_2 = \left\| \x_s -\a^{k-1} \right\|_2 &\leq \frac{ \left(2.124+\frac{2}{\sqrt{1+\delta}}\right)\|\e_{\rm CE} \|_2}{\sqrt{m_{\rm CE}} \left(\sqrt{\frac{1-\delta}{1+\delta}} -0.5 \right)} \nonumber\\
	&\leq  8.625 \frac{\|\e_{\rm CE} \|_2}{\sqrt{m_{\rm CE}}} .
	\label{equ:BoundR}
	\end{align}
	
	 The support identification algorithm $\mathcal{A}$ is assumed to have $\mathcal{O} \left(\mathcal{L}_\mathcal{A}\right)$ runtime complexity in line 9.  A conjugate gradient least square solver can approximate line 12 with $\mathcal{O} \left(s^2 K^2 \log^4 |\mathcal{I}_{N,d}|\right)$ runtime complexity per iteration (see, e.g., Chapter 7 of \cite{doi:10.1137/1.9781611971484}, and Section 3 of \cite{hu2017rapidly}). Furthermore, a constant number of iterations (e.g. three in \cite{needell2009cosamp}) suffices.  Lines 11, 14, and 15 require the generation of an $m_{\rm CE} \times \mathcal{O}(s)$ or $m_{\rm SID} \times \mathcal{O}(s)$ submatrix of either $\Phi_{\rm CE}$ or $\Phi_{\rm SID}$, respectively.  This will take $\mathcal{O}(s  \mathcal{L}_{\Phi} m_{\rm CE}+ s \mathcal{L}_{\Phi} m_{\rm SID})$ -time.
Finally, the iteration number of the entire CoSaMP loop is bounded by $\kappa$, so that the overall runtime complexity is $\mathcal{O} \left( \left (\mathcal{L}_\mathcal{A} + s^2 K^2 \mathcal{L}_{\Phi} \log^4 |\mathcal{I}_{N,d}| + s \mathcal{L}_{\Phi} m_{\rm SID} \right) \cdot \kappa \right)$. With respect to the sampling complexity, the support identification requires $m_{\rm SID} = \mathcal{O} \left(\mathcal{L}'_\mathcal{A} \right)$ function evaluations and the conjugate gradient method requires $m_{\rm CE} = \mathcal{O} \left(s K^2 \log^4 |\mathcal{I}_{N,d}| \right)$ function evaluations \cite{doi:10.1137/1.9781611971484,hu2017rapidly}, and thus the overall sampling complexity is $m_{\rm SID} + m_{\rm CE} = \mathcal{O}\left(\mathcal{L}'_\mathcal{A} + s K^2 \log^4 |\mathcal{I}_{N,d}| \right)$.
\end{proof}

Results concerning randomized constructions of RIP matrices $\frac{1}{\sqrt{m_{\rm CE}}} \Phi_{\rm CE} \in \mathbbm{C}^{m_{\rm CE} \times |\mathcal{I}_{N,d}|}$ for BOBPs with $\delta_{2s} \leq \delta$ and $m_{\rm CE} = \mathcal{O}(s K^2 \log^4 |\mathcal{I}_{N,d}|)$ are well known (see, e.g., Theorem~\ref{thm:BOS_RIP} and Chapter 12 of \cite{foucart2013mathematical}).  Our next result gives a qualitatively similar construction of a triple $\left(\Phi_{\rm SID},\mathcal{A},\bar{\Gamma} \right)$ with what is essentially the SIP for BOPBs (see Proposition~\ref{coro:SIPrevealed} for an explicit SIP statement regarding this triple).  More specifically, Theorem \ref{thm:suppIdforsec3} constructs a support identification procedure with the properties required by Theorem \ref{thm:iterInvariant}, and also bounds its computational and sampling requirements.   We remind the reader that the error vector $\e_{\rm SID} \in \mathbbm{C}^{m_{\rm SID}}$ appearing in both Theorems \ref{thm:iterInvariant} and \ref{thm:suppIdforsec3} does not change from iteration to iteration in the analysis of Algorithm~\ref{alg:main}.\\

\begin{mytheorem}{(Sublinear-Time Support Identification for BOPBs).}
There exists
an algorithm $\mathcal{A}: \mathbbm{C}^{m_{\rm SID}} \rightarrow \mathcal{P} \left(\mathcal{I}_{N,d} \right)$ that always outputs a set of at most $2s$ index vectors $\in \mathcal{I}_{N,d}$, and a sampling strategy for randomly selecting a set of $m_{\rm SID}$ grid points $\{ \xib_\ell \}_{\ell \in [m_{\rm SID}]} \subset \D$, such that the random sampling matrix $\Phi_{\rm SID} \in \C^{m_{\rm SID} \times \left| \mathcal{I}_{N,d} \right|}$ associated with $\{ \xib_\ell \}_{\ell \in [m_{\rm SID}]}$ as per \eqref{def:A}  will have the following property with probability $\geq 0.99$:

\begin{quote}
$\mathcal{A}\left( \Phi_{\rm SID} \r^k + \e_{\rm SID} \right)$ $=$\footnote{Note that $\mathcal{A}\left( \Phi_{\rm SID} \r^k + \e_{\rm SID} \right) = \mathcal{A}\left( \Phi_{\rm SID} \left( \x_s - \a^k \right) + \e_{\rm SID} \right) = \mathcal{A}\left( \Phi_{\rm SID} \left( \ct_{\Omega^{\rm opt}_{\tilde{f},s}} - \a^k \right) + \e_{\rm SID} \right) = \mathcal{A}\left( \Phi_{\rm SID} \left( \ct - \a^k \right) +  \ep_{\rm SID} \right)$ where $\ep_{\rm SID} :=  \e_{\rm SID}  - \Phi_{\rm SID} \left(\ct - \ct_{\Omega^{\rm opt}_{\tilde{f},s}} \right)$.} $\mathcal{A}\left( \Phi_{\rm SID} \left( \ct - \a^k \right) +  \ep_{\rm SID} \right)$
will output a set $\tilde{\Omega} \subset \mathcal{I}_{N,d}$ such that
\begin{equation}
\left\| \r^k_{\tilde{\Omega}^c} \right\|_2 ~\leq ~0.2086 \left\|\r^k \right\|_2 + 2.4172 \left\| \ct-\ct_{\Omega^{\rm opt}_{\tilde{f},s}}  \right\|_2  ~\leq ~0.2203 \left\|\r^k \right\|_2
\label{equ:EngBoundSIP}
\end{equation}
holds for any $\r^k = \x_s - \a^k = \ct_{\Omega^{\rm opt}_{\tilde{f},s}}- \a^k$ satisfying $\left\| \r^k \right\|_2 > \bar{\Gamma}$, where
\begin{equation}
\bar{\Gamma}:=\left( 25 \sqrt{23s} +1 \right) \left\| \ct - \ct_{\Omega^{\rm opt}_{\tilde{f},s}} \right\|_2 + 18 \sqrt{23} \left\| \ct - \ct_{\Omega^{\rm opt}_{\tilde{f},s}}  \right\|_1 + 22 \gamma \sqrt{23s}.
\label{Def:GammaBarsec3}
\end{equation}
\end{quote}

In order to achieve this property with probability $\geq 0.99$ it suffices that 
$$m_{\rm SID} = \mathcal{O}\left(\mathcal{L}'_\mathcal{A} \right)= \mathcal{O} \left( D K^{4\tilde{d}}_\infty s^3 d^4 \cdot  \log^4 \left( \frac{DN}{d} \right) \log^2 (s) \log^2 (D) \right)$$
if the BOS constants $K_j$ are $1$ for all but at most $\tilde{d} \in \mathbbm{Z} \cap [0, D]$ BOS basis sets $\B_j$, and that
$$m_{\rm SID} = \mathcal{O}\left(\mathcal{L}'_\mathcal{A} \right)= \mathcal{O} \left( D K^{4d}_\infty s^3 d^4 \cdot  \log^4 \left( \frac{DN}{d} \right) \log^2 (s) \log^2 (D) \right)$$
if $K_0 = 1$.
In the first case the runtime complexity of $\mathcal{A}$ will be 
$$\mathcal{O}\left(\mathcal{L}_\mathcal{A} \right) = \mathcal{O}\left( \left( s^5 + s^3 N \right) D K^{4\tilde{d}}_\infty d^4 \cdot  \log^4 \left( \frac{DN}{d} \right) \log^2 (s) \log^2 (D) \right),$$ 
and in the second case its runtime complexity will be
$$\mathcal{O}\left(\mathcal{L}_\mathcal{A}\right) = \mathcal{O} \left(  \left( s^5 + s^3 N \right) D K^{4d}_\infty d^4 \cdot  \log^4 \left( \frac{DN}{d} \right) \log^2 (s) \log^2 (D) \right)$$
when $K_0 = 1$.
	\label{thm:suppIdforsec3}
\end{mytheorem}

\begin{proof}
See Section~\ref{sec:SupportID}.  This is a slight restatement of Theorem~\ref{thm:suppId}.  The algorithm $\mathcal{A}$ is Algorithm~\ref{alg:suppid:impl}, and the $m_{\rm SID}$ points $\{ \xib_\ell \}_{\ell \in [m_{\rm SID}]} \subset \D$ should be randomly selected as per the first paragraph of Theorem~\ref{thm:SuppIDWorks}.  The runtime and sampling complexities then also follow from Theorem~\ref{thm:SuppIDWorks}.  

The first inequality in \eqref{equ:EngBoundSIP} follows directly from Theorem~\ref{thm:suppId}.  In order to see that the second inequality 
$$0.2086 \left\| \r^k \right\|_2 + 2.4172 \left\| \ct-\ct_{\Omega^{\rm opt}_{\tilde{f},s}}  \right\|_2 ~\leq ~0.2203 \left \|\r^k \right\|_2$$
	holds whenever $\left\| \r^k \right\|_2 > \bar{\Gamma}$, we note that 
	\begin{align*}
	\left\| \r^k \right\|_2 >&~\bar{\Gamma} = \left( 25 \sqrt{23s} +1 \right) \left\| \ct - \ct_{\Omega^{\rm opt}_{\tilde{f},s}} \right\|_2 + 18 \sqrt{23} \left\| \ct - \ct_{\Omega^{\rm opt}_{\tilde{f},s}}  \right\|_1 + 22 \gamma \sqrt{23s}\\
	&\geq  \left( 25 \sqrt{23s} +1 \right) \left\| \ct - \ct_{\Omega^{\rm opt}_{\tilde{f},s}} \right\|_2 + 18 \sqrt{23} \left\| \ct - \ct_{\Omega^{\rm opt}_{\tilde{f},s}}  \right\|_2\\
	&\geq \left( 43 \sqrt{23} +1 \right) \left\| \ct - \ct_{\Omega^{\rm opt}_{\tilde{f},s}}  \right\|_2.
	\end{align*}
	Thus, one can see that 
	$$0.2086 \left\|\r^k \right\|_2 + 2.4172 \left\| \ct-\ct_{\Omega^{\rm opt}_{\tilde{f},s}}  \right\|_2 < \left( 0.2086 + \frac{2.4172}{43 \sqrt{23} +1} \right) \left\| \r^k \right\|_2$$
	which yields the desired effective SIP constant $\beta = 0.2203 \in(0,0.2228]$.
\end{proof}

The following proposition is a variant of Theorem~\ref{thm:suppIdforsec3} that more formally establishes exactly the type of SIP triple 
$$\left(\Phi_{\rm SID} \in \mathbbm{C}^{m_{\rm SID} \times |\mathcal{I}_{N,d}|},\mathcal{A}: \mathbbm{C}^{m_{\rm SID}} \rightarrow \mathcal{P}([|\mathcal{I}_{N,d}|]),\Gamma: \mathbbm{C}^{m_{\rm SID}} \rightarrow [0,\infty) \right)$$ defined in Section~\ref{sec:Introduction}.  Its main contribution is to explicitly define a function $\Gamma: \mathbbm{C}^{m_{\rm SID}} \rightarrow [0,\infty)$ as per Definition~\ref{def:SIP} which produces a SIP triple when combined with the algorithm $\mathcal{A}$ and matrix $\Phi_{\rm SID}$ from Theorem~\ref{thm:suppIdforsec3}.  We hasten to point out, however, that only a valid upper bound $\bar{\Gamma}$ of $\Gamma(\e_{\rm SID})$ as per \eqref{Def:GammaBarsec3} is actually necessary in order to apply Theorem~\ref{thm:iterInvariant} -- one doesn't actually have to know the exact form of the best achievable function $\Gamma$.  Nonetheless, for completeness we provide a function $\Gamma$ in the next proposition which allows us to formally satisfy Definition~\ref{def:SIP} as stated in Section~\ref{sec:Introduction}.\\

\begin{prop} \sloppypar
        Let $\mathcal{A}: \mathbbm{C}^{m_{\rm SID}} \rightarrow \mathcal{P} \left(\mathcal{I}_{N,d} \right)$ and $\Phi_{\rm SID} \in \C^{m_{\rm SID} \times \left| \mathcal{I}_{N,d} \right|}$ be the algorithm and random sampling matrix referred to by Theorem \ref{thm:suppIdforsec3}, where $\displaystyle \bigcup_{j \in [2D-1]} \left \{ \varrho_{\S_j}\left( \w^j_\ell,\z_k^j \right) \right \}_{\ell \in [m_1], k \in [m_2]}$ denotes the $m_{\rm SID} = m_1 m_2 (2D-1)$ random sampling points\footnote{See the input of Algorithm~\ref{alg:suppid:impl} for a description of the sampling points and note that the $2D-1$ blocks have been reindexed for ease of discussion, and that the index sets $\S_j$ must therefore correspond to either $\{ j \}$ or $[j+1]$ accordingly.  For a description of how to generate the component points $\w^j_\ell,\z_k^j$ we refer the reader to Theorem~\ref{thm:SuppIDWorks}.} used to create $\Phi_{\rm SID}$.  In addition, define
$$\Gamma(\e_{\rm SID}) := C \frac{\sqrt{s}}{\sqrt{m_1 m_2}} \max_{j \in [2D-1]} \left\| \e_{\rm SID}^j \right\|_2, $$
	where $C \in \mathbbm{R}^+$ is an absolute constant $\leq 29$ fixed below\footnote{See \eqref{equ:threshold} in Theorem \ref{thm:pairingFunc} for a definition of $\Gamma$ with explicit constants, where we further point out that $\alpha$ is fixed to be $\sqrt{23}$ in Theorem~\ref{thm:SuppIDWorks}.  When looking at Theorem \ref{thm:pairingFunc} one should keep in mind that the matrix $\mathcal{E}^h_{\S} \in \C^{m_1 \times m_2}$ therein is nothing other than a matricized version of $\e_{\rm SID}^j$ with $\S = \S_j$ for any desired choice of $j \in [2D - 1]$.}, and where $\e_{\rm SID}^j \in \mathbbm{C}^{m_1 m_2}$ corresponds to the portion of $\e_{\rm SID} \in \mathbbm{C}^{m_{\rm SID}}$ formed by evaluating $e$ in \eqref{equ:fEquivalnce} at the evaluation points $\left \{ \varrho_{\S_j}\left( \w^j_\ell,\z_k^j \right) \right \}_{\ell \in [m_1], k \in [m_2]}$ for each $j \in [2D-1]$.  Then, with probability $\geq 0.99$ 
the triple $\left(\Phi_{\rm SID}, \mathcal{A}, \Gamma \right)$ formed using the random evaluation points $\displaystyle \bigcup_{j \in [2D-1]} \left \{ \varrho_{\S_j}\left( \w^j_\ell,\z_k^j \right) \right \}_{\ell \in [m_1], k \in [m_2]}$ will have both of the following properties:
	\begin{enumerate}
	\item[(i)] $\left(\Phi_{\rm SID}, \mathcal{A}, \Gamma \right)$ will have the SIP of order $(2s,\beta = 0.2203)$, and
	\item[(ii)] $\bar{\Gamma} := \bar{\Gamma} \left(\ct - \ct_{\Omega^{\rm opt}_{\tilde{f},s}}, \gamma \right)$ in \eqref{Def:GammaBarsec3} will satisfy $\bar{\Gamma} \geq \Gamma(\e_{\rm SID})$ for all inputs $ \Phi_{\rm SID} \r^k + \e_{\rm SID} = \Phi_{\rm SID} \left( \ct - \a^k \right) +  \ep_{\rm SID}$ with $\| \ep_{\rm SID} \|_\infty \leq \gamma$.
	\end{enumerate}
\label{coro:SIPrevealed}	
\end{prop}

\begin{proof}
The fact that $\bar{\Gamma}$ in \eqref{Def:GammaBarsec3} satisfies $\bar{\Gamma} \geq \Gamma(\e_{\rm SID})$ is ultimately a consequence of Lemmas~\ref{lem:DetermineGamma} and~\ref{lem:errorBound}.  The SIP holding for $(\Phi_{\rm SID}, \mathcal{A}, \Gamma)$ follows from the fact that Theorem~\ref{thm:suppIdforsec3} still holds if the condition $\left\| \r^k \right\|_2 > \bar{\Gamma}$ is replaced by the condition $\left\| \r^k \right\|_2 > \Gamma(\e_{\rm SID})$.  This can be seen by tracing through Theorem~\ref{thm:suppIdforsec3}'s proof beginning with the proof of Theorem~\ref{thm:ExistFastFuncs} where one need not apply Lemma~\ref{lem:errorBound}, and from which an alternate version of Theorem~\ref{thm:GenAlgSuppIDWorks} with \eqref{eqn:SeiveFtnWorkCondition} involving $\Gamma(\e_{\rm SID})$ instead of its current right-hand side trivially follows.  With such an alternate form of Theorem~\ref{thm:GenAlgSuppIDWorks} in hand one can then immediately recover a similar variant of Theorem~\ref{thm:SuppIDWorks} involving $\Gamma(\e_{\rm SID})$ which, in turn, can then provide an alternate (though less easily stated and interpretable) version of Theorem~\ref{thm:suppIdforsec3} involving the condition $\left\| \r^k \right\|_2 > \Gamma(\e_{\rm SID})$.
\end{proof}

Finally, in Theorem~\ref{thm:NewCoSaMP}, it is shown that Algorithm ~\ref{alg:suppid:impl} can be utilized as the support identification algorithm $\mathcal{A}$ in a SIP triple $\left(\Phi_{\rm SID},\mathcal{A},\Gamma: \mathbbm{C}^{m_{\rm SID}} \rightarrow [0,\infty) \right)$ for use in Theorem \ref{thm:iterInvariant}.  The sublinear runtime and sampling complexities of Algorithm~\ref{alg:suppid:impl} listed in Theorem~\ref{thm:suppIdforsec3} then result in a new sublinear-time and memory efficient compressive sensing approach for BOPB-compressible functions $f: \mathcal{D} \rightarrow \mathbbm{C}$.
We would like to remind the reader before stating this main result that $\x_s = \ct_{\Omega^{\rm opt}_{\tilde{f},s}}$ is $s$-sparse with $\y_{\rm SID}=\Phi_{\rm SID}\x_s+\e_{\rm SID}$ and $\y_{\rm CE}=\Phi_{\rm CE}\x_s+\e_{\rm CE}$.  Furthermore, the triple $(\Phi_{\rm SID}, \mathcal{A}, \Gamma)$ constructed from the support identification procedure in Algorithm~\ref{alg:suppid:impl} with $\alpha:=\sqrt{23}$ satisfies the SIP of order $(2s, \beta)$ with $\beta \in [0.2203, 0.2228]$ with high probability (see Theorem~\ref{thm:suppIdforsec3} and Proposition~\ref{coro:SIPrevealed}), and the matrix $\frac{1}{\sqrt{m_{\rm CE}}} \Phi_{\rm CE}$ has a RIP constant $\delta_{2s} \leq \delta$ for $\delta \in (0,0.025]$ with high probability (see Theorem~\ref{thm:BOS_RIP}).  Finally, $\mathcal{A}$ always outputs a set of cardinality at most $2s$ as noted in Algorithm~\ref{alg:main}.\\

\begin{mytheorem}{(Sublinear-Time Compressive Sensing for BOPB-compressible Signals).}
	Let $N, d \in \N \setminus \{1\}$, $s < |\mathcal{I}_{N,d}|/2$,  $\delta \in (0,0.025]$, $\eta \in (0, \infty)$, $\kappa = \left \lceil \log_2 \left( \left\| \ct_{\Omega^{\rm opt}_{\tilde{f},s}} \right \|_2 / \eta\right) \right \rceil$, $K$ the BOS constant of \eqref{defBNd}, $\y_{\rm SID}=\Phi_{\rm SID}\ct_{\Omega^{\rm opt}_{\tilde{f},s}}+\e_{\rm SID}$, and $\y_{\rm CE}=\Phi_{\rm CE}\ct_{\Omega^{\rm opt}_{\tilde{f},s}}+\e_{\rm CE}$ where we assume that both $\ep_{\rm SID} :=  \e_{\rm SID}  - \Phi_{\rm SID} \left(\ct - \ct_{\Omega^{\rm opt}_{\tilde{f},s}} \right)$ and $\ep_{\rm CE} := \e_{\rm CE}  - \Phi_{\rm CE} \left(\ct - \ct_{\Omega^{\rm opt}_{\tilde{f},s}} \right)$ have $\| \ep_{\rm SID} \|_\infty \leq \gamma$ and $\| \ep_{\rm CE} \|_\infty \leq \gamma$, respectively.
	Suppose further that the triple $(\Phi_{\rm SID}, \mathcal{A}, \Gamma)$ with $\bar{\Gamma} \in \mathbbm{R}^+$ such that $\bar{\Gamma} \geq \Gamma(\e_{\rm SID})$ satisfies the SIP of order $(2s, \beta)$ with $\beta \in [0.2203, 0.2228]$ as per Theorem~\ref{thm:suppIdforsec3}, and that $\frac{1}{\sqrt{m_{\rm CE}}} \Phi_{\rm CE} \in \mathbbm{C}^{m_{\rm CE} \times |\mathcal{I}_{N,d}|}$ has a RIP constant $\delta_{2s}\leq \delta$ and $m_{\rm CE} = \mathcal{O}(s K^2 \log^4 |\mathcal{I}_{N,d}|) = \mathcal{O}\left(s K^2 \, d^4 \cdot \log^4 \left( \frac{DN}{d} \right) \right)$ (see \eqref{equ:IdNCardBound}).  Then, 
	for each $k \geq 0$ the signal approximation $\a^k$ in Algorithm~\ref{alg:main} is $s$-sparse and satisfies
	\begin{equation}		
	\left\| \ct_{\Omega^{\rm opt}_{\tilde{f},s}} - \a^{k+1}\right\|_2 \leq 0.5 \left\|	\ct_{\Omega^{\rm opt}_{\tilde{f},s}} - \a^k \right\|_2 + \frac{2.124}{\sqrt{m_{\rm CE}}} \| \e_{\rm CE} \|_2,
	\label{equ:sparseIterInvariant}
	\end{equation}
	as long as  $$\left\| \ct_{\Omega^{\rm opt}_{\tilde{f},s}} - \a^{k} \right\|_2 > \bar{\Gamma}:=\left( 25 \sqrt{23s} +1 \right) \left\| \ct - \ct_{\Omega^{\rm opt}_{\tilde{f},s}} \right\|_2 + 18 \sqrt{23} \left\| \ct - \ct_{\Omega^{\rm opt}_{\tilde{f},s}}  \right\|_1 + 22 \gamma \sqrt{23s}.$$
	As a consequence, 
Algorithm~\ref{alg:main} produces an $s$-sparse approximation $\a$ that satisfies
	\begin{align}
	\left\| \ct_{\Omega^{\rm opt}_{\tilde{f},s}} - \a \right\|_2 &\leq \max  \left\{   1.03   \bar{\Gamma} + 2.03 \frac{\| \e_{\rm CE} \|_2}{{\sqrt{m_{\rm CE}}}} , ~2^{-\kappa} \left\|\ct_{\Omega^{\rm opt}_{\tilde{f},s}} \right\|_2 + 5 \frac{\|\e_{\rm CE} \|_2}{\sqrt{m_{\rm CE}}}, ~9 \frac{\|\e_{\rm CE} \|_2}{\sqrt{m_{\rm CE}}} \right\} \label{equ:mainErrorGuarantee} \\
	&\leq C \left( \sqrt{s} \left\| \ct - \ct_{\Omega^{\rm opt}_{\tilde{f},s}} \right\|_2 + \left\| \ct - \ct_{\Omega^{\rm opt}_{\tilde{f},s}}  \right\|_1 + \gamma \sqrt{s} \right) + \eta,
	 \label{equ:mainErrorGuaranteeFINAL}
	\end{align}	
	where $C \in \mathbbm{R}^+$ is an absolute universal constant.
	
	In order to achieve \eqref{equ:mainErrorGuaranteeFINAL} for all such possible inputs $\y_{\rm SID}$ and $\y_{\rm CE}$ with probability $\geq 0.99$ it suffices that 
	$$m:=m_{\rm SID}+m_{\rm CE} = \mathcal{O} \left( D K^{4\tilde{d}}_\infty s^3 d^4 \cdot  \log^4 \left( \frac{DN}{d} \right) \log^2 (s) \log^2 (D)  \right)$$
	if the BOS constants $K_j$ are $1$ for all but at most $\tilde{d} \in \mathbbm{Z} \cap [0, D]$ BOS basis sets $\B_j$ (BOPB of type I), and that
	$$m = \mathcal{O} \left( D K^{4d}_\infty s^3 d^4 \cdot  \log^4 \left( \frac{DN}{d} \right) \log^2 (s) \log^2 (D) \right)$$
	if $K_0 = 1$ (BOPB of type II). 
	
	In the BOPB of type I, the runtime complexity of the entire algorithm will be 
	$$\mathcal{O}\left( \left( s^5 + s^3 N \right) D^2 K^{4\tilde{d}}_\infty d^4 \cdot  \log^4 \left( \frac{DN}{d} \right) \log^2 (s) \log^2 (D) \log \left( \left\| \ct_{\Omega^{\rm opt}_{\tilde{f},s}} \right \|_2 / \eta\right) \right),$$ 
	and in the BOPB of type II, the runtime complexity will be
	$$\mathcal{O} \left( \left( s^5 + s^3 N \right) D^2 K^{4d}_\infty d^4 \cdot  \log^4 \left( \frac{DN}{d} \right) \log^2 (s) \log^2 (D) \log \left( \left\| \ct_{\Omega^{\rm opt}_{\tilde{f},s}} \right \|_2 / \eta\right) \right).$$
	  Here we have assumed that the runtime complexity of computing any desired matrix entry $\left( \Phi_{\rm CE} \right)_{j,\ell}$, or $\left( \Phi_{\rm SID} \right)_{j,\ell}$, for any valid choice of $j,\ell$ is $\mathcal{O}\left(\mathcal{L}_{\Phi} \right) = \mathcal{O}(N D)$-time.
	\label{thm:NewCoSaMP}
\end{mytheorem}

\begin{proof}
The result follows by combining Theorems \ref{thm:iterInvariant} and \ref{thm:suppIdforsec3} which immediately yields \eqref{equ:mainErrorGuarantee}, as well as the stated runtime and sampling complexities.  Note that Theorem~\ref{thm:suppIdforsec3} assumes that we are sampling from a BOPB-sparse function with arbitrary additive noise $e': \D \rightarrow \mathbbm{C}$ that has $\| e' \|_\infty \leq \gamma$, which leads to the restriction on $\| \ep_{\rm SID} \|_\infty$ and $\| \ep_{\rm CE} \|_\infty$.  To obtain \eqref{equ:mainErrorGuaranteeFINAL} one can simply substitute our choice of $\kappa$ into \eqref{equ:mainErrorGuarantee} and use 
Lemma~\ref{lem:OpBoundRSM} to see that
\begin{align*}
\frac{\|\e_{\rm CE} \|_2}{\sqrt{m_{\rm CE}}} &=~ \left\| \frac{\ep_{\rm CE}}{\sqrt{m_{\rm CE}}} + \frac{1}{\sqrt{m_{\rm CE}}}\Phi_{\rm CE} \left(\ct - \ct_{\Omega^{\rm opt}_{\tilde{f},s}} \right) \right \|_2\\
&\leq~ \gamma ~+~ \sqrt{1 + \delta} \left( \frac{ \left\| \ct - \ct_{\Omega^{\rm opt}_{\tilde{f},s}}  \right\|_1}{\sqrt{s}} + \left\| \ct - \ct_{\Omega^{\rm opt}_{\tilde{f},s}}  \right\|_2 \right). 
\end{align*}
Finally, we note that the runtime and sampling complexity bounds have been simplified by collecting and removing dominated terms along with the fact that $K\leq K_{\infty}^{\tilde{d}}$ (BOPB of type I) or $K\leq K_{\infty}^d$ (BOPB of type II) as discussed in Section \ref{sec:Setup}.
\end{proof}

With Theorem~\ref{thm:NewCoSaMP} in hand we may now prove our main result concerning function approximation in a Hilbert space $L^2(\mathcal{D},\mu)$ spanned by a countable orthonormal product basis $\left\{ T_{\n} ~|~ \n \in \mathbbm{N}^D \right\} \supset \B$.\\

\begin{corollary}{(Main Result).}
Let $\eta \in (0, \infty)$ and $s,d,N \in \mathbbm{N} \setminus \{1 \}$ with $d \leq D$ and $s < |\mathcal{I}_{N,d}|/2$.  There exists a finite set of grid points $\mathcal{G} \subset \D$, an algorithm $\mathcal{H}: \mathbbm{C}^{\left| \mathcal{G} \right|} \rightarrow \left( \mathcal{I}_{N,d} \times \mathbbm{C} \right)^s$, and an absolute universal constant $C' \in \mathbbm{R}^+$ such that the function $a: \D \rightarrow \mathbbm{C}$ defined by $a(\xib) := \sum_{(\n,a_{\n}) \in \mathcal{H}(f(\mathcal{G}))} a_\n T_\n(\xib)$ satisfies
$$\| f - a \|_{L^2(\mathcal{D},\mu)} ~\leq~ \left\| f - \tilde{f} \right\|_{L^2(\mathcal{D},\mu)} + C' \left( \sqrt{s} \left\| \ct - \ct_{\Omega^{\rm opt}_{\tilde{f},s}} \right\|_2 + \left\| \ct - \ct_{\Omega^{\rm opt}_{\tilde{f},s}}  \right\|_1 + \gamma \sqrt{s} \right) + \eta$$
for all $f = \sum_{\n \in \mathbbm{N}^D} c_{\n} T_{\n} \in L^2(\mathcal{D},\mu)$ with $\gamma := \| f - \tilde{f} \|_\infty$ $=~ \sup_{\xib \in \D} \left|  \left(f - \tilde{f} \right) (\xib) \right| ~<~ \infty$, where $\tilde{f}: \D \rightarrow \mathbbm{C}$ is the finite dimensional approximation to $f$ defined as per \eqref{equ:FinApproxf}.

	If the BOS constants $K_j$ are $1$ for all but at most $\tilde{d} \in \mathbbm{Z} \cap [0, D]$ BOS basis sets $\B_j$ then
	$$\left| \mathcal{G} \right| = \mathcal{O} \left( D K^{4\tilde{d}}_\infty s^3 d^4 \cdot  \log^4 \left( \frac{DN}{d} \right) \log^2 (s) \log^2 (D)  \right),$$
	and the algorithm $\mathcal{H}$ will have runtime complexity 
	$$\mathcal{O}\left( \left( s^5 + s^3 N \right) D^2 K^{4\tilde{d}}_\infty d^4 \cdot  \log^4 \left( \frac{DN}{d} \right) \log^2 (s) \log^2 (D) \log \left( \left\| \ct_{\Omega^{\rm opt}_{\tilde{f},s}} \right \|_2 / \eta\right) \right).$$ 	
	If $K_0 = 1$ then 
	$$\left| \mathcal{G} \right| = \mathcal{O} \left( D K^{4d}_\infty s^3 d^4 \cdot  \log^4 \left( \frac{DN}{d} \right) \log^2 (s) \log^2 (D) \right),$$
	and the algorithm $\mathcal{H}$ will have runtime complexity 
	$$\mathcal{O} \left( \left( s^5 + s^3 N \right) D^2 K^{4d}_\infty d^4 \cdot  \log^4 \left( \frac{DN}{d} \right) \log^2 (s) \log^2 (D) \log \left( \left\| \ct_{\Omega^{\rm opt}_{\tilde{f},s}} \right \|_2 / \eta\right) \right).$$
	Here we have assumed that any desired basis function $T_{\n} \in \B$ can be evaluated at any desired point in $\D$ in $\mathcal{O}(N D)$-time.
\label{cor:MainRes}
\end{corollary}

\begin{proof}
This follows from Theorem~\ref{thm:NewCoSaMP}.  The algorithm $\mathcal{H}: \mathbbm{C}^{\left| \mathcal{G} \right|} \rightarrow \left( \mathcal{I}_{N,d} \times \mathbbm{C} \right)^s$ is Algorithm~\ref{alg:main} using Algorithm~\ref{alg:suppid:impl} for line 9.  The set of grid points $\mathcal{G} \subset \D$ is the union of the evaluation points used to create the random sampling matrices $\Phi_{\rm SID}$ and $\Phi_{\rm CE}$ from Theorem~\ref{thm:NewCoSaMP} so that $f(\mathcal{G}) = (\y_{\rm SID},\y_{\rm CE}) \in \mathbbm{C}^{m_{\rm SID}+m_{\rm CE}}$.  And, the error bound follows from \eqref{equ:mainErrorGuaranteeFINAL} and the triangle inequality since
\begin{align*}
\| f - a \|_{L^2(\mathcal{D},\mu)} ~&\leq~ \left\| f - \tilde{f} \right\|_{L^2(\mathcal{D},\mu)} + \left\| \tilde{f} - a \right\|_{L^2(\mathcal{D},\mu)}\\
&=~ \left\| f - \tilde{f} \right\|_{L^2(\mathcal{D},\mu)} + \left\| \ct - \a \right\|_2\\
&\leq~ \left\| f - \tilde{f} \right\|_{L^2(\mathcal{D},\mu)} + \left\| \ct - \ct_{\Omega^{\rm opt}_{\tilde{f},s}}  \right\|_2 + \left\| \ct_{\Omega^{\rm opt}_{\tilde{f},s}} - \a \right\|_2\\
&\leq~\| f - \tilde{f} \|_{L^2(\mathcal{D},\mu)} + (C+1) \left( \sqrt{s} \left\| \ct - \ct_{\Omega^{\rm opt}_{\tilde{f},s}} \right\|_2 + \left\| \ct - \ct_{\Omega^{\rm opt}_{\tilde{f},s}}  \right\|_1 + \gamma \sqrt{s} \right) + \eta
\end{align*}
where the absolute constant $C$ is from Theorem~\ref{thm:NewCoSaMP}.
\end{proof}

Next, in Section \ref{sec:SupportID}, we will focus on developing Algorithm \ref{alg:suppid:impl} and demonstrating that it performs as desired.  We hasten to note before beginning, however, that the development of another support identification method satisfying the SIP with lower runtime or sampling complexity could be used to create a new and potentially superior version of Theorem \ref{thm:NewCoSaMP} in the future.  We leave the development of such improved methods in the hands of the sufficiently interested and clever reader.


\section{Sublinear-Time Support Identification}
\label{sec:SupportID}

We assume herein 
that the function 
$h:  \D \rightarrow \C$ of $D$ variables,
\begin{equation}
h := \tilde{h} + e',
\label{def:ffind}
\end{equation}
where $\tilde{h}: \D \rightarrow \C$ is as per \eqref{equ:FinApproxf} with coefficient vector $\rt \in \C^{\mathcal{I}_{N,d}}$ in the BOS $\B$ as per~\eqref{def:BOS_B}, 
\begin{equation}
\tilde{h}(\xib) : = \sum_{\n \in \mathcal{I}_{N,d}} \tilde{r}_{\n} T_{\n}(\xib),
\label{def:htilde}
\end{equation}
and where $e': \D \rightarrow \C$ is bounded so that $\sup_{\xib \in \D} |e'(\xib)| \leq \gamma$. In terms of our problem setting about $f$, the function $\tilde{h}$ is each residual function $\tilde{f}-a$ where $a$ is the function constructed from the approximation $\a^{k}$ that Algorithm \ref{alg:main} produces in each iteration.
In order to escape exponential sampling dependence on the dimension $D$ we will further assume below the BOPB of type I or II (see Section~\ref{sec:Setup}).  In addition, motivated by Section~\ref{sec:CoSaMPguarantee}, we will be most interested in the case where $\left \| h - \tilde{h}^{\rm opt}_{2s} \right \|_{L^2(\D,\mu)} \lesssim \left\|\tilde{h}^{\rm opt}_{2s} \right \|_{L^2(\D,\mu)} $.  In particular, we will almost exclusively represent $h$ as $h = \tilde{h}^{\rm opt}_{2s} + \left(\tilde{h} - \tilde{h}^{\rm opt}_{2s} + e' \right)$ below where we hope that $e_h := \tilde{h} - \tilde{h}^{\rm opt}_{2s} + e'$ has a relatively small $L^2$-norm compared to that of $\tilde{h}^{\rm opt}_{2s}$.

In order to approximate $h$ we seek to find a near-optimal set of basis functions from $\B$ on which to approximately project $h$.  In particular, we would be quite pleased to identify all of $\Omega^{\rm opt}_{\tilde{h},2s}$ -- that is, all the basis functions which compose $\tilde{h}^{\rm opt}_{2s}$ -- if possible given that $h \approx \tilde{h}^{\rm opt}_{2s}$.  This appears a bit too ambitious goal in general, however.  Instead, we will focus on the easier goal of identifying all the entries of $\Omega^{\rm opt}_{\tilde{h},2s}$ which individually contribute a nontrivial amount of energy to the total $L^2$-norm of $\tilde{h}^{\rm opt}_{2s}$.  We will represent (portions of) these basis element indices via the following sets of (partial) energetic indices. 

Let $\S \subseteq [D]$, $s' \in \N$, $\alpha \in (1, \infty)$ be a fixed constant to be determined later.  We define the set of {\it energetic partial index vectors of $\tilde{h}^{\rm opt}_{s'}$ in $N^{\S}$} to be  
\begin{equation}
\Omega^{\alpha,s'}_{\S} := \left\{ \n_\S  ~\bigg|~ \n \in \mathcal{I}_{N,d} ~\&~ \left\| \left(\rt_{\Omega^{\rm opt}_{\tilde{h},s'}}\right)_{\S;\n} \right\|_2 \geq \frac{ \left\| \rt_{\Omega^{\rm opt}_{\tilde{h},s'}} \right\|_2}{\alpha \sqrt{s'}} \right\} \subseteq N^{\S},
\label{def:HeavySet}
\end{equation}
where $N^{\S} := \left\{ \n_\S ~\bigg|~ \n \in \mathcal{I}_{N,d} \right\} \subseteq \mathcal{I}_{N,d} \subseteq [N]^D$.  Note in particular that $N^{[D]} = \mathcal{I}_{N,d}$ so that $\Omega^{\alpha,s'}_{[D]}$ contains all $\n \in \Omega^{\rm opt}_{\tilde{h},s'}$ whose associated entry has $|\tilde{r}_{\n}| \geq \frac{ \left\| \rt_{\Omega^{\rm opt}_{\tilde{h},s'}} \right\|_2}{\alpha \sqrt{s'}}$.  Furthermore, it is also important to note that $\Omega^{\alpha,s'}_{[D]} \subseteq \Omega^{\rm opt}_{\tilde{h},s'}$ holds for all $s' \in \left[ \hspace{1pt} \left|  \mathcal{I}_{N,d} \right| \hspace{1pt} \right] \setminus \{ 0 \}$.  More generally, $\Omega^{\alpha,s'}_{\S} \subset \Omega^{\rm opt}_{s',\S} := \left\{ \q_{\S}  ~\big| ~ \q \in \Omega^{\rm opt}_{\tilde{h},s'} \right\}$ holds for all $\S \subseteq [D]$ and $s' \in \left[ \hspace{1pt} \left|  \mathcal{I}_{N,d} \right| \hspace{1pt} \right] \setminus \{ 0 \}$.

Our next lemma shows that identifying a superset of $\Omega^{\alpha,2s}_{[D]}$ is enough to ensure that we will find a set of basis elements that can approximate $\tilde{h}^{\rm opt}_{2s}$ (and therefore $h$) well.  In particular, we will find the support of the majority of the energy of $\tilde{h}^{\rm opt}_{2s}$, $\left \| \tilde{h}^{\rm opt}_{2s} \right \|_{L^2(\D,\mu)}  = \left\| \rt_{\Omega^{\rm opt}_{\tilde{h},2s}} \right\|_2$.  With respect to Section~\ref{sec:CoSaMPguarantee}, the next lemma shows that any support set we discover which contains $\Omega^{\alpha,2s}_{[D]}$ will be sufficiently informative to guarantee that CoSaMP will make progress during its current iteration.\\

\begin{lemma}
Let $\alpha \geq \sqrt{23}$. If $\Omega^{\alpha,2s}_{[D]} \subseteq \widetilde{\Omega} \subseteq \mathcal{I}_{N,d}$ then
$$ \left\| \rt_{\Omega^{\rm opt}_{\tilde{h},2s} \cap \tilde{\Omega}^c} \right\|_2 \leq 0.2086 \left\|\rt_{\Omega^{\rm opt}_{\tilde{h},2s}} \right\|_2.$$
\label{lem:BigCoefsSuffice}
\end{lemma}

\begin{proof}
Setting $\rp:=\rt_{\Omega^{\rm opt}_{\tilde{h},2s}}$, one can see that
\begin{equation*}
\left\|\rp_{\widetilde{\Omega}^c} \right\|_2^2 ~=~ \sum_{\n \in \Omega^{\rm opt}_{\tilde{h},2s}\cap \widetilde{\Omega}^c} |r'_{\n} |^2 ~<~ 2s\cdot \frac{\|\rp\|_2^2}{\alpha^2 2s} ~\leq \frac{\|\rp\|_2^2}{23}
\end{equation*}
since $\widetilde{\Omega}^c \cap \Omega^{\alpha,2s}_{[D]} = \emptyset$.
\end{proof}
In light of Lemma~\ref{lem:BigCoefsSuffice} above we will now turn our attention to identifying $\Omega^{\alpha,2s}_{[D]}$ in a computationally and sample efficient fashion.  In particular, we seek to identify $\Omega^{\alpha,2s}_{[D]}$ as quickly as possible while simultaneously using as few fixed and nonadaptive samples from $h = \tilde{h}^{\rm opt}_{2s} + e_h$ as possible.  This is accomplished via Algorithm~\ref{alg:suppid:impl} below.  Theorem~\ref{thm:SuppIDWorks} then proves that it works as intended.\\

\begin{algorithm}[!h]
	\caption{Implemented Support Identification (Special Case of Algorithm~\ref{alg:suppid})}
	\label{alg:suppid:impl}
	\begin{algorithmic}[1]
		\Procedure{$\mathbf{SupportID}$}{}\\
		{\textbf{Parameters: }}{$N\in\N$, $D\in\N$, $\alpha\geq\sqrt{23}$, sparsity $s\in\N$.}\\
		{\textbf{Input: }}{${\v}_{\text{SID}} \in \C^{ m_1 m_2 (2D-1)}$ split into $2D-1$ blocks. The first $D$ blocks ${\v}_{\text{SID},j}:=\big\{ h (\varrho_{\{j\}}(\w^j_\ell,\z^j_k) ) \big\}_{\ell \in [m_1],k \in [m_2]}$ , $j\in [D]$, belong to entry identification where $\w^j_\ell\in\D_{\{j\}}$, $\z^j_k\in\D_{[D]\setminus\{j\}}$, and  $\varrho_{\{j\}}$ as per~\eqref{def:rho}.
		The last $D-1$ blocks ${\v}_{\text{SID},D-1+j}:=\big\{ h (\varrho_{[j+1]}(\w^{D-1+j}_\ell,\z^{D-1+j}_k) ) \big\}_{\ell \in [m_1],k \in [m_2]}$ , $j\in [D]\setminus \{ 0 \}$, belong to the pairing where $\w^{D-1+j}_\ell\in\D_{[j+1]}$ and $\z^{D-1+j}_k\in\D_{[D]\setminus [j+1]}$. 
		 }\\
		{\textbf{Output: }}{A set $\tilde{\Omega} \supset \Omega^{\alpha,2s}_{[D]}$ with $\left| \tilde{\Omega} \right| \leq 2s$.}
		\Statex \,
 		\For{$j = 0$ {\bf up to} $D-1$}{}
		\label{linrefEstimator1b}
		  \State $\displaystyle \mathbf{E}^{\text{EI}}_{j,n} \leftarrow \frac{1}{m_2} \sum_{k\in [m_2]} \Bigg| \frac{1}{m_1} \sum_{\ell \in [m_1]} \left({\v}_{\text{SID},j}\right)_{\ell,k} ~\overline{T_{j;n}\left(\w^j_\ell\right)}  \Bigg|^2$ for each $n\in [N]$, see also~\eqref{equ:EstimatorGen}, with $\left({\v}_{\text{SID},j}\right)_{\ell,k} =  h (\varrho_{\{j\}}(\w^j_\ell,\z^j_k) )$.
		  \State $\mathcal{N}_j \leftarrow \big\{ n\in [N] ~\bigg|~ \min(2s,N) \text{-largest values } \mathbf{E}^{\text{EI}}_{j,n} \big\}$.
		\EndFor
		\label{linrefEstimator1e}
		\State $\T_D \leftarrow \mathcal{N}_0$.
		\For{$j = 1$ {\bf up to} $D-1$}{}
		\State $\T'_{D+j} \leftarrow  \left\{ \n + \m ~\big|~ \n \in \T_{D+j-1},~\m \in \mathcal{N}_{j} \right\} \cap \mathcal{I}_{N,d} \subseteq N^{[j+1]}$.
		\State $\displaystyle \mathbf{E}^{\text{P}}_{j,\n} \leftarrow \frac{1}{m_2} \sum_{k\in [m_2]} \Bigg| \frac{1}{m_1} \sum_{\ell \in [m_1]} \left({\v}_{\text{SID},D-1+j}\right)_{\ell,k} ~\overline{T_{[j+1];\n}\left(\w^{D-1+j}_\ell\right)}  \Bigg|^2$ for each $\n\in \T'_{D+j}$, see also~\eqref{equ:EstimatorGen}.
		\label{linrefEstimator2b}
		\State $\T_{D+j} \leftarrow \left\{ \n\in \T'_{D+j} ~\bigg|~ \min\big(2s,|\T'_{D+j}|\big) \text{-largest values } \mathbf{E}^{\text{P}}_{j,\n} \right\}$.
		\label{linrefEstimator2e}
		\EndFor
		\State Return $\tilde{\Omega} \leftarrow \T_{2D-1}$ \hfill (Note that it will always be true that $\left| \tilde{\Omega} \right| \leq 2s$.)
		\EndProcedure
	\end{algorithmic}
\end{algorithm}

\begin{mytheorem}
Let $\left\{ \w^j_\ell \right\}_{\ell \in [m_1]} \subset \mathcal{D}_{j}$ be $m_1$ points drawn independently at random according to $\mu_j$, and $\left\{ \z^j_k \right\}_{k \in [m_2]} \subset \mathcal{D}_{[D]\setminus\{j\}}$ be $m_2$ points drawn independently at random according to $\mu_{[D]\setminus\{j\}}$, for all $j \in [D]$.  Furthermore, let $\left\{ \w^{D-1+j}_\ell \right\}_{\ell \in [m_1]} \subset \mathcal{D}_{[j+1]}$ be $m_1$ points drawn independently at random according to $\mu_{[j+1]}$, and $\left\{ \z^{D-1+j}_k \right\}_{k \in [m_2]} \subset \mathcal{D}_{[D]\setminus [j+1]}$ be $m_2$ points drawn independently at random according to $\mu_{[D]\setminus [j+1]}$, for all $j \in [D]\setminus \{ 0 \}$.  If $m_1$ and $m_2$ are chosen to be sufficiently large for all $j \in [2D-1]$ then the following property will hold with probability $\geq 0.99$:

{\addtolength{\leftskip}{15 mm}
\noindent 
Algorithm~\ref{alg:suppid:impl} will output a set $\tilde{\Omega} \supset \Omega^{\alpha,2s}_{[D]}$ for all $h = \tilde{h}^{\rm opt}_{2s} + e_h$ as per \eqref{def:ffind} with coefficient vector $\rt \in \C^{\mathcal{I}_{N,d}}$ in the BOS~$\B$ satisfying
\begin{equation}
\left\| \rt_{\Omega^{\rm opt}_{\tilde{h},2s}} \right\|_2 > 25 \sqrt{23s} \left\| \rt - \rt_{\Omega^{\rm opt}_{\tilde{h},2s}} \right\|_2 + 18 \sqrt{23} \left\| \rt - \rt_{\Omega^{\rm opt}_{\tilde{h},2s}} \right\|_1 + 22 \gamma \sqrt{23s}.
 \label{eqn:SuppIDWorks}
\end{equation}
}

In order to achieve this property with probability $\geq 0.99$ it suffices for Algorithm~\ref{alg:suppid:impl} to utilize a total number of function evaluations from $h$ that is of size
$$m_{\rm SID} = m_1 m_2 (2D-1) = \mathcal{O} \left( D K^{4\tilde{d}}_\infty s^3 d^4 \cdot  \log^4 \left( \frac{DN}{d} \right) \log^2 (s) \log^2 (D) \right)$$
if the BOS constants $K_j$ are $1$ for all but at most $\tilde{d} \in \mathbbm{Z} \cap [0, D]$ BOS basis sets $\B_j$ (BOPB of type I), and that is of size
$$m'_{\rm SID} =  m_1 m_2 (2D-1) = \mathcal{O} \left( D K^{4d}_\infty s^3 d^4 \cdot  \log^4 \left( \frac{DN}{d} \right) \log^2 (s) \log^2 (D) \right)$$
if $K_0 = 1$ (BOPB of type II). 

In the BOPB of type I, the runtime complexity of Algorithm~\ref{alg:suppid:impl} will be 
$$\mathcal{O}\left( \left( s^5 + s^3 N \right) D K^{4\tilde{d}}_\infty d^4 \cdot  \log^4 \left( \frac{DN}{d} \right) \log^2 (s) \log^2 (D) \right),$$
and in the BOPB of type II, the runtime complexity will be
$$\mathcal{O} \left(  \left( s^5 + s^3 N \right) D K^{4d}_\infty d^4 \cdot  \log^4 \left( \frac{DN}{d} \right) \log^2 (s) \log^2 (D) \right).$$
\label{thm:SuppIDWorks}
\end{mytheorem}

\begin{proof}
See Section~\ref{sec:SuppIDHeavyEls}.  The desired result follows from a simplified version of Theorem~\ref{thm:GenAlgSuppIDWorks} with $\alpha = \sqrt{23}$.
\end{proof}

The index $j$ in Theorem~\ref{thm:SuppIDWorks} belongs to three different sets, $[D]$, $[D]\setminus \{0\}$ and $[2D-1]$. To explain, the set $[2D-1]$ comprehends all $j$'s belonging to the first two sets, $[D]$ and $[D]\setminus \{0\}$.

Theorem~\ref{thm:SuppIDWorks} combined with Lemma~\ref{lem:BigCoefsSuffice} is enough to guarantee that Algorithm~\ref{alg:suppid:impl} can identify a support set $\tilde{\Omega}$ that contains the majority of the energy of the $2s$-sparse vector $\rt_{\Omega^{\rm opt}_{\tilde{h},2s}}$.  However, Theorem~\ref{thm:iterInvariant} in Section~\ref{sec:CoSaMPguarantee} requires that $\left\|\r_{\widetilde{\Omega}^c} \right\|_2$ should be relatively small, where $\r \in \C^{\mathcal{I}_{N,d}}$ is the $2s$-sparse vector $\r := \x_s - \a^k = \ct_{\Omega^{\rm opt}_{\tilde{f},s}} - \a^k$ ($\r=\r^{k}$ in Section~\ref{sec:CoSaMPguarantee}).\footnote{Recall that $\ct \in \C^{\mathcal{I}_{N,d}}$ is the coefficient vector of $\tilde{f}$ as per \eqref{equ:fDefined}, and that $\a^k \in \mathbbm{C}^{\mathcal{I}_{N,d}}$ is CoSaMP's $s$-sparse approximation to $\x = \ct \in \C^{\mathcal{I}_{N,d}}$ in its $k^{\rm th}$-iteration.}  As a result we must now relate this $\r$ to the coefficients $\rt := \ct - \a^k = \x - \a^k$ of the function $\tilde{h}: = \tilde{f} - \sum_{\n \in \mathcal{I}_{N,d}} a^k_{\n} T_{\n}$ whose noisy samples we are passing into Algorithm~\ref{alg:suppid:impl} in line~\ref{suppIDcalled} of Algorithm~\ref{alg:main}.  The following lemma can be used to relate $\| \r \|_2$ to $\left\| \rt_{\Omega^{\rm opt}_{\tilde{h},2s}} \right\|_2$.\\

\begin{lemma}
Let $s \in \left[ \left| \mathcal{I}_{N,d} \right| / 2 \right]$, $\ct, \a^k \in \C^{\mathcal{I}_{N,d}}$ where $\left\| \a^k \right\|_0 \leq s$, and recall that $\rt := \ct - \a^k$, $\tilde{h}(\xib) : = \sum_{\n \in \mathcal{I}_{N,d}} \tilde{r}_{\n} T_{\n}(\xib)$, and $\tilde{f}(\xib) : = \sum_{\n \in \mathcal{I}_{N,d}} \tilde{c}_{\n} T_{\n}(\xib)$. 
One can see that
$$\| \r \|_2 = \left\| \ct_{\Omega^{\rm opt}_{\tilde{f},s}} - \a^k \right\|_2 \leq \left\| \rt_{\Omega^{\rm opt}_{\tilde{h},2s}} \right\|_2 + \left\| \ct - \ct_{\Omega^{\rm opt}_{\tilde{f},s}} \right\|_2.$$
\label{lem:suppIDFailErrorB}
\end{lemma}

\begin{proof}
Let $\Q := \left(\supp(\a^k) \cup \Omega^{\rm opt}_{\tilde{f},s} \right) \cap~ \supp(\rt)$, and note that $\left| \Omega^{\rm opt}_{\tilde{h},2s} \right| = \min \left\{2s, \left| \supp(\rt) \right| \right \} \geq \left| \Q \right|$.  As a result one can see that
\begin{align*}
\left\| \ct_{\Omega^{\rm opt}_{\tilde{f},s}} - \a^k \right\|_2 &= \left\| \ct_{\Omega^{\rm opt}_{\tilde{f},s}} - \a^k + \ct_{\supp(\a^k) \setminus \Omega^{\rm opt}_{\tilde{f},s}} - \ct_{\supp(\a^k) \setminus \Omega^{\rm opt}_{\tilde{f},s}} \right\|_2\\
&= \left\| \left( \ct - \a^k \right)_{\Omega^{\rm opt}_{\tilde{f},s} \cup \supp(\a^k)} - \ct_{\supp(\a^k) \setminus \Omega^{\rm opt}_{\tilde{f},s}} \right\|_2\\
&\leq \left\| \left( \ct - \a^k \right)_{\Omega^{\rm opt}_{\tilde{f},s} \cup \supp(\a^k)} \right\|_2 + \left\| \ct_{\supp(\a^k) \setminus \Omega^{\rm opt}_{\tilde{f},s}} \right\|_2\\
&= \left\| \rt_{\Q}  \right\|_2 + \left\| \ct_{\supp(\a^k) \setminus \Omega^{\rm opt}_{\tilde{f},s}} \right\|_2\\
&\leq  \left\| \rt_{\Q}  \right\|_2 + \left\| \ct - \ct_{\Omega^{\rm opt}_{\tilde{f},s}} \right\|_2\\
&\leq  \left\| \rt_{\Omega^{\rm opt}_{\tilde{h},2s}}  \right\|_2 + \left\| \ct - \ct_{\Omega^{\rm opt}_{\tilde{f},s}} \right\|_2.
\end{align*}
as we wished to show.
\end{proof}

The next lemma upper bounds the best $2s$-term approximation error of $\rt$ by the best $s$-term approximation error of $\ct$.  It will allow us to relate the condition \eqref{eqn:SuppIDWorks} under which Algorithm~\ref{alg:suppid:impl} succeeds to $\ct$.\\

\begin{lemma}
Let $s \in \left[ \left| \mathcal{I}_{N,d} \right| / 2 \right]$, $\ct, \a^k \in \C^{\mathcal{I}_{N,d}}$ where $\left\| \a^k \right\|_0 \leq s$, and recall that $\rt := \ct - \a^k$, $\tilde{h}(\xib) : = \sum_{\n \in \mathcal{I}_{N,d}} \tilde{r}_{\n} T_{\n}(\xib)$, and $\tilde{f}(\xib) : = \sum_{\n \in \mathcal{I}_{N,d}} \tilde{c}_{\n} T_{\n}(\xib)$. One can see that $\left\| \rt - \rt_{\Omega^{\rm opt}_{\tilde{h},2s}}  \right\|_p  \leq \left\| \ct - \ct_{\Omega^{\rm opt}_{\tilde{f},s}}  \right\|_p$ holds for all $p \geq 1$.  As a consequence, it will always be the case that 
\begin{align}
25 \sqrt{23s} \left\| \rt - \rt_{\Omega^{\rm opt}_{\tilde{h},2s}} \right\|_2 &+ 18 \sqrt{23} \left\| \rt - \rt_{\Omega^{\rm opt}_{\tilde{h},2s}} \right\|_1 + 22 \gamma \sqrt{23s} \label{Def:Gamma} \\ &\leq 25 \sqrt{23s} \left\| \ct - \ct_{\Omega^{\rm opt}_{\tilde{f},s}} \right\|_2 + 18 \sqrt{23} \left\| \ct - \ct_{\Omega^{\rm opt}_{\tilde{f},s}}  \right\|_1 + 22 \gamma \sqrt{23s} =: \Gamma. \nonumber
\end{align}
\label{lem:DetermineGamma}
\end{lemma}

\begin{proof}
A quick calculation reveals that
\begin{align*}
\left\| \rt - \rt_{\Omega^{\rm opt}_{\tilde{h},2s}} \right\|^p_p \leq  \left\| \rt - \rt_{\Omega^{\rm opt}_{\tilde{f},s} \cup \supp(\a^k)} \right\|^p_p &= \sum_{\n \in \mathcal{I}_{N,d} \setminus \left( \Omega^{\rm opt}_{\tilde{f},s} \cup \supp(\a^k) \right)} \left|\tilde{c}_{\n} - a^k_{\n} \right|^p\\ 
&= \sum_{\n \in \mathcal{I}_{N,d} \setminus \left( \Omega^{\rm opt}_{\tilde{f},s} \cup \supp(\a^k) \right)} |\tilde{c}_{\n}|^p\\
&\leq \left\| \ct - \ct_{\Omega^{\rm opt}_{\tilde{f},s}}  \right\|_p^p,
\end{align*}
as we wished to show.
\end{proof}

We are now able to assert that our support identification algorithm will work for all $2s$-sparse vectors $\rt_{\Omega^{\rm opt}_{\tilde{h},2s}}$ whose norms are sufficiently large with respect to the best $s$-term approximation error $\Gamma$ defined above in \eqref{Def:Gamma}.\\

\begin{lemma}
Let the $\w^j_\ell$ and $\z^j_k$ for $j \in [2D-1]$ in Algorithm~\ref{alg:suppid:impl} be chosen independently at random as per Theorem~\ref{thm:SuppIDWorks} above.  Then, the following property will hold with probability $\geq 0.99$:

\begin{quote}
Algorithm~\ref{alg:suppid:impl} will output a set $\tilde{\Omega} \subset \mathcal{I}_{N,d}$ with $\left| \tilde{\Omega} \right| \leq 2s$ that will also have 
\begin{equation}
\left\| \rt_{\Omega^{\rm opt}_{\tilde{h},2s} \cap \tilde{\Omega}^c} \right\|_2 \leq 0.2086 \left\|\rt_{\Omega^{\rm opt}_{\tilde{h},2s}} \right\|_2
\label{equ:WrongSuppID}
\end{equation}
for all $h = \tilde{h}^{\rm opt}_{2s} + e_h$ as per \eqref{def:ffind} with coefficient vector $\rt \in \C^{\mathcal{I}_{N,d}}$ in the BOS~$\B$ satisfying $\left\| \rt_{\Omega^{\rm opt}_{\tilde{h},2s}} \right\|_2 > \Gamma$, where $\Gamma$ is defined in \eqref{Def:Gamma}.
\end{quote}
\noindent
The runtime and sampling complexities of Algorithm~\ref{alg:suppid:impl} will remain as in Theorem~\ref{thm:SuppIDWorks} above.
\label{lem:SuppIDworks}
\end{lemma}

\begin{proof}
By Lemma \ref{lem:DetermineGamma}, $\left\| \rt_{\Omega^{\rm opt}_{\tilde{h},2s}} \right\|_2 > \Gamma$ implies that \eqref{eqn:SuppIDWorks} holds. Thus, the result follows from Theorem~\ref{thm:SuppIDWorks} combined with Lemma \ref{lem:BigCoefsSuffice}.
\end{proof}

The following theorem is the main theorem of this section.  It proves that the support set $\tilde{\Omega}$ found by Algorithm~\ref{alg:suppid:impl} will also contain the majority of the energy of the $2s$-sparse vector $\r := \ct_{\Omega^{\rm opt}_{\tilde{f},s}}- \a^k \in \C^{\mathcal{I}_{N,d}}$, as needed in Section~\ref{sec:CoSaMPguarantee}.\\

\begin{mytheorem}{(Support Identification).}
Let $s \in \left[ \left| \mathcal{I}_{N,d} \right| / 2 \right]$, $\ct, \a^k \in \C^{\mathcal{I}_{N,d}}$ where $\left\| \a^k \right\|_0 \leq s$, and recall that $\rt := \ct - \a^k$, $\tilde{h}(\xib) : = \sum_{\n \in \mathcal{I}_{N,d}} \tilde{r}_{\n} T_{\n}(\xib)$, and $\tilde{f}(\xib) : = \sum_{\n \in \mathcal{I}_{N,d}} \tilde{c}_{\n} T_{\n}(\xib)$. 
Suppose that the $\w^j_\ell$ and $\z^j_k$ in Algorithm~\ref{alg:suppid:impl} are chosen independently at random as per Theorem~\ref{thm:SuppIDWorks} above.  Then the following property will hold with probability $\geq 0.99$:

\begin{quote}
Algorithm~\ref{alg:suppid:impl} will output a set $\tilde{\Omega} \subset \mathcal{I}_{N,d}$ with 
\begin{equation*}
\left\| \r_{\tilde{\Omega}^c} \right\|_2 \leq 0.2086 \|\r\|_2 + 2.4172 \left\| \ct-\ct_{\Omega^{\rm opt}_{\tilde{f},s}}  \right\|_2
\end{equation*}
for any $\r = \ct_{\Omega^{\rm opt}_{\tilde{f},s}}- \a^k$ satisfying $\left\| \r \right\|_2 > \bar{\Gamma}$, where
\begin{equation}
\bar{\Gamma}:=\left( 25 \sqrt{23s} +1 \right) \left\| \ct - \ct_{\Omega^{\rm opt}_{\tilde{f},s}} \right\|_2 + 18 \sqrt{23} \left\| \ct - \ct_{\Omega^{\rm opt}_{\tilde{f},s}}  \right\|_1 + 22 \gamma \sqrt{23s}.
\label{Def:GammaBar}
\end{equation}
\end{quote}
\noindent
The runtime and sampling complexities of Algorithm~\ref{alg:suppid:impl} will remain as in Theorem~\ref{thm:SuppIDWorks} above.
	\label{thm:suppId}
\end{mytheorem}

\begin{proof}

Let $\rp := \rt_{\Omega^{\rm opt}_{\tilde{h},2s}}$.  Note that $\rt-\r = \ct - \ct_{\Omega^{\rm opt}_{\tilde{f},s}} $, and so
\begin{align}
\left\| \rp - \r \right\|_2
&= \left\|  \rp -\rt +\rt - \r  \right\|_2 \nonumber\\
&\leq \left\|  \rp - \rt \right\|_2 + \left\| \rt - \r  \right\|_2 \nonumber\\
&= \left\| \rt - \rp \right\|_2 + \left\| \ct - \ct_{\Omega^{\rm opt}_{\tilde{f},s}}  \right\|_2 \nonumber\\
&\leq 2\left\| \ct - \ct_{\Omega^{\rm opt}_{\tilde{f},s}}  \right\|_2,
\label{egn:diffRPR}
\end{align}
where the last inequality  holds by Lemma \ref{lem:DetermineGamma}.
Thus, we have that the following holds whenever \eqref{equ:WrongSuppID} does:
\begin{align*}
\left\| \r_{\tilde{\Omega}^c} \right\|_2 
&= \left\| \r_{\tilde{\Omega}^c} -\rp_{\tilde{\Omega}^c}+\rp_{\tilde{\Omega}^c}\right\|_2 \\
&\leq \left\| (\r-\rp )_{\tilde{\Omega}^c} \right\|_2 + \left\| \rp_{\tilde{\Omega}^c}\right\|_2 \\
&\leq \left\| (\r-\rp )_{\tilde{\Omega}^c} \right\|_2 + 0.2086 \left\| \rp \right\|_2\\
&\leq 2\left\| \ct - \ct_{\Omega^{\rm opt}_{\tilde{f},s}}  \right\|_2 + 0.2086 \left\| \rp -\r + \r \right\|_2\\
&\leq 2\left\| \ct - \ct_{\Omega^{\rm opt}_{\tilde{f},s}}  \right\|_2 + 0.2086 \left( \left\| \rp -\r \right\|_2 + \left\| \r \right\|_2 \right) \\
&\leq 2\left\| \ct - \ct_{\Omega^{\rm opt}_{\tilde{f},s}}  \right\|_2 + 0.2086 \left( 2\left\|\ct - \ct_{\Omega^{\rm opt}_{\tilde{f},s}}   \right\|_2 + \left\| \r \right\|_2 \right) \\
&= 2.4172\left\| \ct - \ct_{\Omega^{\rm opt}_{\tilde{f},s}}  \right\|_2 + 0.2086 \left\| \r \right\|_2, 
\end{align*}
where the second inequality holds if \eqref{equ:WrongSuppID} does, and the third and fifth inequalities hold by (\ref{egn:diffRPR}).

To finish we note that \eqref{equ:WrongSuppID} will indeed hold by Lemma \ref{lem:SuppIDworks} as long as $\left\| \rt_{\Omega^{\rm opt}_{\tilde{h},2s}} \right\|_2 > \Gamma = \bar{\Gamma} - \left\| \ct - \ct_{\Omega^{\rm opt}_{\tilde{f},s}} \right\|_2$.  And, $\left\| \rt_{\Omega^{\rm opt}_{\tilde{h},2s}} \right\|_2 > \bar{\Gamma} - \left\| \ct - \ct_{\Omega^{\rm opt}_{\tilde{f},s}} \right\|_2$ will hold whenever $\left\| \r \right\|_2 > \bar{\Gamma}$ holds by Lemma~\ref{lem:suppIDFailErrorB}.
\end{proof}


We will now focus on proving Theorem~\ref{thm:SuppIDWorks}.

\subsection{Proof of Theorem~\ref{thm:SuppIDWorks}:  Identifying $\Omega^{\alpha,2s}_{[D]}$ for $\tilde{h}^{\rm opt}_{2s}$ Using Samples from $h = \tilde{h}^{\rm opt}_{2s} + e_h$}
\label{sec:SuppIDHeavyEls}

Our strategy for finding $\Omega^{\alpha,2s}_{[D]}$ will involve building it up from a sequence of energetic partial index vectors of $\tilde{h}^{\rm opt}_{2s}$ that correspond to, e.g., the disjoint subsets of indices $$\S^{\rm EI}_j = \{ j \} \text{ for all }j \in [D].\footnote{Here the ``EI'' in the superscript of $\S^{\rm EI}_j$ stands for ``Entry Identification'' in the terminology of \cite{choi2018sparse}.  In fact many other valid choices for these sets also exist -- please see Algorithm~\ref{alg:suppid} for the general criteria they must satisfy.}$$  Note that the energetic partial index vectors in this case will contain the entries of the index vectors which have large associated values in $\rt_{\Omega^{\rm opt}_{\tilde{h},2s}}$.  That is, 
$$\Omega^{\alpha,2s}_{\S^{\rm EI}_j } = \Omega^{\alpha,2s}_{\{ j \}} \supseteq \left \{ n \e_j ~\big|~ \exists \n \in  \Omega^{\rm opt}_{\tilde{h},2s}~{\rm with}~  |\tilde{r}_{\n}| \geq \frac{ \left\| \rt_{\Omega^{\rm opt}_{\tilde{h},2s}} \right\|_2}{\alpha \sqrt{2s}} {\rm ~whose~ }j^{\rm th}~{\rm entry~ is~ }n \in [N] \right \},$$
where $\e_j \in \mathcal{I}_{N,d}$ is the $j^{\rm th}$ standard basis vector.
As a result, the set $\Omega^{\alpha,2s}_{\S^{\rm EI}_j }$ effectively contains all the $j^{\rm th}$-entries of the largest-magnitude coefficient vector indices in $\Omega^{\rm opt}_{\tilde{h},2s}$.  Furthermore, it is trivial to find a reasonably small superset of $\Omega^{\alpha,2s}_{\S^{\rm EI}_j }$ when, e.g., $N$ is not too large -- one can simply use the set $N^{\{j\}} = \left\{ n \e_j ~\big|~ n \in [N]  \right\} \supset \Omega^{\alpha,2s}_{\S^{\rm EI}_j }$.

Of course, the sets $\Omega^{\alpha,2s}_{\S^{\rm EI}_0 }, \dots, \Omega^{\alpha,2s}_{\S^{\rm EI}_{D-1}}$ are of limited utility in their own right when it comes to finding $\Omega^{\alpha,2s}_{[D]}$.
Our strategy will therefore be to use these sets to build up a sequence of new energetic partial index sets $\Omega^{\alpha,2s}_{\S^{\rm P}_1 }, \dots, \Omega^{\alpha,2s}_{\S^{\rm P}_{D-1}}$ each of which corresponds to an increasingly large subset of indices $\S^{\rm P}_j \subseteq [D]$.  In particular, if we define $\S^{\rm P}_j := \cup^j_{\ell = 0} \S^{\rm EI}_{\ell} =  \cup^j_{\ell = 0}  \{ \ell \} = [j+1]$ for all $j \in [D]\setminus \{ 0 \}$ we will eventually obtain a superset of $\Omega^{\alpha,2s}_{\S^{\rm P}_{D-1}} = \Omega^{\alpha,2s}_{[D]}$ (as desired) in a process that is analogous to the ``Pairing'' method utilized in \cite{choi2018sparse}.  
The following lemma is the basis for building up $\Omega^{\alpha,2s}_{[D]}$ by combining energetic partial index vectors of $\tilde{h}^{\rm opt}_{2s}$ that correspond to smaller index sets $\S_1, \S_2 \subset [D]$ in this fashion.  Recall that $\mathcal{P}\left(N^{\S} \right)$ denotes the power set of $N^{\S}$ for any given $\S \subseteq [D]$.\\ 

\begin{lemma}
Let $s' \in \N$, $\alpha \in (1, \infty)$, and $\S_1, \S_2 \subset [D]$ be disjoint.  If $\T_1 \in \mathcal{P}\left(N^{\S_1} \right)$ and $\T_2 \in \mathcal{P}\left(N^{\S_2} \right)$ are such that $\Omega^{\alpha,s'}_{\S_1} \subseteq \T_1$ and $\Omega^{\alpha,s'}_{\S_2} \subseteq \T_2$, then 
$$\Omega^{\alpha,s'}_{\S_1 \cup \S_2} \subseteq \T_{1,2} := \left\{ \n + \m ~\big|~ \n \in \T_{1},~\m \in \T_{2} \right\} \cap \mathcal{I}_{N,d} \subseteq N^{\S_1 \cup \S_2}.$$
\label{lem:IndexExpansion}
\end{lemma}
\vspace{-.3in}
\begin{proof}
Let $\rp := \rt_{\Omega^{\rm opt}_{\tilde{h},2s}}$, and note that for all $\n \in \mathcal{I}_{N,d}$ it is the case that 
$$\left\{ \m \in \mathcal{I}_{N,d} ~\big|~ \m_{\S_1 \cup \S_2} = \n_{ \S_1 \cup \S_2} \right \} \subseteq \left\{ \m \in \mathcal{I}_{N,d} ~\big|~ \m_{\S_1} = \n_{ \S_1} \right \} \cap \left\{ \m \in \mathcal{I}_{N,d} ~\big|~ \m_{\S_2} = \n_{\S_2} \right \}$$ 
holds.  As a consequence, for any $\n \in \mathcal{I}_{N,d}$ it will be the case that both 
$$\| \rp_{\S_1;\n} \|^2_2 ~= \sum_{\substack{\m \in \mathcal{I}_{N,d}~{\rm s.t.}\\ \m_{\S_1} = \n_{\S_1} }} \left| r'_{\m} \right|^2 ~\geq~\sum_{\substack{\m \in \mathcal{I}_{N,d}~{\rm s.t.}\\ \m_{\S_1 \cup \S_2} = \n_{ \S_1 \cup \S_2} }} \left| r'_{\m} \right|^2 = \| \rp_{\S_1 \cup \S_2;\n} \|^2_2$$
and 
$$\| \rp_{\S_2;\n} \|^2_2 ~= \sum_{\substack{\m \in \mathcal{I}_{N,d}~{\rm s.t.}\\ \m_{\S_2} = \n_{\S_2} }} \left| r'_{\m} \right|^2 ~\geq~\sum_{\substack{\m \in \mathcal{I}_{N,d}~{\rm s.t.}\\ \m_{\S_1 \cup \S_2} = \n_{ \S_1 \cup \S_2} }} \left| r'_{\m} \right|^2 = \| \rp_{\S_1 \cup \S_2;\n} \|^2_2$$
hold.  These inequalities in turn imply that $\Omega'_1 := \left \{ \n_{\S_1} ~\big|~\n \in \Omega^{\alpha,s'}_{\S_1 \cup \S_2} \right \} \subseteq \Omega^{\alpha,s'}_{\S_1} \subseteq \T_1$ and $\Omega'_2 := \left \{ \n_{\S_2} ~\big|~\n \in \Omega^{\alpha,s'}_{\S_1 \cup \S_2} \right \} \subseteq \Omega^{\alpha,s'}_{\S_2} \subseteq \T_2$.
Finally, the fact that $\S_1$ and $\S_2$ are disjoint now implies that 
$$\Omega^{\alpha,s'}_{\S_1 \cup \S_2} \subseteq \left\{ \n + \m ~\big|~ \n \in \Omega'_1 ,~\m \in \Omega'_2 \right\} \subseteq \left\{ \n + \m ~\big|~ \n \in \T_{1},~\m \in \T_{2} \right\}$$
is true as desired.
\end{proof}

Let $\S^{\rm P}_0 := \S^{\rm EI}_0$.  Note that applying Lemma~\ref{lem:IndexExpansion} repeatedly with, e.g., $\S_1 = \S^{\rm P}_j$, $\S_2 = \S^{\rm EI}_{j+1}$, $\T_1 = \T_{1,2}$ from the $(j-1)^{\rm st}$ application of Lemma~\ref{lem:IndexExpansion},\footnote{with, e.g., $\T_1 = N^{\S^{\rm EI}_0}$ when $j = 0$} and $\T_2 = N^{\S^{\rm EI}_{j+1}}$ for $j = 0, 1, \dots, D - 2$ will yield a superset of $\Omega^{\alpha,2s}_{[D]}$ on its $(D - 1)^{\rm st}$ application.  However, the cardinality of the resulting superset $\T_{1,2}$ of $\Omega^{\alpha,2s}_{\S_1 \cup \S_2}$ will also ballon to $| \T_1 | \cdot | \T_2 |$ at each step, eventually becoming exponentially large in $D$ on the $(D - 1)^{\rm st}$ application of Lemma~\ref{lem:IndexExpansion} in the worst case.  In order to prevent this worst case exponential growth in the size of the resulting sets $\T_{1,2}$ we will interleave the applications of Lemma~\ref{lem:IndexExpansion} with the use of an {\it energetic-index sieve function} $\F^{2s}_{\S_1 \cup \S_2}:  \mathcal{P}\left(N^{\S_1 \cup \S_2} \right) \rightarrow \mathcal{P}\left(N^{\S_1 \cup \S_2} \right)$ as in \eqref{def:SetValCardRedux} which reduces the cardinality of any $\T_{1,2} \supseteq \Omega^{\alpha,2s}_{\S_1 \cup \S_2}$ to $2s$ without loosing any of $\Omega^{\alpha,2s}_{\S_1 \cup \S_2}$
 These sieve functions will allow Lemma~\ref{lem:IndexExpansion} to be applied repeatedly as above while maintaining output sets of small cardinality at all stages, which we can see how they work in lines \ref{linrefEstimator2b} and \ref{linrefEstimator2e} of Algorithm \ref{alg:suppid:impl}.

The next theorem proves the existence of a set-valued function $\F^{2s}_{\S}: \mathcal{P}\left(N^{\S} \right) \rightarrow \mathcal{P}\left(N^{\S} \right)$ for any given $\S \subseteq [D]$ which, when given any subset $\T \subset N^{\S}$ containing $\Omega^{\alpha,2s}_{\S}$ as per \eqref{def:HeavySet} as input, will output a smaller subset $\T' \subset \T$  of cardinality at most $2s$ which still contains $\Omega^{\alpha,2s}_{\S}$.  Note that these set valued functions necessarily depend on the function $\tilde{h}^{\rm opt}_{2s}$ in question via the definition of $\Omega^{\alpha,2s}_{\S}$.  However, it is crucial to note that all the $\F^{2s}_{\S}$ considered herein only utilize a few point samples from $h = \tilde{h}^{\rm opt}_{2s} + e_h$ (i.e., noisy point samples from $\tilde{h}^{\rm opt}_{2s}$) on a fixed and nonadaptive grid.  More specifically, the grid on which each $\F^{2s}_{\S}$ samples $h$ depends only on $2s, \S,$ and the BOPB $\B$ with respect to which $h$ is presumed to be approximately sparse, and not at all on the particular function $h$ in question.\\

\begin{mytheorem}[Existence of Low-Complexity Energetic-Index Sieve Functions]
Choose $t' \in [2D]$ and any desired $\S_0, \dots, \S_{t'} \subseteq [D]$.  For all $j \in [t']$ there exists an associated energetic-index sieve function $\F^{2s}_{\S_j}:  \mathcal{P}\left(N^{\S_j} \right) \rightarrow \mathcal{P}\left(N^{\S_j} \right)$ 
for which both 
\begin{enumerate}
\item $\Omega^{\alpha,2s}_{\S_j} \cap \T \subseteq \F^{2s}_{\S_j} \left( \T \right)$ holds for all $\T \in \mathcal{P}\left(N^{\S_j} \right)$, and 
\item $\left| \F^{2s}_{\S_j} \left( \T \right) \right| \leq 2s$ holds for all $\T \in \mathcal{P}\left(N^{\S_j} \right)$, 
\end{enumerate}
are true for all $h: \D \rightarrow \C$ as above \eqref{def:ffind} that satisfy
$$ \left\|\tilde{h}^{\rm opt}_{2s} \right \|_{L^2(\D,\mu)} =  \left\| \rt_{\Omega^{\rm opt}_{\tilde{h},2s}} \right\|_2 ~>~ 25 \alpha \sqrt{s} \left\| \rt - \rt_{\Omega^{\rm opt}_{\tilde{h},2s}} \right\|_2 + 18 \alpha \left\| \rt - \rt_{\Omega^{\rm opt}_{\tilde{h},2s}} \right\|_1 + 22 \alpha  \gamma \sqrt{s}.\footnote{The constants here have been rounded up to the nearest integer from those implied by Theorem~\ref{thm:pairingFunc} and Lemma~\ref{lem:errorBound} after substituting $s'=2s$.}$$
Furthermore, each $\F^{2s}_{\S_j}:  \mathcal{P}\left(N^{\S_j} \right) \rightarrow \mathcal{P}\left(N^{\S_j} \right)$ is computed using evaluations of any given $h = \tilde{h}^{\rm opt}_{2s} + e_h$ at $m^j_1 m^j_2$ fixed and nonadaptive grid points $\left\{ \left(\w^j_\ell,\z^j_k \right) \right\}_{\ell \in [m^j_1],k \in [m^j_2]} \subset \D$, $m^j_1, m^j_2 \in \N$, where $\w^j_\ell \in \D_{\S_j}$ and $\z^j_k \in \D_{\S^c_j}$ for all $j \in [t']$, $\ell \in [m^j_1]$, and $k \in [m^j_2]$.

If the BOS constants $K_j$ are $1$ for all but at most $\tilde{d} \in \mathbbm{Z} \cap [0, D]$ BOS basis sets $\B_j$, then each such $\F^{2s}_{\S_j}:  \mathcal{P}\left(N^{\S_j} \right) \rightarrow \mathcal{P}\left(N^{\S_j} \right)$ above requires only
$$m_j = m^j_1 m^j_2 = \mathcal{O} \left( K^{4\tilde{d}}_\infty s^3 d^4 \cdot  \log^4 \left( \frac{DN}{d} \right) \log^2 (s) \log^2 (D) \right)$$
evaluations of any given $h = \tilde{h}^{\rm opt}_{2s} + e_h$ at $m_j$ fixed and nonadaptive grid points $\subset \D$.\footnote{It is important to emphasize here that the grid on which we must evaluate each function $f$ is a fixed grid which does not change depending on $h$.}  As a result, $\F^{2s}_{\S_j} \left( \T \right)$ can be computed in just $\mathcal{O}(m_j |\T|)$-time for any given $\T \in \mathcal{P}\left(N^{\S_j} \right)$ and $h$ in this case.\footnote{Herein we assume that $h$ has been evaluated in advance on our non-adaptive grid so that its values at each grid point can be retrieved in $\mathcal{O}(1)$-time.  In addition, note that setting $d = D$ above still leads to sampling and runtime complexities for each sieve function that scale only polynomially in $D$.  This is due to $\tilde{d}$ being independent of $d$.}
If, on the other hand, $K_0 = 1$ then each such $\F^{2s}_{\S_j}:  \mathcal{P}\left(N^{\S_j} \right) \rightarrow \mathcal{P}\left(N^{\S_j} \right)$ requires only
$$m'_j = m^j_1 m^j_2  = \mathcal{O} \left( K^{4d}_\infty s^3 d^4 \cdot  \log^4 \left( \frac{DN}{d} \right) \log^2 (s) \log^2 (D) \right)$$
evaluations of any given $h$ at $m'_j$ fixed and nonadaptive grid points in $\D$.  As a result, $\F^{2s}_{\S_j} \left( \T \right)$ can be computed in just $\mathcal{O}(m'_j |\T|)$-time for any given $\T \in \mathcal{P}\left(N^{\S_j} \right)$ and $h$ in this case.
\label{thm:ExistFastFuncs}
\end{mytheorem}

\begin{proof}
See Section~\ref{sec:proofBigSuppIDThm} below.  The proof follows by applying Theorem~\ref{thm:pairingFunc} with $e_h = \tilde{h} - \tilde{h}^{\rm opt}_{2s} + (h - \tilde{h}) = h - \tilde{h}^{\rm opt}_{2s}$ for each set $\S_0, \dots, \S_{t'}$.  After recalling that $\displaystyle \sup_{\xib \in \D} \left| \left(h - \tilde{h} \right)(\xib) \right| \leq \gamma$ we can see that Lemma~\ref{lem:errorBound} will also apply in each case.  Finally, the runtime and sampling complexity bounds follow from Lemma~\ref{lem:SamplingBounds} and Remark~\ref{rem:SampBound}.  
\end{proof}

\begin{remark}
It is important to note that Section~\ref{sec:proofBigSuppIDThm} proves more than mere existence of the collection of low-complexity energetic-index sieve functions promised in Theorem~\ref{thm:ExistFastFuncs}. In fact it proves their existence by proving that one can generate such a collection with high probability $\geq$, e.g., $0.99$ by letting $\{ \w^j_\ell \}_{\ell \in [m^j_1]} \subset \D_{\S_j}$ be $m^j_1$ sampling points drawn independently at random according to $\mu_{\S_j}$, and by letting $\{ \z^j_k \}_{k \in [m^j_2]} \subset \D_{\S^c_j}$ be $m^j_2$ sampling points drawn independently at random according to $\mu_{\S^c_j}$, for all $j \in [t' ]$.  This is done by showing that randomly selecting the nonadaptive grid points in this fashion ultimately guarantees that their related random sampling matrices in \eqref{equ:EI_RIPmat} and \eqref{equ:EI_RIPmat2} have well behaved restricted isometry constants.  See Remark~\ref{rem:SampBound} for additional details and related discussion.
\end{remark}

With Lemma~\ref{lem:IndexExpansion} and Theorem~\ref{thm:ExistFastFuncs} in hand one can now see that Algorithm~\ref{alg:suppid} will be guaranteed to return a superset $\tilde{\Omega}$ of $\Omega^{\alpha,2s}_{[D]}$ whose cardinality is at most $2s$.  \\

\begin{algorithm}[ht]
\caption{Support Identification}
\label{alg:suppid}
\begin{algorithmic}[1]
\Procedure{$\mathbf{Generalized~SupportID}$}{}\\
{\textbf{Parameters: }}{$s\in\mathbbm{N}$, $t \in [D] \setminus \{0\}$, ~A Partition of $[D]$ into $\S^{\rm EI}_0, \dots, \S^{\rm EI}_{t} \subset [D]$, ~and the Associated Pairing Index Sets $\S^{\rm P}_j := \cup^j_{\ell = 0} \S^{\rm EI}_{\ell} \subseteq [D]$ for all $j \in [t+1] \setminus \{0\}$.}\\
{\textbf{Input: }}{${\v}_{\text{SID}} \in \C^{\sum^{2t}_{j=0} m^j_1 m^j_2}$ split into $2t+1$ blocks $\left\{ h \left( \varrho_{\S^{\rm EI}_j} \left( \w^j_\ell,\z^j_k \right) \right) \right\}_{\ell \in [m^j_1],k \in [m^j_2]}^{j \in [t+1]} \; \bigcup \; \left\{ h \left( \varrho_{\S^{\rm P}_{j-t}} \left(\w^j_\ell,\z^j_k \right) \right) \right\}_{\ell \in [m^j_1],k \in [m^j_2]}^{j\in [2t+1] \setminus [t+1]}$ indexed by $j$ w/ $\w^j_\ell \in \D_{\S^{\rm EI}_j}~\forall~j \in [t+1] ~\&~ \ell \in [m^j_1]$,  $\w^{j}_\ell \in \D_{\S^{\rm P}_{j - t}} ~\forall~ j \in [2t+1] \setminus[t+1] ~\&~\ell \in \left[m^{j}_1 \right]$, $\z^j_k \in \D_{\left(\S^{\rm EI}_j \right)^c}~\forall~j \in [t+1] ~\&~ k \in [m^j_2]$, \& $\z^{j}_k \in \D_{\left(\S^{\rm P}_{j - t}\right)^c} ~\forall~ j \in [2t+1] \setminus[t+1] ~\&~k \in \left[m^{j}_2 \right]$. }\\
{\textbf{Output: }}{A set $\tilde{\Omega} \supset \Omega^{\alpha,2s}_{[D]}$}
\State Compute $\mathcal{N}_j \leftarrow \F^{2s}_{\S^{\rm EI}_j}\left( N^{\S^{\rm EI}_j} \right)$ using $\left\{ h \left( \varrho_{\S^{\rm EI}_j} \left( \w^j_\ell,\z^j_k \right) \right) \right\}_{\ell \in [m^j_1],k \in [m^j_2]}$ for each $j \in [t+1]$
\State $\T_{t} \leftarrow \mathcal{N}_0$
\For{$j = t+1$ {\bf up to} $2t$}{}
\State $\T'_j \leftarrow  \left\{ \n + \m ~\big|~ \n \in \T_{j-1},~\m \in \mathcal{N}_{j-t} \right\} \cap \mathcal{I}_{N,d} \subseteq N^{\S^{\rm P}_{j-t}}$ 
\State $\T_j \leftarrow \F^{2s}_{\S^{\rm P}_{j-t}} \left( \T'_j \right)$ using $\left\{ h \left( \varrho_{\S^{\rm P}_{j-t}} \left(\w^j_\ell,\z^j_k \right) \right) \right\}_{\ell \in [m^j_1],k \in [m^j_2]}$ 
\EndFor
\State Return $\tilde{\Omega} \leftarrow \T_{2t}$
\EndProcedure
\end{algorithmic}
\end{algorithm}

\begin{mytheorem}
Let $\S^{\rm EI}_0, \dots, \S^{\rm EI}_{t} \subset [D]$ form a partition of $[D]$ for $t \in [D] \setminus \{0\}$ and set $\S^{\rm P}_j := \cup^j_{\ell = 0} \S^{\rm EI}_{\ell} \subseteq [D]$ for all $j \in [t+1] \setminus \{0\}$ as per Algorithm~\ref{alg:suppid}.  Let $\F^{2s}_{\S^{\rm EI}_j}:  \mathcal{P}\left(N^{\S^{\rm EI}_j} \right) \rightarrow \mathcal{P}\left(N^{\S^{\rm EI}_j} \right)$ and $\F^{2s}_{\S^{\rm P}_{j}}:  \mathcal{P}\left(N^{\S^{\rm P}_{j}} \right) \rightarrow \mathcal{P}\left(N^{\S^{\rm P}_{j}} \right)$ be their associated energetic-index sieve functions.  When executed using these energetic-index sieve functions Algorithm~\ref{alg:suppid} will output a set $\tilde{\Omega}$ with $\left| \tilde{\Omega} \right| \leq 2s$ that will also have $\Omega^{\alpha,2s}_{[D]} \subset \tilde{\Omega}$ provided that $h = \tilde{h}^{\rm opt}_{2s} + e_h$ has
\begin{equation}
\left\|\tilde{h}^{\rm opt}_{2s} \right \|_{L^2(\D,\mu)} =  \left\| \rt_{\Omega^{\rm opt}_{\tilde{h},2s}} \right\|_2 ~>~ 25 \alpha \sqrt{s} \left\| \rt - \rt_{\Omega^{\rm opt}_{\tilde{h},2s}} \right\|_2 + 18 \alpha \left\| \rt - \rt_{\Omega^{\rm opt}_{\tilde{h},2s}} \right\|_1 + 22 \alpha  \gamma \sqrt{s}.
 \label{eqn:SeiveFtnWorkCondition}
\end{equation}

The total number of function evaluations required \footnote{In the bounds below $t$ may be upper bounded by $D$.} by Algorithm~\ref{alg:suppid} is 
$$m_{\rm SID} = \mathcal{O} \left( t K^{4\tilde{d}}_\infty s^3 d^4 \cdot  \log^4 \left( \frac{DN}{d} \right) \log^2 (s) \log^2 (D) \right)$$
if the BOS constants $K_j$ are $1$ for all but at most $\tilde{d} \in \mathbbm{Z} \cap [0, D]$ BOS basis sets $\B_j$, and is
$$m'_{\rm SID} = \mathcal{O} \left( t K^{4d}_\infty s^3 d^4 \cdot  \log^4 \left( \frac{DN}{d} \right) \log^2 (s) \log^2 (D) \right)$$
if $K_0 = 1$.

The runtime complexity of Algorithm~\ref{alg:suppid} will be 
$$\mathcal{O}\left( \left( s^5 + s^3 \max_{j \in [t+1]} \left | N^{\S^{\rm EI}_j} \right | \right) t K^{4\tilde{d}}_\infty d^4 \cdot  \log^4 \left( \frac{DN}{d} \right) \log^2 (s) \log^2 (D) \right)$$
if the BOS constants $K_j$ are $1$ for all but at most $\tilde{d} \in \mathbbm{Z} \cap [0, D]$ BOS basis sets $\B_j$, and 
$$\mathcal{O} \left(  \left( s^5 + s^3 \max_{j \in [t+1]} \left | N^{\S^{\rm EI}_j} \right | \right) t K^{4d}_\infty d^4 \cdot  \log^4 \left( \frac{DN}{d} \right) \log^2 (s) \log^2 (D) \right)$$
if $K_0 = 1$.
\label{thm:GenAlgSuppIDWorks}
\end{mytheorem}

\begin{proof}
The proof follows directly from Lemma~\ref{lem:IndexExpansion} and Theorem~\ref{thm:ExistFastFuncs}.
\end{proof}

\begin{remark}
Note that Algorithm~\ref{alg:suppid:impl} is a special case of Algorithm~\ref{alg:suppid} with $t = D-1$, $\S^{\rm EI}_j := \{j\}$ for all $j \in [D]$, $\S^{\rm P}_j := \cup^j_{\ell = 0} \S^{\rm EI}_{\ell} = \cup^j_{\ell = 0} \{ \ell \} \subseteq [D]$ for all $j \in [D] \setminus \{0\}$, and where the sieve functions $\F^{2s}_{\S^{\rm EI}_j}$, $\F^{2s}_{\S^{\rm P}_{j-t}}$ have been written down explicitly using \eqref{equ:EstimatorGen}, \eqref{def:SizeOrder}, and \eqref{def:SetValCardRedux}.  Therein the $\F^{2s}_{\S^{\rm EI}_j}\left( N^{\S^{\rm EI}_j} \right)$ are computed for all $j \in [t+1]$ by lines \ref{linrefEstimator1b} -- \ref{linrefEstimator1e} of Algorithm~\ref{alg:suppid:impl}, and each $\F^{2s}_{\S^{\rm P}_{j-t}} \left( \T'_j \right)$ in Algorithm~\ref{alg:suppid} is computed by lines \ref{linrefEstimator2b} -- \ref{linrefEstimator2e} of Algorithm~\ref{alg:suppid:impl}.
\end{remark}

Though dedicated to proving Theorem~\ref{thm:ExistFastFuncs}, this next subsection will be initially focussed on learning $\Omega^{\alpha,s'}_{\S}$ for arbitrary BOPB-sparse functions with $h = \tilde{h} = \tilde{h}^{\rm opt}_{\tilde{h},s'}$ for which $\rp = \rt_{\Omega^{\rm opt}_{\tilde{h},s'}}$.  It will then be generalized to cover more general functions $h$ of the type discussed above \eqref{def:ffind} toward its end as an extension of the noisy sparse case.  A proof of Theorem~\ref{thm:ExistFastFuncs} may then be obtained by setting $s'=2s$.

\subsection{Proof of Theorem~\ref{thm:ExistFastFuncs}: Generalized Entry Identification \& Pairing}
\label{sec:proofBigSuppIDThm}

In the vast majority of this subsection we will be considering an arbitrary function $\tilde{h}\colon \D \rightarrow \C$ of $D$ variables as per \eqref{def:htilde} whose coefficient vector $\rt \in \C^{\mathcal{I}_{N,d}}$ is only nonzero for entries indexed by index vectors $\q \in \mathcal{I}_{N,d}$.  In particular, we will be focussing almost exclusively on the development of efficient strategies for learning about the support of the coefficient vector $\rt$ of such $\tilde{h}$ in the special case where $\rt$ is $s'$-sparse so that $\tilde{h} = \tilde{h}^{\rm opt}_{s'}$ and $\rt = \rt_{\Omega^{\rm opt}_{\tilde{h},s'}}$.
Our first lemma does this by telling us how to estimate the $\ell^2$-norm of any $\rpp = \left(\rt_{\Omega^{\rm opt}_{\tilde{h},s'}}\right)_{\S;\n} = \rt_{\left\{ \q \in \Omega^{\rm opt}_{\tilde{h},s'} ~\big| ~ \q_{\S} = \n_{\S} \right\}} \in \C^{\mathcal{I}_{N,d}}$ in that case 
(i.e., how to estimate of the energy of all the coefficients of $\tilde{h}^{\rm opt}_{s'}$ whose index vectors $\q \in \mathcal{I}_{N,d}$ match another fixed index vector $\n \in \mathcal{I}_{N,d}$ in all index positions $\S \subset [D]$) by using just a few inner products with ``simpler'' functions of only $|\S| < D$ variables.  The idea is that these inner products will be easy to approximate numerically for $|\S|$ small.  As a result, one can hope to learn about the index vectors of the nonzero entries of any such $\rpp$ by approximately computing just a few inner products involving functions of just a few variables in order to, e.g., discover values of $\n$ for which $\| \rpp \|_2$ is large.\\

\begin{lemma}
Let $\delta \in (0, 3/4]$, $\S \subset [D]$, and $\{\z_{k}\}_{k \in [m_2]} \subset \D_{\S^c}$ be $m_2$ sampling points drawn independently at random according to $\mu_{\S^c}$ in order to form a zero-padded random sampling matrix $\Phi_{\S^c;\zero} \in\C^{m_2\times |\mathcal{I}_{N,d}|}$ for the BOS $\B_{\S^c}$ as in \eqref{def:B_S}  with entries
\begin{equation}
\left( \Phi_{\S^c;\zero} \right)_{k,\q} :=  \begin{cases}
       T_{\S^c;\q}(\z_k) & \textrm{if}~\q_{\S} = \zero ~\&~ \q \in \mathcal{I}_{N,d}\\
       0 &\text{ otherwise}\\
   \end{cases}
\label{equ:EI_RIPmat}
\end{equation}
indexed by $k \in [m_2]$ and $\q \in \mathcal{I}_{N,d}$.   Suppose the nonzero columns of $\frac{1}{\sqrt{m_2}}\Phi_{\S^c;\zero}$ have the restricted isometry property (RIP) of order $(s,\delta)$.  Then, for all $\n \in \mathcal{I}_{N,d} \subseteq [N]^D$, vectors of additive evaluation errors $\e^h \in \C^{m_2}$, and functions $\tilde{h}$ as per \eqref{def:htilde} one will have 
\begin{equation}
\left| \sqrt{\sum_{k\in [m_2]} \frac{1}{m_2} \left| \left \langle (\tilde{h}_{s'}^{\rm opt})_{\S^c;\z_k}, T_{\S;\n} \right \rangle_{\left( \D_{\S}, \mu_{\S}\right)} + e^h_k \right|^2 } - \| \rpp \|_2 \right|
\leq
\frac{2}{3}
\delta \| \rpp \|_2 +\frac{\| \e^h \|_2}{\sqrt{m_2}}
\label{equ:EI_error_guar}
\end{equation}
where $\rpp := \rt_{\left\{ \q \in \Omega^{\rm opt}_{\tilde{h},s'} ~\big| ~ \q_{\S} = \n_{\S} \right\}}$.
\label{lem:EntryID_ExactInnerProd}
\end{lemma}  

\begin{proof}
Consider the zero-padded random sampling matrix $\Phi_{\S^c;\n} \in\C^{m_2\times |\mathcal{I}_{N,d}|}$ for the BOS $\B_{\S^c}$ as in \eqref{def:B_S} with entries
\begin{equation}
\left( \Phi_{\S^c;\n} \right)_{k,\q} :=  \begin{cases}
       T_{\S^c;\q}(\z_k) & \textrm{if}~\q_{\S} = \n_{\S} ~\&~ \q \in \mathcal{I}_{N,d}\\
       0 &\text{ otherwise}\\
   \end{cases}
\label{equ:EI_RIPmatn}
\end{equation}
indexed by $k \in [m_2]$ and $\q \in \mathcal{I}_{N,d}$.  Note that $\q_{\S} = \n_{\S} ~\&~ \q \in \mathcal{I}_{N,d} \implies (\q,\zero)_{\S^c} \in \mathcal{I}_{N,d}$ for all $d \in [D+1] \setminus \{0\}, \S^c \subset [D],$ and $\n,\q \in \mathcal{I}_{N,d}$. As a result, the nonzero columns of $\Phi_{\S^c;\zero}$ will contain the nonzero columns of $\Phi_{\S^c;\n}$ as a subset.\footnote{Note that the nonzero columns of $\Phi_{\S^c;\n}$ will be indexed by different $\q$ in $\Phi_{\S^c;\zero}$.  However, this reindexing will ultimately just represent a permutation of the nonzero columns of $\Phi_{\S^c;\n}$ as a submatrix of $\Phi_{\S^c;\zero}$.  And, permuting the columns of a matrix does not change its restricted isometry constants.}  This further implies that the matrix consisting of the nonzero columns of $\frac{1}{\sqrt{m_2}}\Phi_{\S^c;\n}$ will also have the restricted isometry property (RIP) of order $(s,\delta)$.

Applying Lemma~\ref{lem:PartEvalinnerP} together with the definition of $\tilde{h}_{s'}^{\rm opt}$ we now have that
$$\sum_{k \in [m_2]} \frac{1}{m_2} \left| \left \langle (\tilde{h}_{s'}^{\rm opt})_{\S^c;\z_k}, T_{\S;\n} \right \rangle_{\left( \D_{\S}, \mu_{\S}\right)} + e^h_k \right|^2 =\sum_{k \in [m_2]}  \left| \frac{1}{\sqrt{m_2}} \left \langle~ \left(\rt_{\Omega^{\rm opt}_{\tilde{h},s'}}\right)_{\S;\n},  \overline{\Phi_{\S^c;\n;\z_k}} ~\right \rangle +   \frac{e^h_k}{\sqrt{m_2}} \right|^2 .$$
Noting now that each vector $\Phi_{\S^c;\n;\z_k}$ as per \eqref{equ:RowRIPlem1} can be replaced by an equivalent row of $\Phi_{\S^c;\n}$ in \eqref{equ:EI_RIPmatn} we can further see that
\begin{align*}
\sqrt{\sum_{k \in [m_2]} \frac{1}{m_2} \left| \left \langle (\tilde{h}_{s'}^{\rm opt})_{\S^c;\z_k}, T_{\S;\n} \right \rangle_{\left( \D_{\S}, \mu_{\S}\right)} + e^h_k \right|^2 } &= \sqrt{\sum_{k \in [m_2]}  \left| \frac{1}{\sqrt{m_2}} \left(\Phi_{\S^c;\n} (\rt_{\Omega^{\rm opt}_{\tilde{h},s'}})_{\S;\n} \right)_k +  \frac{e^h_k}{\sqrt{m_2}} \right|^2 }\\
&= \left \| \frac{1}{\sqrt{m_2}} \Phi_{\S^c;\n} (\rt_{\Omega^{\rm opt}_{\tilde{h},s'}})_{\S;\n} + \frac{\e^h}{\sqrt{m_2}} \right\|_2\\
&= \left \| \frac{1}{\sqrt{m_2}} \Phi_{\S^c;\n} \, \rpp + \frac{\e^h}{\sqrt{m_2}} \right\|_2.
\end{align*}
Using that $\Phi_{\S^c;\n}$ has the restricted isometry property (RIP) of order $(s,\delta)$ together with the (reverse) triangle inequality on this last line we now can see that
\begin{align*}
\sqrt{1 - \delta} \| \rpp \|_2 - \frac{\| \e^h \|_2}{\sqrt{m_2}} &\leq \sqrt{\sum_{k \in [m_2]} \frac{1}{m_2} \left| \left \langle (\tilde{h}_{s'}^{\rm opt})_{\S^c;\z_k}, T_{\S;\n} \right \rangle_{\left( \D_{\S}, \mu_{\S}\right)} + e^h_k \right|^2 } \\ &\leq \sqrt{1 + \delta} \| \rpp \|_2 +\frac{\| \e^h \|_2}{\sqrt{m_2}}.
\end{align*}
After subtracting $\| \rpp \|_2$ from the quantities in the inequality above, we finally use the bounds
$(\sqrt{1 + \delta} - 1) \frac{\sqrt{1 + \delta} + 1}{\sqrt{1 + \delta} + 1} = \frac{\delta}{\sqrt{1 + \delta} + 1} < \frac{\delta}{2}$ and $(\sqrt{1 - \delta} - 1) \frac{\sqrt{1 - \delta} + 1}{\sqrt{1 - \delta} + 1} = \frac{-\delta}{\sqrt{1 - \delta} + 1} \geq -\frac{2}{3} \delta$ to finish the proof.
\end{proof}

Lemma~\ref{lem:EntryID_ExactInnerProd} yields an alternate entry identification technique to that provided in Section 4.1 of \cite{choi2018sparse}.  In particular, if $\S^c = [D] \setminus \{ j \}$ for some $j \in [D]$ the inner products $\left \langle (\tilde{h}_{s'}^{\rm opt})_{\S^c;\z_k}, T_{\S;\n} \right \rangle_{\left( \D_{\S}, \mu_{\S}\right)}$ are just one-dimensional integrals that can be computed to high accuracy for any desired $\n \in \mathcal{I}_{N,d}$ using only $\mathcal{O}(N)$ function evaluations of $(\tilde{h}_{s'}^{\rm opt})_{\S^c;\z_k}: \D_{\S} \rightarrow \C$ via, e.g., a quadrature rule whenever the the basis functions in the $j^{\rm th}$-dimension, $\B_j = \left \{ T_{\S;\n} ~|~ \n \in \mathcal{I}_{N,d} \right \}$, are polynomials of degree at most $N$.  If $\B_j$ is either the Fourier or Chebyshev basis and $N$ is very large then these one-dimensional integrals can also be computed for all $\n \in \mathcal{I}_{N,d}$ in sublinear-in-$N$ time since $(\tilde{h}_{s'}^{\rm opt})_{\S^c;\z_k}$ will be $\B_j$-sparse (see, e.g., \cite{gilbert2014recent,gilbert2005improved,iwen2007empirical,iwen2008deterministic,iwen2010combinatorial,bailey2012design,hassanieh2012simple,iwen2013improved,segal2013improved,merhi2017new,hu2017rapidly,bittens2019deterministic}).  

When $|\S^c| \ll D-1$ the situation becomes more difficult.  However, to efficiently evaluate the higher-dimensional inner products $\left \langle (\tilde{h}_{s'}^{\rm opt})_{\S^c;\z_k}, T_{\S;\n} \right \rangle_{\left( \D_{\S}, \mu_{\S}\right)}$ that arise in these settings one can instead utilize non-adaptive random sampling techniques motivated by compressive sensing theory.  The following lemma does this by quantifying how well the estimator $$\frac{1}{m_1} \sum_{\ell \in [m_1]} (\tilde{h}_{s'}^{\rm opt})_{\S^c;\z_k} (\w_\ell) ~\overline{T_{\S;\n}(\w_\ell)}$$ based on the $m_1$ randomly chosen grid points $\left\{ \w_\ell \right\}_{\ell \in [m_1]} \subset \D_{\S}$ approximates all such $\left \langle (\tilde{h}_{s'}^{\rm opt})_{\S^c;\z_k}, T_{\S;\n} \right \rangle_{\left( \D_{\S}, \mu_{\S}\right)}$.\\

\begin{lemma}
Let $\tilde{\delta} \in (0,1)$, $\S \subset [D]$, and $\{\w_{\ell}\}_{\ell \in [m_1]} \subset \D_{\S}$ be $m_1$ sampling points drawn independently at random according to $\mu_{\S}$ in order to form a zero-padded random sampling matrix $\Phi_{\S;\zero} \in\C^{m_1 \times |\mathcal{I}_{N,d}|}$ for the BOS $\B_{\S}$  in \eqref{def:B_S} with entries
\begin{equation}
\left( \Phi_{\S;\zero} \right)_{\ell,\q} :=  \begin{cases}
       T_{\S;\q}(\w_\ell) & \textrm{if}~\q_{\S^c} = \zero ~\&~ \q \in \mathcal{I}_{N,d}\\
       0 &\text{ otherwise}\\
   \end{cases}
\label{equ:EI_RIPmat2}
\end{equation}
indexed by $\ell \in [m_1]$ and $\q \in \mathcal{I}_{N,d}$.   Suppose the nonzero columns of $\frac{1}{\sqrt{m_1}}\Phi_{\S;\zero}$ have the restricted isometry property (RIP) of order $(2,\tilde{\delta})$, and let  
$$e^h_k := \frac{1}{m_1} \sum_{\ell \in [m_1]} (\tilde{h}_{s'}^{\rm opt})_{\S^c;\z_k} (\w_\ell) ~\overline{T_{\S;\n}(\w_\ell)} - \left \langle (\tilde{h}_{s'}^{\rm opt})_{\S^c;\z_k}, T_{\S;\n} \right \rangle_{\left( \D_{\S}, \mu_{\S}\right)}$$
for any desired $\n \in \mathcal{I}_{N,d} \subseteq [N]^D$, function $\tilde{h}$ as per \eqref{def:htilde}, and point $\z_k \in \D_{\S^c}$.
Then,
\begin{equation}
\left| e^h_k \right| \leq \left\| \rt_{\Omega^{\rm opt}_{\tilde{h},s'}} \right\|_2 \tilde{\delta} \left \| \left(\Phi_{\S^c;\zero;\z_k}\right)_{\Omega^{\rm opt}_{\tilde{h},s'}} \right \|_2
\end{equation}
where $\Phi_{\S^c;\zero;\z_k}$ is defined as in \eqref{equ:RowRIPlem1}.
\label{lem:SupportLem2}
\end{lemma}

\begin{proof}
We begin by noting that
\begin{align}
\frac{1}{m_1} \sum_{\ell \in [m_1]} (\tilde{h}_{s'}^{\rm opt})_{\S^c;\z_k} (\w_\ell) ~\overline{T_{\S;\n}(\w_\ell)} ~=~ \sum_{\q \in \Omega^{\rm opt}_{\tilde{h},s'}} \tilde{r}_\q \, T_{\S^c;\q}(\z_k) ~\nu_{\S}\left( \q,\n \right)
\label{equ:ProofNewLem7_1}
\end{align}
where
\begin{align*}
\nu_{\S}\left( \q,\n \right) &:= \sum_{\ell \in [m_1]} \frac{1}{\sqrt{m_1}}T_{\S;\q}(\w_\ell) ~\frac{1}{\sqrt{m_1}}\overline{T_{\S;\n}(\w_\ell)}\\  
&= \left \langle \frac{1}{\sqrt{m_1}} \left( \Phi_{\S;\zero} \right)_{(\q,\zero)_{\S}}, \frac{1}{\sqrt{m_1}} \left( \Phi_{\S;\zero} \right)_{(\n,\zero)_{\S}} \right \rangle.
\end{align*}
Appealing to standard results concerning coherence in, e.g., Chapter 6 of \cite{foucart2013mathematical} one can see that
$ \left| \nu_{\S}\left( \q,\n \right) - 1 \right | \leq \tilde{\delta}$ holds if $\n_{\S} = \q_{\S}$, and that $ \left| \nu_{\S}\left( \q,\n \right) \right | \leq \tilde{\delta}$ holds if $\n_{\S} \neq \q_{\S}$.  

Let $\Omega' := \left\{ \q \in \Omega^{\rm opt}_{\tilde{h},s'} ~|~ \q_{\S} = \n_{\S} \right \} \subset \Omega^{\rm opt}_{\tilde{h},s'}$, and $\Omega'' := \Omega^{\rm opt}_{\tilde{h},s'} \setminus \Omega'$.  
Using \eqref{equ:ProofNewLem7_1} one has that 
\begin{align*}
\Bigg| \frac{1}{m_1} \sum_{\ell \in [m_1]} &(\tilde{h}_{s'}^{\rm opt})_{\S^c;\z_k} (\w_\ell) ~\overline{T_{\S;\n}(\w_\ell)} - \sum_{\q \in \Omega'} \tilde{r}_\q T_{\S^c;\q}(\z_k) \Bigg| \\ &= \left| \sum_{\q \in \Omega'} \tilde{r}_\q T_{\S^c;\q}(\z_k) \left( \nu_{\S}\left( \q,\n \right) - 1 \right) + \sum_{\q \in \Omega''} \tilde{r}_\q T_{\S^c;\q}(\z_k) \nu_{\S}\left( \q,\n \right) \right|\\ 
&\leq \left\| \rt_{\Omega^{\rm opt}_{\tilde{h},s'}} \right\|_2 \sqrt{ \sum_{\q \in \Omega'} \left | T_{\S^c;\q}(\z_k) \left( \nu_{\S}\left( \q,\n \right) - 1 \right) \right |^2 + \sum_{\q \in \Omega''} \left | T_{\S^c;\q}(\z_k) ~\nu_{\S}\left( \q,\n \right) \right |^2} 
\end{align*}
where the last inequality follows from Cauchy-Schwarz.  Continuing from this the last line we can further see that
\begin{align*}
\Bigg| \frac{1}{m_1} \sum_{\ell \in [m_1]} &(\tilde{h}_{s'}^{\rm opt})_{\S^c;\z_k} (\w_\ell) ~\overline{T_{\S;\n}(\w_\ell)} - \sum_{\q \in \Omega'} \tilde{r}_\q T_{\S^c;\q}(\z_k) \Bigg| \\ 
&\leq \left\| \rt_{\Omega^{\rm opt}_{\tilde{h},s'}} \right\|_2 \tilde{\delta} \sqrt{ \sum_{\q \in \Omega'} \left | T_{\S^c;\q}(\z_k) \right |^2 + \sum_{\q \in \Omega''} \left | T_{\S^c;\q}(\z_k) \right |^2} \\
&= \left\| \rt_{\Omega^{\rm opt}_{\tilde{h},s'}} \right\|_2 \tilde{\delta} \left \| \left(\Phi_{\S^c;\zero;\z_k}\right)_{\Omega^{\rm opt}_{\tilde{h},s'}} \right \|_2.
\end{align*}

To finish we note that
\begin{equation*}
\sum_{\q \in \Omega'} \tilde{r}_\q T_{\S^c;\q}(\z_k) ~= ~\left \langle ~ \left(\rt_{\Omega^{\rm opt}_{\tilde{h},s'}}\right)_{\S;\n}, \overline{\Phi_{\S^c;\n;\z_k}} ~\right \rangle ~=~ \left \langle \left(\tilde{h}_{s'}^{\rm opt}\right)_{\S^c;\z_k}, T_{\S;\n} \right \rangle_{\left( \D_{\S}, \mu_{\S}\right)}
\end{equation*}
by Lemma~\ref{lem:PartEvalinnerP}.  The desired result follows.
\end{proof}

Choose any $\n \in \mathcal{I}_{N,d}$ and $\S \subset [D]$ you like.  Using Lemma~\ref{lem:SupportLem2} to approximate the inner product appearing in Lemma~\ref{lem:EntryID_ExactInnerProd}'s \eqref{equ:EI_error_guar} then yields the following estimator for accurately approximating 
the $\ell^2$-norm of $\rpp := \rt_{\left\{ \q \in \Omega^{\rm opt}_{\tilde{h},s'} ~\big| ~ \q_{\S} = \n_{\S} \right\}} \in \C^{\mathcal{I}_{N,d}}$ for the coefficient vector $\rt$ of any function $\tilde{h}$ as in \eqref{def:htilde}.  The estimator is defined for any function $u: \D \rightarrow \C$, $\S \subset [D]$, and $\n \in \mathcal{I}_{N,d}$ to be
\begin{equation}
E^u_{\S; \n} := \frac{1}{m_2} \sum_{k\in [m_2]} \left| \frac{1}{m_1} \sum_{\ell \in [m_1]} u_{\S^c;\z_k} (\w_\ell) ~\overline{T_{\S;\n}(\w_\ell)}  \right|^2
\label{equ:EstimatorGen}
\end{equation}
for fixed nodes $\{ \w_\ell \}_{\ell \in [m_1]} \subset \D_{\S}$, and $\{ \z_k \}_{k \in [m_2]} \subset \D_{\S^c}$.  Note that \eqref{equ:EstimatorGen} is essentially identical to the pairing energy estimator defined in Section 4.2 of \cite{choi2018sparse}.  The following lemma provides an error guarantee for this estimator that matches the quality of those in \cite{choi2018sparse} despite having a simpler proof (see Lemma 7 in \cite{choi2018sparse}).\\

\begin{lemma}
Let $\S \subset [D]$, $\delta \in (0, 3/4]$, $\tilde{\delta} \in (0,1/s']$, and $m_1, m_2 \in \N$.  Furthermore, suppose that $\{ \w_\ell \}_{\ell \in [m_1]} \subset \D_{\S}$, and $\{ \z_k \}_{k \in [m_2]} \subset \D_{\S^c}$ satisfy the RIP assumptions concerning \eqref{equ:EI_RIPmat2} and \eqref{equ:EI_RIPmat} in Lemmas~\ref{lem:SupportLem2} and~\ref{lem:EntryID_ExactInnerProd}, respectively.
Then, for all $\n \in \mathcal{I}_{N,d} \subseteq [N]^D$ and functions $\tilde{h}$ as per \eqref{def:htilde} one will have 
$$\left| \sqrt{E^{\tilde{h}_{s'}^{\rm opt}}_{\S; \n}} - \| \rpp \|_2 \right|
\leq
\frac{2}{3}
\delta \| \rpp \|_2 + 
\sqrt{\frac{7}{4}} \sqrt{\tilde{\delta}} \cdot \left\| \rt_{\Omega^{\rm opt}_{\tilde{h},s'}} \right\|_2, $$
where $\rpp := \rt_{\left\{ \q \in \Omega^{\rm opt}_{\tilde{h},s'} ~\big| ~ \q_{\S} = \n_{\S} \right\}}$.
\label{lem:EstErrorforEstimator}
\end{lemma}

\begin{proof}
Applying Lemma~\ref{lem:EntryID_ExactInnerProd} we can immediately see that
\begin{equation}
\left| \sqrt{E^{\tilde{h}_{s'}^{\rm opt}}_{\S; \n}} - \| \rpp \|_2 \right| \leq
\frac{2}{3}
\delta
\| \rpp \|_2 +\frac{\| \e^h \|_2}{\sqrt{m_2}}
\label{equ:FinalNewlem7b}
\end{equation}
where $\e^h \in \C^{m_2}$ has its entries given by 
$$e^h_k := \frac{1}{m_1} \sum_{\ell \in [m_1]} (\tilde{h}_{s'}^{\rm opt})_{\S^c;\z_k} (\w_\ell) ~\overline{T_{\S;\n}(\w_\ell)} - \left \langle (\tilde{h}_{s'}^{\rm opt})_{\S^c;\z_k}, T_{\S;\n} \right \rangle_{\left( \D_{\S}, \mu_{\S}\right)}.$$
Thus, it suffices to bound $\| \e^h \|_2$ in order to obtain our final result.

Applying Lemma~\ref{lem:SupportLem2} we can see that
\begin{align*}
\| \e^h \|_2^2 &\leq \sum_{k\in [m_2]} \left\| \rt_{\Omega^{\rm opt}_{\tilde{h},s'}} \right\|_2^2 \tilde{\delta}^2 \left \| \left(\Phi_{\S^c;\zero;\z_k}\right)_{\Omega^{\rm opt}_{\tilde{h},s'}} \right \|_2^2\\
&= \left\| \rt_{\Omega^{\rm opt}_{\tilde{h},s'}} \right\|_2^2 \tilde{\delta}^2 \left \| \left(\Phi_{\S^c;\zero}\right)_{\Omega^{\rm opt}_{\tilde{h},s'}} \right \|_{\rm F}^2,
\end{align*}
where we have used that $\left(\Phi_{\S^c;\zero;\z_k}\right)_{\Omega^{\rm opt}_{\tilde{h},s'}}$ are the rows of the submatrix $\left( \Phi_{\S^c;\zero} \right)_{\Omega^{\rm opt}_{\tilde{h},s'}} \in \C^{m_2 \times s'}$ of $\Phi_{\S^c;\zero}$ in \eqref{equ:EI_RIPmat}.  Using the RIP property of the nonzero columns of $\frac{1}{\sqrt{m_2}} \Phi_{\S^c;\zero}$ we can now finish bounding $\| \e^h \|_2^2$ by noting that

\begin{align*}
\| \e^h \|_2^2 ~\leq~ \left\| \rt_{\Omega^{\rm opt}_{\tilde{h},s'}} \right\|_2^2 \tilde{\delta}^2 \cdot s'\cdot m_2 (1 + \delta) ~\leq~ 
\frac{7}{4}
 m_2 \left\| \rt_{\Omega^{\rm opt}_{\tilde{h},s'}} \right\|_2^2 \tilde{\delta} 
\end{align*}
where we have used that $\delta \in (0, 3/4]$ and that $\tilde{\delta} \in (0,1/s']$.  Substituting this last bound into \eqref{equ:FinalNewlem7b} now finishes the proof.
\end{proof}

Though useful, Lemma~\ref{lem:EstErrorforEstimator} presupposes that one has access to noiseless samples from $\tilde{h}_{s'}^{\rm opt}$.  This will rarely be the case in practice.  The next lemma bounds the error of the estimator \eqref{equ:EstimatorGen} in the setting where one instead has noisy samples from $\tilde{h}_{s'}^{\rm opt}$.  Such noisy samples will be represented with the help of an arbitrary additive noise/error function, $e_h: \D \rightarrow \C$, below.\\

\begin{lemma}
Let $\S \subset [D]$, $\delta \in (0, 3/4]$, $\tilde{\delta} \in (0,1/s']$, and $m_1, m_2 \in \N$.  Furthermore, suppose that $\{ \w_\ell \}_{\ell \in [m_1]} \subset \D_{\S}$, and $\{ \z_k \}_{k \in [m_2]} \subset \D_{\S^c}$ satisfy the RIP assumptions concerning \eqref{equ:EI_RIPmat2} and \eqref{equ:EI_RIPmat} in Lemmas~\ref{lem:SupportLem2} and~\ref{lem:EntryID_ExactInnerProd}, respectively.
Then, for all $\n \in \mathcal{I}_{N,d} \subseteq [N]^D$, $\tilde{h}$ as per \eqref{def:htilde}, and additive error functions $e_h: \D \rightarrow \C$ one will have
$$\left| \sqrt{E^{\tilde{h}_{s'}^{\rm opt}+e_h}_{\S; \n}} - \| \rpp \|_2 \right| \leq 
\frac{2}{3}
\delta \| \rpp \|_2 + 
\sqrt{\frac{7}{4} \tilde{\delta}} \cdot \left\| \rt_{\Omega^{\rm opt}_{\tilde{h},s'}} \right\|_2 + \frac{
\sqrt{2}\| \mathcal{E}^h_{\S} \|_{\rm F}}{\sqrt{m_1m_2}}, $$
where $\rpp := \rt_{\left\{ \q \in \Omega^{\rm opt}_{\tilde{h},s'} ~\big| ~ \q_{\S} = \n_{\S} \right\}}$ and $\mathcal{E}^h_{\S} \in \C^{m_1 \times m_2}$ has entries $(\mathcal{E}^h_{\S})_{\ell, k}=e_h(\varrho_{\mathcal{S}}(\w_\ell, \z_k))$ with the permutation function
$\varrho_\S\colon \D_{\S}\times\D_{\S^c}\rightarrow\D$ defined in 
\eqref{def:rho}.
\label{lem:EstErrorforEstimatorNoise}
\end{lemma}

\begin{proof}
Note that 
\begin{equation*}
\sqrt{E^{\tilde{h}_{s'}^{\rm opt}}_{\S; \n}}-\sqrt{E^{e_h}_{\S; \n}} \leq \sqrt{E^{\tilde{h}_{s'}^{\rm opt}+e_h}_{\S; \n}} \leq \sqrt{E^{\tilde{h}_{s'}^{\rm opt}}_{\S; \n}} + \sqrt{E^{e_h}_{\S; \n}}
\end{equation*}
by the (reverse) triangle inequality.  As a result one can immediately see that 
\begin{align*}
\left| \sqrt{E^{\tilde{h}_{s'}^{\rm opt}+e_h}_{\S; \n}} - \| \rpp \|_2 \right| &\leq \left| \sqrt{E^{\tilde{h}_{s'}^{\rm opt}+e_h}_{\S; \n}} -  \sqrt{E^{\tilde{h}_{s'}^{\rm opt}}_{\S; \n}}\right| + \left| \sqrt{E^{\tilde{h}_{s'}^{\rm opt}}_{\S; \n}} - \| \rpp \|_2 \right|\\
&\leq  \sqrt{E^{e_h}_{\S; \n}}  + 
\frac{2}{3}
\delta \| \rpp \|_2 +  \sqrt{ \frac{7}{4} \tilde{\delta}} \cdot \left\| \rt_{\Omega^{\rm opt}_{\tilde{h},s'}} \right\|_2
\end{align*}
where the bound on the second term above follows from Lemma~\ref{lem:EstErrorforEstimator}. It remains to show that $\sqrt{E^{e_h}_{\S; \n}} \leq \frac{\sqrt{2} \| \mathcal{E}^h_{\S} \|_{\rm F}}{\sqrt{m_1m_2}}$.  

Define $\v \in \C^{m_2}$ by $v_k:=\frac{1}{\sqrt{m_1}} \left\langle (\mathcal{E}^h_{\S})_k,  \frac{1}{\sqrt{m_1}} \left( \Phi_{\S;\zero} \right)_{(\n,\zero)_{\S}} \right\rangle$ where $\Phi_{\S;\zero}$ is defined in \eqref{equ:EI_RIPmat2}, and note that  $|v_k| \leq  \frac{1}{\sqrt{m_1}} \left\| (\mathcal{E}^h_{\S})_k \right\|_2 \left\|  \frac{1}{\sqrt{m_1}} \left( \Phi_{\S;\zero} \right)_{(\n,\zero)_{\S}} \right\|_2$. Furthermore, $$\left\|  \frac{1}{\sqrt{m_1}} \left( \Phi_{\S;\zero} \right)_{(\n,\zero)_{\S}} \right\|_2 = \left\|  \frac{1}{\sqrt{m_1}} \left( \Phi_{\S;\zero} \right)_{\left\{(\n,\zero)_{\S}, (\nt,\zero)_{\S} \right\}} \begin{bmatrix} 1\\0 \end{bmatrix} \right\|_2 \leq \sqrt{1+\widetilde{\delta}} 
\leq \sqrt{1+\frac{1}{s'}}
\leq \sqrt{2}$$ for any $\nt \neq \n$. Thus, $|v_k| \leq \sqrt{ \frac{2}{m_1} } \left\| (\mathcal{E}^h_{\S})_k \right\|_2$. As a result, $\sqrt{E^{e_h}_{\S; \n}} =\frac{1}{\sqrt{m_2}} \| \v \|_2 \leq \frac{\sqrt{2} \| \mathcal{E}^h_{\S} \|_{\rm F}}{\sqrt{m_1m_2}}$.
\end{proof}

For any given $\S \subseteq [D]$ we denote the power set of $N^{\S}$ by $\mathcal{P}\left(N^{\S} \right)$.  In the final theorem of this subsection we will prove that the energy estimator in \eqref{equ:EstimatorGen} can be used for an arbitrary $s'$-sparse function $\tilde{h}^{\rm opt}_{s'}$ to define a new set-valued function $\F^{s'}_{\S}: \mathcal{P}\left(N^{\S} \right) \rightarrow \mathcal{P}\left(N^{\S} \right)$ for each $\S$ which, when given any subset $\T \subset N^{\S}$ containing the heavy set $\Omega^{\alpha,{s'}}_{\S}$as per \eqref{def:HeavySet} as input, will output a smaller subset $\T' \subset \T$  which still contains $\Omega^{\alpha,{s'}}_{\S}$.  These set-valued functions were also called ``energetic-index sieve function'' in Section~\ref{sec:SuppIDHeavyEls} and will then be used to iteratively build up subsets $\T$ of controlled cardinality for larger and larger sets of indices $\S$ until we eventually have a set of full index vectors $\T'' \subset \mathcal{I}_{N,d}$ which contains all of $\Omega^{\alpha,{s'}}_{[D]}$.  This set of full index vectors $\T''$ will then be able to be used as an accurate estimate of $\Omega^{\rm opt}_{\tilde{h},s'}$, the support of $\tilde{h}^{\rm opt}_{s'}$.

Before we can state our final theorem we must define the set-valued functions $\F^{s'}_{\S}: \mathcal{P}\left(N^{\S} \right) \rightarrow \mathcal{P}\left(N^{\S} \right)$ in question. For a given $\S \subseteq [D]$, $\T \subseteq N^{\S}$, $\tilde{h}$ as per \eqref{def:htilde}, and additive error function $e_h: \D \rightarrow \C$, let an ordering of the elements of $\T$, $\n_1, \n_2, \cdots, \n_{|\T|} \in \T$, be defined by 
\begin{equation}
E^{\tilde{h}_{s'}^{\rm opt}+e_h}_{\S; \n_1} ~\geq~ E^{\tilde{h}_{s'}^{\rm opt}+e_h}_{\S; \n_2} ~\geq~ E^{\tilde{h}_{s'}^{\rm opt}+e_h}_{\S; \n_3} ~\geq~ \dots ~\geq~ E^{\tilde{h}_{s'}^{\rm opt}+e_h}_{\S; \n_{|\T|}}
\label{def:SizeOrder}
\end{equation}
with ties broken lexicographically.
We define $\F^{s'}_{\S}$ based on this ordering by 
\begin{equation}
\F^{s'}_{\S} (\T) := \left \{ \n_1, \n_2, \cdots, \n_{\min({s'},|\T|)} \right\} \subseteq \T.  
\label{def:SetValCardRedux}
\end{equation}
The following theorem proves that $\Omega^{\alpha,{s'}}_{\S} \cap \T \subseteq \F^{s'}_{\S} (\T)$ provided that the additive error~$e_h$ is sufficiently mild.\\

\begin{mytheorem}[{Entry Identification and Pairing}]
Let $\S \subseteq [D]$ with $|\S| > 0$, $\delta \in (0, 1/2]$, $\tilde{\delta} \in \left(0,\frac{1}{256\alpha^2s'} \right]$, and $m_1, m_2 \in \N$.  Furthermore, suppose that $\{ \w_\ell \}_{\ell \in [m_1]} \subset \D_{\S}$, and $\{ \z_k \}_{k \in [m_2]} \subset \D_{\S^c}$ satisfy the RIP assumptions concerning \eqref{equ:EI_RIPmat2} and \eqref{equ:EI_RIPmat} in Lemmas~\ref{lem:SupportLem2} and~\ref{lem:EntryID_ExactInnerProd}, respectively.  Then, $\Omega^{\alpha,s'}_{\S} \cap \T \subseteq \F^{s'}_{\S} (\T)$ for all $s'$-sparse $\tilde{h} = \tilde{h}^{\rm opt}_{s'}$, $\T \subseteq N^{\S}$, and additive error functions $e_h: \D \rightarrow \C$ provided that 
\begin{equation}
\| \rt \|_2 = \| \rt_{\Omega^{\rm opt}_{\tilde{h},s'}} \|_2 >  \frac{6 \alpha \sqrt{{s'}}}{\sqrt{m_1m_2}} \cdot \| \mathcal{E}^h_{\S} \|_{\rm F} \label{equ:threshold}
\end{equation}
holds, where 
$\mathcal{E}^h_{\S} \in \C^{m_1 \times m_2}$ has entries $(\mathcal{E}^h_{\S})_{\ell, k}=e_h(\varrho_\S (\w_\ell, \z_k))$ with the permutation function
$\varrho_\S\colon \D_{\S}\times\D_{\S^c}\rightarrow\D$ 
defined in \eqref{def:rho}.
\label{thm:pairingFunc}
\end{mytheorem}

\begin{proof}
We will focus on the case where $|\T| > s'$ since the result holds trivially when $|\T| \leq s'$.  Suppose for the sake of contradiction that $\m \in \Omega^{\alpha,s'}_{\S} \cap \T$, but that $\m \notin \F^{s'}_{\S} (\T)$.  It must then be the case that $E^{\tilde{h}_{s'}^{\rm opt}+e_h}_{\S; \k} ~\geq~ E^{\tilde{h}_{s'}^{\rm opt}+e_h}_{\S; \m}$ for some $\k \in \T$ with $\k_{\S} \notin \Omega^{\rm opt}_{s',\S} := \left\{ \q_{\S}  ~\big| ~ \q \in \Omega^{\rm opt}_{\tilde{h},s'} \right\}$ since $\Omega^{\alpha,s'}_{\S} \subset \Omega^{\rm opt}_{s',\S}$ and $| \F^{s'}_{\S} (\T) | = s' \geq |\Omega^{\rm opt}_{s',\S}|$.  Thus, $\| \rt_{\S;\k} \|_2 = 0$.  As a result, Lemma~\ref{lem:EstErrorforEstimatorNoise} implies that
\begin{equation}\label{equ:estE_upper_bound_epsilon}
\sqrt{E^{\tilde{h}_{s'}^{\rm opt}+e_h}_{\S; \k}} ~\leq~ \sqrt{ \frac{7}{4} \tilde{\delta}} \cdot \left\| \rt \right\|_2 + \frac{\sqrt{2} \| \mathcal{E}^h_{\S} \|_{\rm F}}{\sqrt{m_1m_2}} ~\leq~ \frac{\sqrt{7}}{32} \frac{\left\| \rt \right\|_2}{\alpha \sqrt{s'}} + \frac{\sqrt{2} \| \mathcal{E}^h_{\S} \|_{\rm F}}{\sqrt{m_1m_2}} =: \epsilon.
\end{equation}
On the other hand, Lemma~\ref{lem:EstErrorforEstimatorNoise} also implies that $\sqrt{E^{\tilde{h}_{s'}^{\rm opt}+e_h}_{\S; \m}} \geq \| \rt_{\S;\m} \|_2 (1 - \frac{2}{3}\delta) - \epsilon$.
Combining this with~\eqref{equ:estE_upper_bound_epsilon}, we have $\frac{2 \epsilon}{1 - \frac{2}{3}\delta} ~\geq~ \| \rt_{\S;\m} \|_2$.
Since $\delta\leq\frac{1}{2}$, it must also be the case that
$$3 \epsilon ~\geq~ \frac{2 \epsilon}{1 - \frac{2}{3}\delta} ~\geq~ \| \rt_{\S;\m} \|_2 ~\geq~ \frac{\| \rt \|_2}{\alpha \sqrt{s'}}.$$
However, it is impossible that $4 \alpha \sqrt{s'} \epsilon \geq \| \rt \|_2$ since by assumption
\begin{align*}
3 \alpha \sqrt{s'} \epsilon ~=~
3 \alpha \sqrt{s'} \left( \frac{\sqrt{7}}{32} \frac{\left\| \rt \right\|_2}{\alpha \sqrt{s'}} + \frac{ \sqrt{2}\| \mathcal{E}^h_{\S} \|_{\rm F}}{\sqrt{m_1m_2}} \right)
&< 3 \alpha \sqrt{s'} \left( \frac{\sqrt{7}}{32} \frac{\left\| \rt \right\|_2}{\alpha \sqrt{s'}} +  \frac{ \sqrt{2}\left\| \rt \right\|_2}{6 \alpha \sqrt{s'}} \right) < \left\| \rt \right\|_2.
\end{align*}
Hence, $\m \in \Omega^{\alpha,s'}_{\S} \cap \T \implies \m \in \F^{s'}_{\S} (\T)$.
\end{proof}

Theorem~\ref{thm:pairingFunc} forms the basis of our support identification strategy.  As such, it behooves us to investigate its associate resource demands and error performance more closely.  We do this in the next subsection.

\subsubsection{Associated Runtime, Sampling, and Error Bounds}
\label{sec:runtimeerrorsamp}

The following lemmas provide evaluation complexity, sampling, and error bounds for the set valued functions $\F^{s'}_{\S}: \mathcal{P}\left(N^{\S} \right) \rightarrow \mathcal{P}\left(N^{\S} \right)$ defined in \eqref{def:SizeOrder} -- \eqref{def:SetValCardRedux}.  We will begin by providing more meaningful error bounds for the case where the function $\tilde{h}$ in question is not exactly BOPB-sparse.\\

\begin{lemma}
Let $\S \subseteq [D]$ with $|\S| > 0$, $\delta \in (0, 1/2]$, $\tilde{\delta} \in \left(0,\frac{1}{256\alpha^2s'} \right]$, $\gamma \in \mathbbm{R}^+$, and $m_1, m_2 \in \N$.  Furthermore, suppose that $\{ \w_\ell \}_{\ell \in [m_1]} \subset \D_{\S}$, and $\{ \z_k \}_{k \in [m_2]} \subset \D_{\S^c}$ satisfy the RIP assumptions concerning \eqref{equ:EI_RIPmat2} and \eqref{equ:EI_RIPmat} in Lemmas~\ref{lem:SupportLem2} and~\ref{lem:EntryID_ExactInnerProd}, respectively.  Finally, suppose also that $e_h = \tilde{h} - \tilde{h}^{\rm opt}_{s'} + e'$ for an arbitrary function $e': \D \rightarrow \C$ with $\displaystyle \sup_{\xib \in \D} |e'(\xib)| \leq \gamma$.  Then, the additive sampling error $\mathcal{E}^h_{\S} \in \C^{m_1 \times m_2}$ satisfies
$$\frac{\| \mathcal{E}^h_{\S} \|_{\rm F}}{\sqrt{m_1 m_2}} ~\leq~ 
\sqrt{\frac{771}{512}} \left\| \rt - \rt_{\Omega^{\rm opt}_{\tilde{h},s'}} \right\|_2 + 
\sqrt{\frac{771}{512\,s'}} \left\| \rt - \rt_{\Omega^{\rm opt}_{\tilde{h},s'}} \right\|_1 + \gamma$$
where $\mathcal{E}^h_{\S}$ has entries $(\mathcal{E}^h_{\S})_{\ell, k}=e_h\left(\varrho_\S(\w_\ell, \z_k)\right)$ as in Theorem~\ref{thm:pairingFunc}.
\label{lem:errorBound}
\end{lemma}

\begin{proof}
Note that
\begin{align}
\frac{\| \mathcal{E}^h_{\S} \|_{\rm F}}{\sqrt{m_1 m_2}} &\leq \gamma + \frac{1}{\sqrt{m_1 m_2}} \sqrt{\sum_{\ell,k} \left| \left( \tilde{h} - \tilde{h}^{\rm opt}_{s'} \right) \left(\varrho_{\mathcal{S}}( \w_\ell, \z_k)\right) \right|^2} \nonumber \\
&= \gamma  + \frac{1}{\sqrt{m_1 m_2}} \sqrt{\sum_{\ell,k} \left| \left\langle \rt - \rt_{\Omega^{\rm opt}_{\tilde{h},s'}}, {\Phi}^*_{(\ell,k)}\right\rangle \right|^2} \nonumber\\
&=\left\| \frac{1}{\sqrt{m_1 m_2}} {\Phi} \left( \rt - \rt_{\Omega^{\rm opt}_{\tilde{h},s'}} \right) \right\|_2 + \gamma \label{equ:energybound},
\end{align}
where ${\Phi} \in \C^{m_1 m_2 \times \left| \mathcal{I}_{N,d} \right|}$ has entries given by ${\Phi}_{(\ell,k),\n} = T_{\n}(\varrho_{\mathcal{S}}(\w_\ell, \z_k))$.  Note also that $\frac{1}{\sqrt{m_1 m_2}} {\Phi}$ consists of a subset of the columns of the Kronecker product $\left( \frac{1}{\sqrt{m_1}} {\Phi}_{\S;\zero} \right) \otimes \left( \frac{1}{\sqrt{m_2}} {\Phi}_{\S^c;\zero} \right)$
where ${\Phi}_{\S;\zero} \in \C^{m_1 \times \left| \mathcal{I}_{N,d} \right|}$ is defined in \eqref{equ:EI_RIPmat2}, and ${\Phi}_{\S^c;\zero} \in \C^{m_2 \times \left|  \mathcal{I}_{N,d} \right|}$ is defined in \eqref{equ:EI_RIPmat}.  Furthermore, Proposition 6.6 of \cite{foucart2013mathematical} implies that the nonzero columns of $\frac{1}{\sqrt{m_1}} {\Phi}_{\S;\zero}$ also has the RIP of order $(s',\frac{1}{256\alpha^2})$ since it has the RIP of order $(2,\frac{1}{256\alpha^2s'})$.  Hence, $\frac{1}{\sqrt{m_1 m_2}} {\Phi}$ has the RIP of order $\left(s',\left(1+\frac{1}{256\alpha^2}\right)\left(1+\frac{1}{2}\right)-1\right)$ by Lemma 2 of \cite{duarte2012kronecker}, and consequently of order $\left(s',\frac{259}{512}\right)$ for $\alpha\geq 1$.
Returning to \eqref{equ:energybound}, we can now use 
Lemma~\ref{lem:OpBoundRSM} to see that 
$$\frac{\| \mathcal{E}^h_{\S} \|_{\rm F}}{\sqrt{m_1 m_2}} ~\leq~ \sqrt{\frac{771}{512}} \left\| \rt - \rt_{\Omega^{\rm opt}_{\tilde{h},s'}} \right\|_2 + \sqrt{\frac{771}{512s'}} \left\| \rt - \rt_{\Omega^{\rm opt}_{\tilde{h},s'}} \right\|_1 + \gamma$$
as desired.
\end{proof}

The next lemma tells us how many evaluation points we need to randomly generate in Lemmas~\ref{lem:SupportLem2} and~\ref{lem:EntryID_ExactInnerProd} before we can be sure to have the RIP properties required by both Theorem~\ref{thm:pairingFunc} and Lemma~\ref{lem:errorBound} above hold with high probability.\\

\begin{lemma}
Let $\S \subseteq [D]$, $\delta \in (0, 1/2]$, and $\tilde{\delta} = \left(0,\frac{1}{256\alpha^2s'}\right]$.
Furthermore, suppose that $m_1,m_2,s',N,D \in \mathbbm{Z}^+ \setminus \{1\}$, $d \in \mathbbm{Z} \cap [1, D]$, and $p \in (0,1)$ satisfy
$$m_1 \geq a_1 \alpha^4 K_{\S}^2  (s')^2 \cdot \max \left\{ d \ln \left( \frac{DN}{d} \right) \ln(m_1), \ln \left(p^{-1} \right) \right\},$$
and
$$m_2 \geq a_2 K_{\S^c}^2 \delta^{-2} s' \cdot \max \left\{ d \ln^2(s') \ln \left( \frac{DN}{d} \right) \ln(m_2), \ln \left(p^{-1} \right) \right\},$$
where $a_1,a_2 \in \mathbbm{R}^+$ are universal constants.
Then, the samples $\{ \w_\ell \}_{\ell \in [m_1]} \subset \D_{\S}$ and $\{ \z_k \}_{k \in [m_2]} \subset \D_{\S^c}$ will both simultaneously satisfy their respective RIP assumptions concerning \eqref{equ:EI_RIPmat2} and \eqref{equ:EI_RIPmat} in Lemmas~\ref{lem:SupportLem2} and~\ref{lem:EntryID_ExactInnerProd} above with probability at least $1 - p$.
\label{lem:SamplingBounds}
\end{lemma}

\begin{proof}

The bounds on both $m_1$ and $m_2$ follow from applications of Theorem~\ref{thm:BOS_RIP}.  To bound $m_1$ we note that the normalized nonzero columns of \eqref{equ:EI_RIPmat2} need to have the RIP of order $(2,\tilde{\delta})$, and have an associated BOS constant of $K_{\S}$.  Furthermore, there will never be more than 
\begin{equation}
\left|\mathcal{I}_{N,d} \right|  = {D \choose d} N^d \leq \left( \frac{\mathbbm{e}D}{d} \right)^d N^d = \left( \frac{\mathbbm{e}DN}{d} \right)^d
\label{equ:IdNCardBound}
\end{equation}
nonzero columns of \eqref{equ:EI_RIPmat2} for any choice of $\S \subseteq [D]$.  As a consequence we can see that it suffices to have
\begin{align*}
m_1 \geq&~ a'_1 K_{\S}^2 \tilde{\delta}^{-2} \cdot \max \left\{ d \ln \left( \frac{DN}{d} \right) \ln(m_1), \ln \left(p^{-1} \right) \right\}\\
\geq&~ a_1 \alpha^4 K_{\S}^2  (s')^2 \cdot \max \left\{ d \ln \left( \frac{DN}{d} \right) \ln(m_1), \ln \left(p^{-1} \right) \right\}
\end{align*}
in order to satisfy the required RIP conditions for \eqref{equ:EI_RIPmat2} with probability at least $1 - p/2$.

To bound $m_2$ we note that the normalized nonzero columns of \eqref{equ:EI_RIPmat} need to have the RIP of order $(s',\delta)$, and have an associated BOS constant of $K_{\S^c}$.  As a result, \eqref{equ:IdNCardBound} together with Theorem~\ref{thm:BOS_RIP} implies that it suffices to have
$$m_2 \geq a_2 K_{\S^c}^2 \delta^{-2} s' \cdot \max \left\{ d \ln^2(s') \ln \left( \frac{DN}{d} \right) \ln(m_2), \ln \left(p^{-1} \right) \right\}$$ 
in order to satisfy the required RIP conditions for \eqref{equ:EI_RIPmat} with probability at least $1 - p/2$.
The final desired probability of success now results from the union bound.
\end{proof}

\begin{remark}
\label{rem:SampBound}
To simplify the appearance of our bounds from Lemma~\ref{lem:SamplingBounds} we will make use of the following additional facts and mild assumptions.  First, we will assume hereafter that both $m_1$ and $m_2$ are less than $\left|\mathcal{I}_{N,d} \right|$.  We consider this a reasonable assumption given that the techniques presented herein should only be used in situations where this is the case.  Furthermore, we will use $\delta = 1/2$ above as this is its largest valid parameter setting, and will also consider $\alpha$ to be a universal constant given that it is ultimately set to a fixed value. Finally, we will also replace our probability of failure parameter $p$ by $c / 2D$ for some small constant $c < 0.01$ (for example) in anticipation of wanting to survive a union bound involving $2D-1$ applications of Lemma~\ref{lem:SamplingBounds} for $2D-1$ different sets of random samples.  This will allow us to assert that any at most $2D-1$ different set valued functions $\F^{s'}_{\S}$ will all simultaneously satisfy both Theorem~\ref{thm:pairingFunc} and Lemma~\ref{lem:errorBound} with a ``high probability'' of at least $0.99$.  Utilizing these simplifications we obtain the simplified sufficient sampling conditions
\begin{align*}
m_1 &\geq~ c'_1 K_{\S}^2  (s')^2 \cdot d^2 \ln^2 \left( \frac{DN}{d} \right) \ln \left(D \right),\\
m_2 &\geq~ c'_2 K_{\S^c}^2 s' \cdot d^2 \ln^2(s') \ln^2 \left( \frac{DN}{d} \right) \ln \left(D \right)
\end{align*}
for new absolute constants $c'_1, c'_2 \in \mathbbm{R}^+$.

Finally, and perhaps most controversially, we will make the additional assumption above that either $(i)$ the BOS constants $K_j$ are $1$ for all but at most $\tilde{d} \in \mathbbm{Z} \cap [0, D]$ BOS basis sets $\B_j$ (note that $\tilde{d}$ can be independent of $d$), or else that $(ii)$ $K_0 = 1$.  In either case we will have that both $K_{\S}$ and $K_{\S^c}$ will be bounded above by a constant that depends on only $\tilde{d}$ or $d$, respectively, for all $\S \subseteq [D]$.  In particular, in case $(i)$ we will have that $K_{\S}$ and $K_{\S^c}$ are both at most $K^{\tilde{d}}_\infty$, and in case $(ii)$ that $K_{\S}$ and $K_{\S^c}$ are both at most $K^{d}_\infty$.  Utilizing this final assumption now allows us to bound the total number of samples we need in order to compute any $2D-1$ informative $\F^{s'}_{\S}$ sets for any given $2D-1$ sets $\S$ with high probability (w.h.p.) by either
$$m_1 m_2 \geq c_1 K^{4\tilde{d}}_\infty (s')^3 d^4 \cdot  \ln^4 \left( \frac{DN}{d} \right) \ln^2 (s') \ln^2 (D)$$
in case $(i)$ (note here that letting $d = D$ still avoids exponential dependence on $D$ in this setting), and by 
$$m_1 m_2 \geq c_2 K^{4d}_\infty (s')^3 d^4 \cdot  \ln^4 \left( \frac{DN}{d} \right) \ln^2 (s') \ln^2 (D)$$
in case $(ii)$, where $c_1,c_2 \in \mathbbm{R}^+$ are absolute constants.
\end{remark}

We are now ready to demonstrate the numerical performance of our proposed method.

\section{Empirical Evaluation}
\label{sec:Numerics}
  In this section, Algorithm~\ref{alg:main} in combination with Algorithm~\ref{alg:suppid:impl} is evaluated numerically for the exactly sparse case with noisy measurements as well as the approximately sparse case. The algorithms were implemented in MATLAB and are publicly available.\footnote{See ``SHT II: Best $s$-Term Approximation Guarantees for Bounded Orthonormal Product Bases in Sublinear-Time'' on Mark Iwen's code page \url{https://www.math.msu.edu/~markiwen/Code.html}.} For the entry identification, we use the pairing approach.
  In addition to the stopping criterion ``$\|\v_{\rm CE} \|_2^2 > \|\v_{\rm CEold}\|_2^2 $ or $k \geq \kappa$'' in line~16 in Algorithm~\ref{alg:main},
    we also stop Algorithm~\ref{alg:main} if ${\rm supp}({\a}^{k})={\rm supp}({\a}^{k-1})={\rm supp}({\a}^{k-2})$, i.e., the identified index vectors are the same for three consecutive iterations.
All time measurements were performed on a computer with 2 x 6-core Intel Xeon CPU E5-2620 v3 (2.40GHz), 64 GB RAM, using 12 threads.

  \subsection{Exactly sparse case and noisy measurements}\label{sec:Numerics:exactly_sparse}

  We start with the exactly sparse case, and we consider tensor product basis functions with different bases in $d=D$ spatial dimensions, where we choose $d$ up to 100. We set $N=200$ and use $\mathcal{I}_{N,d}=\mathcal{I}_{200,d}$ as search space of possible basis indices, where e.g. $|\mathcal{I}_{200,50}|\approx 10^{115}$ and $|\mathcal{I}_{200,100}|\approx 10^{230}$. We set the maximum number of iterations $\kappa:=20$, and we always use $m_{\rm CE}:=50\,s$ samples for the coefficient estimation where $s=|\S|$.
  For every data point in every plot below, we use 100 different randomly generated trial signals
  \begin{equation} \label{equ:trial_signal_f_sparse}
   f(\xib) = \sum_{\n\in\S} c_{\n} \, T_{\n}(\xib),
  \end{equation}
  where we draw the function's support set $\S\subset\mathcal{I}_{N,d}$ uniformly at random without repetition and the coefficients $c_{\n}\in\{-1,1\}$ uniformly at random.

  Below, a trial will always refer to the execution of Algorithm~\ref{alg:main} on a particular randomly generated trial function $f$ as defined in~\eqref{equ:trial_signal_f_sparse}. A failed trial will refer to any trial where Algorithm~\ref{alg:main} failed to recover the correct support set $\S$ for $f$.
  
  We assume that the function evaluations of $f$ are contaminated with (white) Gaussian noise, i.e., we provide Algorithm~\ref{alg:main} with noisy samples
  $$
   \boldsymbol{y'} = \y + \boldsymbol{g'} = \y + \sigma \frac{\|\y\|_2}{\|\boldsymbol{g}\|_2} \boldsymbol{g},
  $$
  where $\y$ contains noiseless samples from $f$, $\boldsymbol{g}\sim \mathcal{N}(\zero,I)$, and $\sigma\in\mathbbm{R}^+$ is used to control the Signal to Noise Ratio (SNR) defined herein by
  $$
   \mathrm{SNR}_{\mathrm{db}} := 10 \, \log_{10} \left(\frac{\|\y\|_2^2}{\|\boldsymbol{g'}\|_2^2}\right) = -10 \, \log_{10} (\sigma^2).
 $$
 
In the following subsections, we consider different types of basis functions. First, in Section~\ref{sec:Numerics:exactly_sparse:mixed1}, we use mixed bases in up to 100 spatial dimensions, which consist of Fourier, Chebyshev, and Legendre bases. Afterwards, we use bases which only consist of Fourier type in Section~\ref{sec:Numerics:exactly_sparse:fourier}, Chebyshev type in Section~\ref{sec:Numerics:exactly_sparse:chebyshev}, and Legendre type in Section~\ref{sec:Numerics:exactly_sparse:legendre}.

\subsubsection{Mixed bases}\label{sec:Numerics:exactly_sparse:mixed1}

First, we consider basis functions $T_{n_j}$ of mixed type: $T_{n_0}$, $T_{n_{\lfloor d/2 \rfloor}}$, and $T_{n_{d-2}}$ are of Chebyshev type; $T_{n_1}$, $T_{n_{\lfloor d/2 \rfloor - 1}}$, and $T_{n_{d-1}}$ are of (preconditioned) Legendre type; and the remaining $d-6$ basis functions $T_{n_j}$, $j\in [d] \setminus\{0,1,\lfloor d/2 \rfloor - 1,\lfloor d/2 \rfloor,d-2,d-1\}$ are of Fourier type.
Preconditioned Legendre type means that instead of using standard Legendre polynomials $L_n(x)=(2n-1)/n \, x \, L_{n-1}(x) - (n-1)/n \, L_{n-2}(x)$, $L_1:=x$, $L_0:=1$, with BOS constant $K=\sqrt{2n+1}$, we apply the preconditioning method from~\cite{R12}, i.e., we use the preconditioned Legendre polynomials $Q_n(x) := \sqrt{\pi/2}\,(1-x^2)^{1/4}\, L_n(x)$ with BOS constant $K=\sqrt{3}$ and choose the sampling nodes randomly with respect to the Chebyshev measure for the basis functions $T_{n_1}$, $T_{n_{\lfloor d/2 \rfloor - 1}}$, and $T_{n_{d-1}}$. Consequently, the overall BOS constant is $\sqrt{2}^3\cdot\sqrt{3}^3\cdot 1^{d-6}=\sqrt{6}^3$ independent of the spatial dimension~$d=D$.
For the entry identification and pairing steps, we set the parameter $m_2=\#\z_{j,k}:=4\,s$ for different sparsities $s=|\S|$.
The parameter $m_1=\#\w_{j,\ell}$ is chosen as $c \, s$, where the constant $c\geq 1$ does not depend on the sparsity~$s$, which is distinctly smaller than the theoretical results of $m_1 \sim s^2$ in Theorem~\ref{thm:SuppIDWorks} and Lemma~\ref{lem:SamplingBounds}. 

\begin{figure}[!h]
\subfloat[number of samples vs.\ spatial dimension $d$]{
	\label{fig:mixed1:d100:samples_runtime_iter_success_vs_d:samples}
		\begin{tikzpicture}[baseline]
		\begin{axis}[
			font=\footnotesize,
			enlarge x limits=true,
			enlarge y limits=true,
			height=0.3\textwidth,
			grid=major,
			width=0.45\textwidth,
			xtick={6,10,25,50,100},
            ymin=0.6e4,
            xmode=log,
            ymode=log,
      xticklabel={
        \pgfkeys{/pgf/fpu=true}
        \pgfmathparse{exp(\tick)}%
        \pgfmathprintnumber[fixed relative, precision=3]{\pgfmathresult}
        \pgfkeys{/pgf/fpu=false}
      },
   			xlabel={$d$},
			ylabel={\#samples},
			legend style={legend cell align=left}, legend pos=south east,
			legend columns = 2,
		]
		\addplot[black,mark=square,mark size=2.5pt,mark options={solid},only marks] coordinates {
(6,3.570e+04) (10,6.130e+04) (25,1.573e+05) (50,3.173e+05) (100,6.373e+05)
};
\addlegendentry{$s=10$}
\addplot [black,domain=6:100, samples=100, dashed]{6.1e3*x};
\addlegendentry{$6100\,d$}
		\addplot[blue,mark=triangle,mark size=2.5pt,mark options={solid},only marks] coordinates {
(6,2.212e+05) (10,3.812e+05) (25,9.812e+05) (50,1.981e+06) (100,3.981e+06)
};
\addlegendentry{$s=25$}
\addplot [blue,domain=6:100, samples=100, dashed]{4e4*x};
\addlegendentry{$4\cdot 10^4 d$}
		\end{axis}
		\end{tikzpicture}
}
\hfill
\subfloat[average runtime vs.\ spatial dimension $d$]{
	\label{fig:mixed1:d100:samples_runtime_iter_success_vs_d:runtime}
		\begin{tikzpicture}[baseline]
		\begin{axis}[
			font=\footnotesize,
			enlarge x limits=true,
			enlarge y limits=true,
			height=0.30\textwidth,
			grid=major,
			width=0.48\textwidth,
			xtick={6,10,25,50,100},
            ymin=0.2,
            xmode=log,
            ymode=log,
            log ticks with fixed point,
			xlabel={$d$},
			ylabel={avg. runtime in seconds},
		y label style={xshift=-0.8em},
			legend style={legend cell align=left}, legend pos=south east,
			legend columns = 2,
		]
		\addplot[black,mark=square,mark size=2.5pt,mark options={solid},only marks] coordinates {
(6,2.428e+00) (10,4.067e+00) (25,1.140e+01) (50,2.922e+01) (100,8.350e+01)
};
\addlegendentry{$s=10$}
\addplot [black,domain=6:100, samples=100, dashed]{5 + 0.008*x*x};
\addlegendentry{$0.008d^2\!+\!5$}
		\addplot[blue,mark=triangle,mark size=2.5pt,mark options={solid},only marks] coordinates {
(6,1.136e+01) (10,1.797e+01) (25,5.903e+01) (50,1.561e+02) (100,4.912e+02)
};
\addlegendentry{$s=25$}
\addplot [blue,domain=6:100, samples=100, dashed]{26 + 0.045*x*x};
\addlegendentry{$0.045d^2\!+\!26$}
		\end{axis}
		\end{tikzpicture}
}
\\
\subfloat[average iteration vs.\ spatial dimension $d$]{
	\label{fig:mixed1:d100:samples_runtime_iter_success_vs_d:iterations}
		\begin{tikzpicture}[baseline]
		\begin{axis}[
			font=\footnotesize,
			enlarge x limits=true,
			enlarge y limits=true,
			height=0.25\textwidth,
			grid=major,
			width=0.48\textwidth,
			xtick={6,10,25,50,100},
			ytick={1,2,3,4,5,6},
			ymin=1,ymax=5.7,
            xmode=log,
            log ticks with fixed point,
			xlabel={$d$},
			ylabel={avg. iteration},
			legend style={legend cell align=left}, legend pos=south east,
			legend columns = -1,
		]
		\addplot[black,mark=square,mark size=2.5pt,mark options={solid}] coordinates {
(6,4.190e+00) (10,4.070e+00) (25,4.230e+00) (50,4.730e+00) (100,5.380e+00)
};
\addlegendentry{$s=10$}
		\addplot[blue,mark=triangle,mark size=2.5pt,mark options={solid}] coordinates {
(6,4.280e+00) (10,3.630e+00) (25,3.720e+00) (50,3.650e+00) (100,3.730e+00)
};
\addlegendentry{$s=25$}
		\end{axis}
		\end{tikzpicture}
}
\hfill
\subfloat[success rate vs.\ spatial dimension~$d$]{
	\label{fig:mixed1:d100:samples_runtime_iter_success_vs_d:success_rate}
		\begin{tikzpicture}[baseline]
		\begin{axis}[
			font=\footnotesize,
			enlarge x limits=true,
			enlarge y limits=true,
			height=0.25\textwidth,
			grid=major,
			width=0.48\textwidth,
			xtick={6,10,25,50,100},
			ytick={0,0.5,0.8,1},
			ymin=0.0,ymax=1,
            xmode=log,
            log ticks with fixed point,
			xlabel={$d$},
			ylabel={success rate},
			legend style={legend cell align=left}, legend pos=south east,
			legend columns = -1,
		]
		\addplot[black,mark=square,mark size=2.5pt,mark options={solid}] coordinates {
(6,1.000e+00) (10,1.000e+00) (25,9.900e-01) (50,9.900e-01) (100,9.900e-01)
};
\addlegendentry{$s=10$}
		\addplot[blue,mark=triangle,mark size=2.5pt,mark options={solid}] coordinates {
(6,1.000e+00) (10,1.000e+00) (25,1.000e+00) (50,1.000e+00) (100,1.000e+00)
};
\addlegendentry{$s=25$}
		\end{axis}
		\end{tikzpicture}
}
\caption{Number of samples, runtime, number of iterations, success rates
 vs.\ spatial dimension $d=D\in\{6,10,25,50,100\}$ for mixed bases ($3$~Chebyshev, $3$~Legendre, $d-6$ Fourier), $N=200$, sparsity $s\in\{10,25\}$, $\mathrm{SNR}_{\mathrm{db}}=10$, $m_1=8s$, $m_2=4s$.}
\label{fig:mixed1:d100:samples_runtime_iter_success_vs_d}
\end{figure}
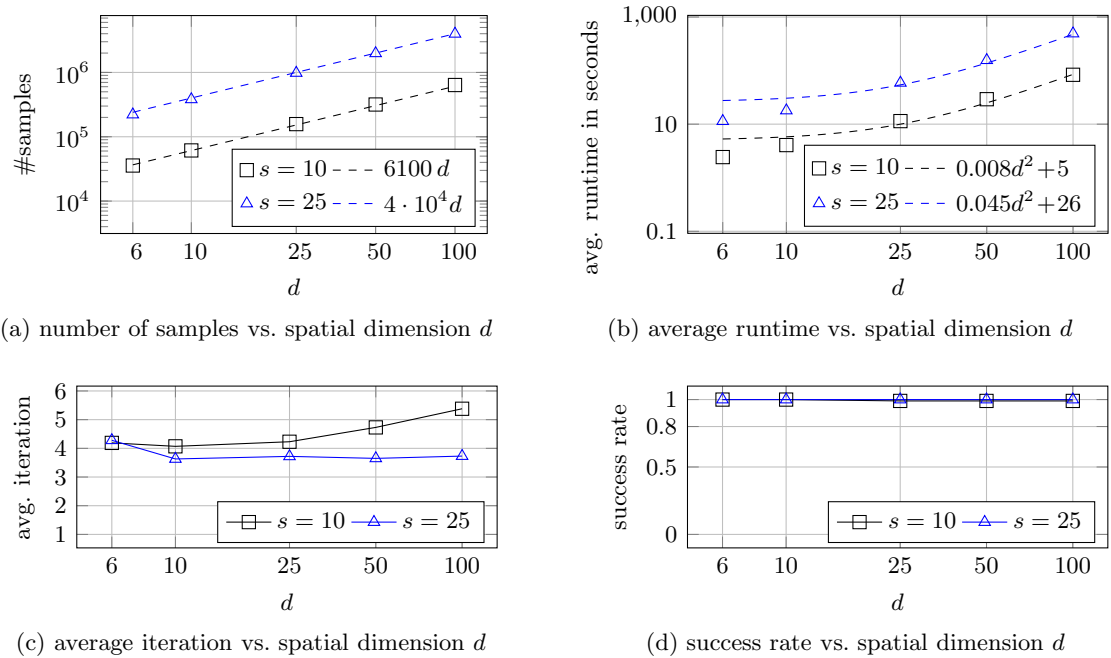

In Figure~\ref{fig:mixed1:d100:samples_runtime_iter_success_vs_d}, we visualize the obtained results in dependence of the spatial dimensions $d=D\in\{6,10,25,50,100\}$ for sparsity $s\in\{10,25\}$ and signal to noise ratio $\mathrm{SNR}_{\mathrm{db}}=10$.
In Figure~\ref{fig:mixed1:d100:samples_runtime_iter_success_vs_d:samples}, we plot the number of samples with respect to the spatial dimension $d$. We observe that the number of samples grows nearly linearly in~$d$. Additionally, we plot the average runtime of the 100 test runs with respect to $d$ in Figure~\ref{fig:mixed1:d100:samples_runtime_iter_success_vs_d:runtime}, and we observe that it grows approximately like~$\sim d^2$. When having a look at the average number of iterations in Figure~\ref{fig:mixed1:d100:samples_runtime_iter_success_vs_d:iterations}, we observe that 4.1 to 5.4 iterations were required for sparsity $s=10$ and around 4 iterations for sparsity $s=25$. For the considered test setting, the observed success rate was 100\% for sparsity $s=25$ and at least 99\% for sparsity $s=10$, cf.\ Figure~\ref{fig:mixed1:d100:samples_runtime_iter_success_vs_d:success_rate}.

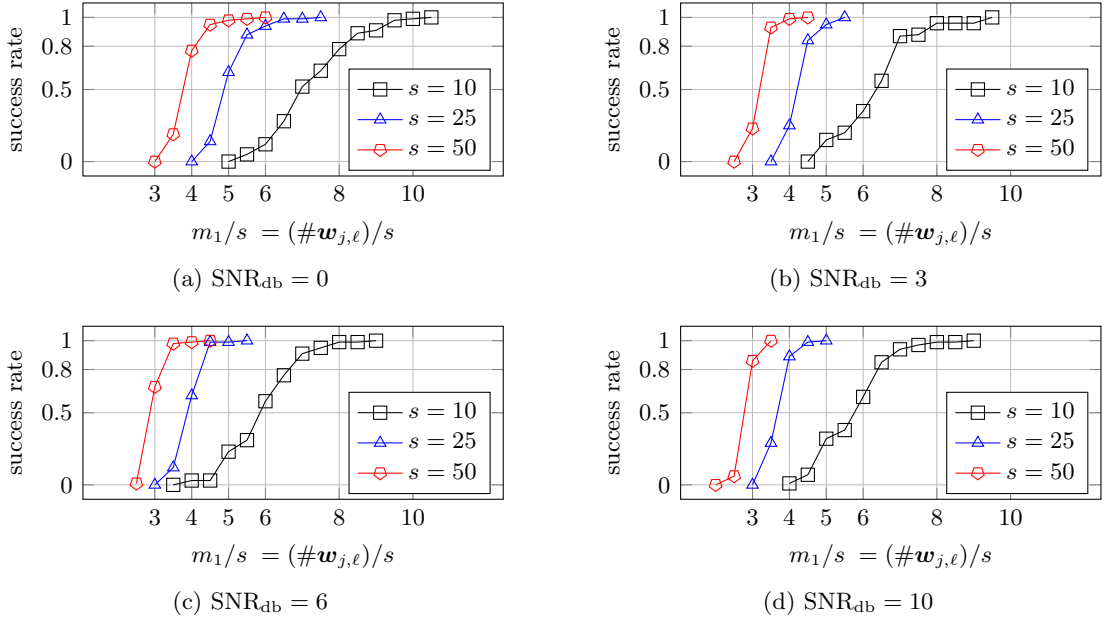
\begin{figure}[htbp]
\subfloat[$\mathrm{SNR}_{\mathrm{db}}=0$]{
	\label{fig:mixed1:d50:success_vs_m1s:snr0}
\begin{tikzpicture}[baseline]
		\begin{axis}[
			font=\footnotesize,
			enlarge x limits=true,
			enlarge y limits=true,
			height=0.26\textwidth,
			grid=major,
			width=0.48\textwidth,
			xtick={3,4,5,6,8,10},
			ytick={0,0.5,0.8,1},
			xmin=2,xmax=11.5,
			ymin=0,ymax=1,
			xlabel={$m_1/s\;=\;$(\#$\w_{j,\ell})/s$},
			ylabel={success rate},
			legend style={legend cell align=left}, legend pos=south east,
			legend columns = 1,
		]
		\addplot[black,mark=square,mark size=2.5pt,mark options={solid}] coordinates {
(5.0,0.00) (5.5,0.05) (6.0,0.12) (6.5,0.28) (7.0,0.52) (7.5,0.63) (8.0,0.78) (8.5,0.89) (9.0,0.91) (9.5,0.98) (10.0,0.99) (10.5,1.00)
};
\addlegendentry{$s=10$}
		\addplot[blue,mark=triangle,mark size=2.5pt,mark options={solid}] coordinates {
(4.0,0.00) (4.5,0.14) (5.0,0.62) (5.5,0.88) (6.0,0.94) (6.5,0.99) (7.0,0.99) (7.5,1.00)
};
\addlegendentry{$s=25$}
		\addplot[red,mark=pentagon,mark size=2.5pt,mark options={solid,rotate=180}] coordinates {
(3.0,0.00) (3.5,0.19) (4.0,0.77) (4.5,0.95) (5.0,0.98) (5.5,0.99) (6.0,1.00)
};
\addlegendentry{$s=50$}
		\end{axis}
		\end{tikzpicture}
}
\hfill
\subfloat[$\mathrm{SNR}_{\mathrm{db}}=3$]{
	\label{fig:mixed1:d50:success_vs_m1s:snr3}
		\begin{tikzpicture}[baseline]
		\begin{axis}[
			font=\footnotesize,
			enlarge x limits=true,
			enlarge y limits=true,
			height=0.26\textwidth,
			grid=major,
			width=0.48\textwidth,
			xtick={3,4,5,6,8,10},
			ytick={0,0.5,0.8,1},
			xmin=2,xmax=11.5,
			ymin=0,ymax=1,
			xlabel={$m_1/s\;=\;$(\#$\w_{j,\ell})/s$},
			ylabel={success rate},
			legend style={legend cell align=left}, legend pos=south east,
			legend columns = 1,
		]
		\addplot[black,mark=square,mark size=2.5pt,mark options={solid}] coordinates {
(4.5,0.00) (5.0,0.15) (5.5,0.20) (6.0,0.35) (6.5,0.56) (7.0,0.87) (7.5,0.88) (8.0,0.96) (8.5,0.96) (9.0,0.96) (9.5,1.00)
};
\addlegendentry{$s=10$}
		\addplot[blue,mark=triangle,mark size=2.5pt,mark options={solid}] coordinates {
(3.5,0.00) (4.0,0.25) (4.5,0.84) (5.0,0.95) (5.5,1.00)
};
\addlegendentry{$s=25$}
		\addplot[red,mark=pentagon,mark size=2.5pt,mark options={solid,rotate=180}] coordinates {
(2.5,0.00) (3.0,0.23) (3.5,0.93) (4.0,0.99) (4.5,1.00)
};
\addlegendentry{$s=50$}
		\end{axis}
		\end{tikzpicture}
}
\\
\subfloat[$\mathrm{SNR}_{\mathrm{db}}=6$]{
	\label{fig:mixed1:d50:success_vs_m1s:snr6}
		\begin{tikzpicture}[baseline]
		\begin{axis}[
			font=\footnotesize,
			enlarge x limits=true,
			enlarge y limits=true,
			height=0.26\textwidth,
			grid=major,
			width=0.48\textwidth,
			xtick={3,4,5,6,8,10},
			ytick={0,0.5,0.8,1},
			xmin=2,xmax=11.5,
			ymin=0,ymax=1,
			xlabel={$m_1/s\;=\;$(\#$\w_{j,\ell})/s$},
			ylabel={success rate},
			legend style={legend cell align=left}, legend pos=south east,
			legend columns = 1,
		]
		\addplot[black,mark=square,mark size=2.5pt,mark options={solid}] coordinates {
(3.5,0.00) (4.0,0.03) (4.5,0.03) (5.0,0.23) (5.5,0.31) (6.0,0.58) (6.5,0.76) (7.0,0.91) (7.5,0.95) (8.0,0.99) (8.5,0.99) (9.0,1.00)
};
\addlegendentry{$s=10$}
		\addplot[blue,mark=triangle,mark size=2.5pt,mark options={solid}] coordinates {
(3.0,0.00) (3.5,0.12) (4.0,0.62) (4.5,0.99) (5.0,0.99) (5.5,1.00)

};
\addlegendentry{$s=25$}
		\addplot[red,mark=pentagon,mark size=2.5pt,mark options={solid,rotate=180}] coordinates {
(2.5,0.01) (3.0,0.68) (3.5,0.98) (4.0,0.99) (4.5,1.00)
};
\addlegendentry{$s=50$}
		\end{axis}
		\end{tikzpicture}
}
\hfill
\subfloat[$\mathrm{SNR}_{\mathrm{db}}=10$]{
	\label{fig:mixed1:d50:success_vs_m1s:snr10}
		\begin{tikzpicture}[baseline]
		\begin{axis}[
			font=\footnotesize,
			enlarge x limits=true,
			enlarge y limits=true,
			height=0.26\textwidth,
			grid=major,
			width=0.48\textwidth,
			xtick={3,4,5,6,8,10},
			ytick={0,0.5,0.8,1},
			xmin=2,xmax=11.5,
			ymin=0,ymax=1,
			xlabel={$m_1/s\;=\;$(\#$\w_{j,\ell})/s$},
			ylabel={success rate},
			legend style={legend cell align=left}, legend pos=south east,
			legend columns = 1,
		]
		\addplot[black,mark=square,mark size=2.5pt,mark options={solid}] coordinates {
(4.0,0.01) (4.5,0.07) (5.0,0.32) (5.5,0.38) (6.0,0.61) (6.5,0.85) (7.0,0.94) (7.5,0.97) (8.0,0.99) (8.5,0.99) (9.0,1.00)
};
\addlegendentry{$s=10$}
		\addplot[blue,mark=triangle,mark size=2.5pt,mark options={solid}] coordinates {
(3.0,0.00) (3.5,0.29) (4.0,0.89) (4.5,0.99) (5.0,1.00)
};
\addlegendentry{$s=25$}
		\addplot[red,mark=pentagon,mark size=2.5pt,mark options={solid,rotate=180}] coordinates {
(2.0,0.00) (2.5,0.06) (3.0,0.86) (3.5,1.00)
};
\addlegendentry{$s=50$}
		\end{axis}
		\end{tikzpicture}
}
\caption{Success rate vs.\ $m_1/s$ for mixed bases, $d=D=50$, $N=200$, sparsity $s\in\{10,25,50\}$, $m_2=\#\z_{j,k}=4\,s$, and $\mathrm{SNR}_{\mathrm{db}}\in\{0,3,6,10\}$.}
\label{fig:mixed1:d50:success_vs_m1s}
\end{figure}

For different choices of the parameter $m_1=\#\w_{j,\ell}\in\{2.5s,3s,3.5s,\ldots,10s\}$, we investigate the success rate for spatial dimension $d=D=50$ and sparsities $s\in\{10,25,50\}$ in Figure~\ref{fig:mixed1:d50:success_vs_m1s}, where we set the signal to noise ratio $\mathrm{SNR}_{\mathrm{db}}$ to 0, 3, 6, and 10 in Figure~\ref{fig:mixed1:d50:success_vs_m1s:snr0}, \ref{fig:mixed1:d50:success_vs_m1s:snr3}, \ref{fig:mixed1:d50:success_vs_m1s:snr6}, and~\ref{fig:mixed1:d50:success_vs_m1s:snr10}, respectively. We observe that the success rates increase for growing parameter~$m_1$. Moreover, the transition between 0\% success rate and 99\%--100\% success rate occurs relatively fast. Additionally, the value $m_1/s$ where the success rate reaches 99\% seems to decrease for increasing sparsity~$s$ and for increasing signal to noise ratio.

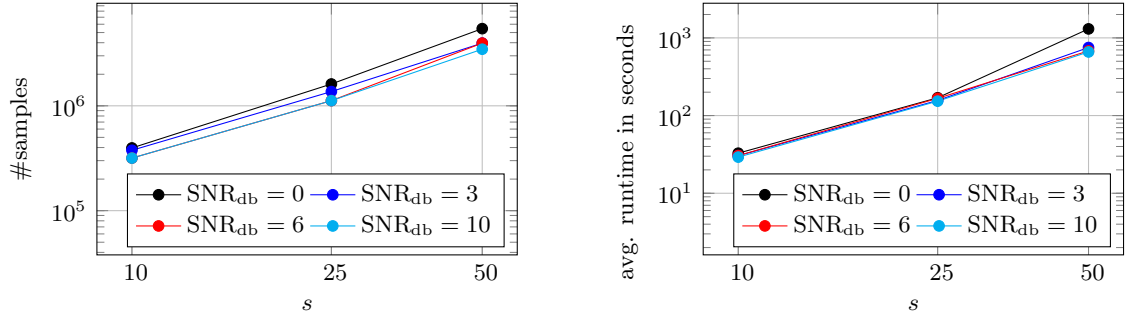
\begin{figure}[htbp]
\begin{tikzpicture}[baseline]
		\begin{axis}[
			font=\footnotesize,
			enlarge x limits=true,
			enlarge y limits=true,
			height=0.33\textwidth,
			grid=major,
			width=0.48\textwidth,
			xtick={10,25,50},
            ymin=6e4,ymax=6e6,
            xmode=log,
            ymode=log,
      xticklabel={
        \pgfkeys{/pgf/fpu=true}
        \pgfmathparse{exp(\tick)}%
        \pgfmathprintnumber[fixed relative, precision=3]{\pgfmathresult}
        \pgfkeys{/pgf/fpu=false}
      },
			xlabel={$s$},
			ylabel={\#samples},
			legend style={legend cell align=left}, legend pos=south east,
			legend columns = 2,
		]
		\addplot[black,mark=*,mark size=2pt,mark options={solid}] coordinates {
(10,396500)  (25,1614950)  (50,5447500)
};
\addlegendentry{$\mathrm{SNR}_{\mathrm{db}}=0$}
		\addplot[blue,mark=*,mark size=2pt,mark options={solid}] coordinates {
(10,376700)  (25,1367450)  (50,3962500)
};
\addlegendentry{$\mathrm{SNR}_{\mathrm{db}}=3$}
		\addplot[red,mark=*,mark size=2pt,mark options={solid}] coordinates {
(10,317300)  (25,1119950)  (50,3962500)
};
\addlegendentry{$\mathrm{SNR}_{\mathrm{db}}=6$}
		\addplot[cyan,mark=*,mark size=2pt,mark options={solid}] coordinates {
(10,317300)  (25,1119950)  (50,3467500)
};
\addlegendentry{$\mathrm{SNR}_{\mathrm{db}}=10$}
		\end{axis}
		\end{tikzpicture}
%
\hfill
%
\begin{tikzpicture}[baseline]
		\begin{axis}[
			font=\footnotesize,
			enlarge x limits=true,
			enlarge y limits=true,
			height=0.33\textwidth,
			grid=major,
			width=0.48\textwidth,
			xtick={10,25,50},
			ymin=3,ymax=1.5e3,
            xmode=log,
            ymode=log,
      xticklabel={
        \pgfkeys{/pgf/fpu=true}
        \pgfmathparse{exp(\tick)}%
        \pgfmathprintnumber[fixed relative, precision=3]{\pgfmathresult}
        \pgfkeys{/pgf/fpu=false}
      },
			xlabel={$s$},
			ylabel={avg. runtime in seconds},
		y label style={xshift=-0.8em},
			legend style={legend cell align=left}, legend pos=south east,
			legend columns = 2,
		]
		\addplot[black,mark=*,mark size=2pt,mark options={solid}] coordinates {
 (10,32.78) (25,169.43) (50,1.305e3)
};
\addlegendentry{$\mathrm{SNR}_{\mathrm{db}}=0$}
		\addplot[blue,mark=*,mark size=2pt,mark options={solid}] coordinates {
 (10,30.46) (25,156.52) (50,750.97)
};
\addlegendentry{$\mathrm{SNR}_{\mathrm{db}}=3$}
		\addplot[red,mark=*,mark size=2pt,mark options={solid}] coordinates {
 (10,30.56) (25,166.59) (50,680.28)
};
\addlegendentry{$\mathrm{SNR}_{\mathrm{db}}=6$}
		\addplot[cyan,mark=*,mark size=2pt,mark options={solid}] coordinates {
 (10,29.20) (25,153.32) (50,659.03)
};
\addlegendentry{$\mathrm{SNR}_{\mathrm{db}}=10$}
		\end{axis}
		\end{tikzpicture}
\caption{Number of samples and average runtime vs.\ sparsity $s$ for mixed bases, $\geq 99\%$ success rate, $d=D=50$, $N=200$.}
\label{fig:mixed1:d50:samples_runtime_vs_s}
\end{figure}

In Figure~\ref{fig:mixed1:d50:samples_runtime_vs_s}, we plot the used number of samples and average runtime as a function of the sparsity $s\in\{10,25,50\}$ for spatial dimensoin $d=D=50$ and for each signal to noise ratio $\mathrm{SNR}_{\mathrm{db}}\in\{0,3,6,10\}$. We observe that the plots only differ slightly for the different signal to noise ratios $\mathrm{SNR}_{\mathrm{db}}\in\{3,6,10\}$, i.e.\ the numbers of samples and runtimes seem to depend only mildly on the signal to noise ratios for $\geq 99\%$ success rate. In the case $\mathrm{SNR}_{\mathrm{db}}=0$, i.e., when the energy of the signal and of the noise match, the runtimes for sparsities $s\in\{10,25\}$ are similar to the ones of $\mathrm{SNR}_{\mathrm{db}}\in\{3,6,10\}$ and approximately double for $s=50$.

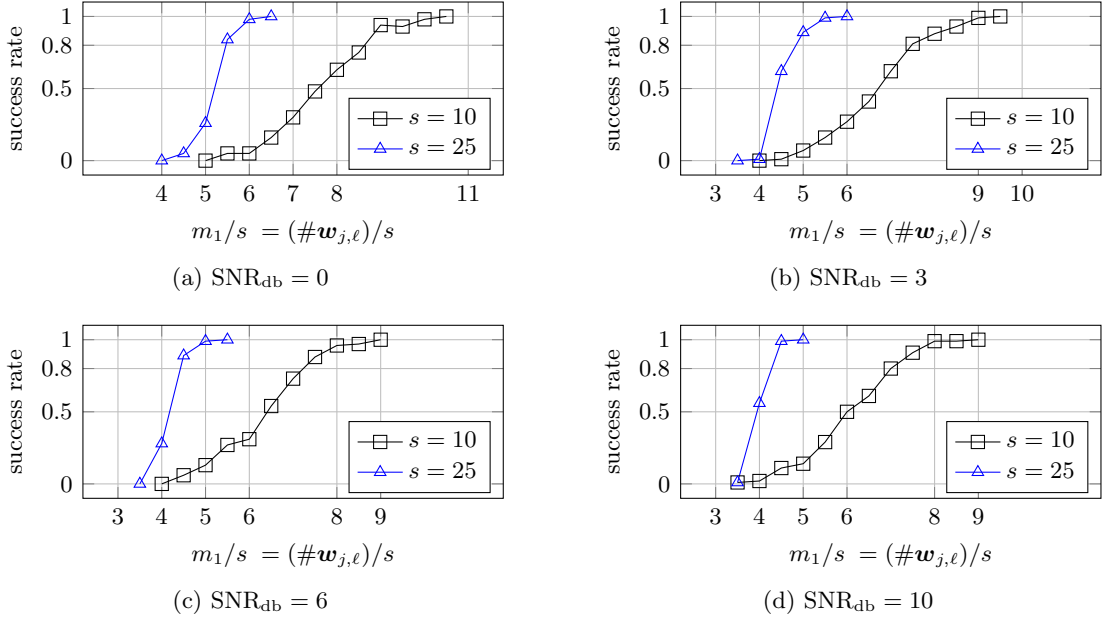
\begin{figure}[htb]
\subfloat[$\mathrm{SNR}_{\mathrm{db}}=0$]{
		\begin{tikzpicture}[baseline]
		\begin{axis}[
			font=\footnotesize,
			enlarge x limits=true,
			enlarge y limits=true,
			height=0.26\textwidth,
			grid=major,
			width=0.48\textwidth,
			xtick={4,5,6,7,8,11},
			ytick={0,0.5,0.8,1},
			xmin=3,xmax=11,
			ymin=0,ymax=1,
			xlabel={$m_1/s\;=\;$(\#$\w_{j,\ell})/s$},
			ylabel={success rate},
			legend style={legend cell align=left}, legend pos=south east,
			legend columns = 1,
		]
		\addplot[black,mark=square,mark size=2.5pt,mark options={solid}] coordinates {
(5.0,0.00) (5.5,0.05) (6.0,0.05) (6.5,0.16) (7.0,0.30) (7.5,0.48) (8.0,0.63) (8.5,0.75) (9.0,0.94) (9.5,0.93) (10.0,0.98) (10.5,1.00)
};
\addlegendentry{$s=10$}
		\addplot[blue,mark=triangle,mark size=2.5pt,mark options={solid}] coordinates {
(4.0,0.00) (4.5,0.05) (5.0,0.26) (5.5,0.84) (6.0,0.98) (6.5,1.00) 
};
\addlegendentry{$s=25$}
		\end{axis}
		\end{tikzpicture}
}
\hfill
\subfloat[$\mathrm{SNR}_{\mathrm{db}}=3$]{
		\begin{tikzpicture}[baseline]
		\begin{axis}[
			font=\footnotesize,
			enlarge x limits=true,
			enlarge y limits=true,
			height=0.26\textwidth,
			grid=major,
			width=0.48\textwidth,
			xtick={3,4,5,6,9,10},
			ytick={0,0.5,0.8,1},
			xmin=3,xmax=11,
			ymin=0,ymax=1,
			xlabel={$m_1/s\;=\;$(\#$\w_{j,\ell})/s$},
			ylabel={success rate},
			legend style={legend cell align=left}, legend pos=south east,
			legend columns = 1,
		]
		\addplot[black,mark=square,mark size=2.5pt,mark options={solid}] coordinates {
(4.0,0.00) (4.5,0.01) (5.0,0.07) (5.5,0.16) (6.0,0.27) (6.5,0.41) (7.0,0.62) (7.5,0.81) (8.0,0.88) (8.5,0.93) (9.0,0.99) (9.5,1.00)
};
\addlegendentry{$s=10$}
		\addplot[blue,mark=triangle,mark size=2.5pt,mark options={solid}] coordinates {
(3.5,0.00) (4.0,0.01) (4.5,0.62) (5.0,0.89) (5.5,0.99) (6.0,1.00)
};
\addlegendentry{$s=25$}
		\end{axis}
		\end{tikzpicture}
}
\\
\subfloat[$\mathrm{SNR}_{\mathrm{db}}=6$]{
		\begin{tikzpicture}[baseline]
		\begin{axis}[
			font=\footnotesize,
			enlarge x limits=true,
			enlarge y limits=true,
			height=0.26\textwidth,
			grid=major,
			width=0.48\textwidth,
			xtick={3,4,5,6,8,9},
			ytick={0,0.5,0.8,1},
			xmin=3,xmax=11,
			ymin=0,ymax=1,
			xlabel={$m_1/s\;=\;$(\#$\w_{j,\ell})/s$},
			ylabel={success rate},
			legend style={legend cell align=left}, legend pos=south east,
			legend columns = 1,
		]
		\addplot[black,mark=square,mark size=2.5pt,mark options={solid}] coordinates {
(4.0,0.00) (4.5,0.06) (5.0,0.13) (5.5,0.27) (6.0,0.31) (6.5,0.54) (7.0,0.73) (7.5,0.88) (8.0,0.96) (8.5,0.97) (9.0,1.00)
};
\addlegendentry{$s=10$}
		\addplot[blue,mark=triangle,mark size=2.5pt,mark options={solid}] coordinates {
(3.5,0.00) (4.0,0.28) (4.5,0.89) (5.0,0.99) (5.5,1.00)
};
\addlegendentry{$s=25$}
		\end{axis}
		\end{tikzpicture}
}
\hfill
\subfloat[$\mathrm{SNR}_{\mathrm{db}}=10$]{
		\begin{tikzpicture}[baseline]
		\begin{axis}[
			font=\footnotesize,
			enlarge x limits=true,
			enlarge y limits=true,
			height=0.26\textwidth,
			grid=major,
			width=0.48\textwidth,
			xtick={3,4,5,6,8,9},
			ytick={0,0.5,0.8,1},
			xmin=3,xmax=11,
			ymin=0,ymax=1,
			xlabel={$m_1/s\;=\;$(\#$\w_{j,\ell})/s$},
			ylabel={success rate},
			legend style={legend cell align=left}, legend pos=south east,
			legend columns = 1,
		]
		\addplot[black,mark=square,mark size=2.5pt,mark options={solid}] coordinates {
(3.5,0.01) (4.0,0.02) (4.5,0.11) (5.0,0.14) (5.5,0.29) (6.0,0.50) (6.5,0.61) (7.0,0.80) (7.5,0.91) (8.0,0.99) (8.5,0.99) (9.0,1.00)
};
\addlegendentry{$s=10$}
		\addplot[blue,mark=triangle,mark size=2.5pt,mark options={solid}] coordinates {
(3.5,0.01) (4.0,0.56) (4.5,0.99) (5.0,1.00)
};
\addlegendentry{$s=25$}
		\end{axis}
		\end{tikzpicture}
}
\caption{Success rate vs.\ $m_1/s$ for mixed bases, $d=D=100$, $N=200$, sparsity $s\in\{10,25\}$, $m_2=\#\z_{j,k}=4\,s$, and $\mathrm{SNR}_{\mathrm{db}}\in\{0,3,6,10\}$.}
\label{fig:mixed1:d100:success_vs_m1s}
\end{figure}

Additionally, we repeat the tests of Figure~\ref{fig:mixed1:d50:success_vs_m1s} for spatial dimension $d=D=100$ and sparsities $s\in\{10,25\}$, and we visualize the corresponding results in Figure~\ref{fig:mixed1:d100:success_vs_m1s}. We obtain results analogously to the previous ones.

\FloatBarrier

\newpage  

\subsubsection{Fourier bases}\label{sec:Numerics:exactly_sparse:fourier}

As for the case of mixed bases in Section~\ref{sec:Numerics:exactly_sparse:mixed1}, we now perform the numerical tests for tensor products of Fourier bases and show the results in Figure~\ref{fig:fourier:d100:samples_runtime_iter_success_vs_d}. Here, the overall BOS constant $K$ is $1$ independent of the spatial dimension~$d$. Due to the smaller BOS constant, we can reduce the parameters $m_1$ to $5s$ and $m_2$ to $s$ while still obtaining a success rate of 100\%, cf.\ Figure~\ref{fig:fourier:d100:samples_runtime_iter_success_vs_d:success}. As in Section~\ref{sec:Numerics:exactly_sparse:mixed1}, the numbers of samples in Figure~\ref{fig:fourier:d100:samples_runtime_iter_success_vs_d:samples} grow nearly linearly in~$d$ and the average runtimes in Figure~\ref{fig:fourier:d100:samples_runtime_iter_success_vs_d:runtimes} approximately like $\sim d^2$. The average number of iterations in Figure~\ref{fig:fourier:d100:samples_runtime_iter_success_vs_d:iterations} is between $3$ and $4$.


\begin{figure}[!h]
\subfloat[number of samples vs.\ spatial dimension $d$]{
\label{fig:fourier:d100:samples_runtime_iter_success_vs_d:samples}
		\begin{tikzpicture}[baseline]
		\begin{axis}[
			font=\footnotesize,
			enlarge x limits=true,
			enlarge y limits=true,
			height=0.3\textwidth,
			grid=major,
			width=0.45\textwidth,
			xtick={6,10,25,50,100},
            ymin=2e3,
            xmode=log,
            ymode=log,
      xticklabel={
        \pgfkeys{/pgf/fpu=true}
        \pgfmathparse{exp(\tick)}%
        \pgfmathprintnumber[fixed relative, precision=3]{\pgfmathresult}
        \pgfkeys{/pgf/fpu=false}
      },
			xlabel={$d$},
			ylabel={\#samples},
			legend style={legend cell align=left}, legend pos=south east,
			legend columns = -1,
		]
		\addplot[black,mark=square,mark size=2.5pt,mark options={solid},only marks] coordinates {
(6,6.000e+03) (10,1.000e+04) (25,2.500e+04) (50,5.000e+04) (75,7.500e+04) (100,1.000e+05)
};
\addlegendentry{$s=10$}
\addplot [black,domain=6:100, samples=100, dashed]{1e3*x};
\addlegendentry{$\sim d$}
		\addplot[blue,mark=triangle,mark size=2.5pt,mark options={solid},only marks] coordinates {
(6,3.562e+04) (10,6.062e+04) (25,1.544e+05) (50,3.106e+05) (75,4.669e+05) (100,6.231e+05)
};
\addlegendentry{$s=25$}
\addplot [forget plot,blue,domain=6:100, samples=100, dashed]{6.2e3*x};
		\end{axis}
		\end{tikzpicture}
}
\hfill
\subfloat[average runtime vs.\ spatial dimension $d$]{
\label{fig:fourier:d100:samples_runtime_iter_success_vs_d:runtimes}
		\begin{tikzpicture}[baseline]
		\begin{axis}[
			font=\footnotesize,
			enlarge x limits=true,
			enlarge y limits=true,
			height=0.3\textwidth,
			grid=major,
			width=0.48\textwidth,
			xtick={6,10,25,50,100},
			ymin=0.006,ymax=100,
            xmode=log,
            ymode=log,
            log ticks with fixed point,
			xlabel={$d$},
			ylabel={avg. runtime in seconds},
		y label style={xshift=-0.8em},
			legend style={legend cell align=left}, legend pos=south east,
			legend columns = -1,
		]
		\addplot[black,mark=square,mark size=2.5pt,mark options={solid},only marks] coordinates {
(6,7.290e-02) (10,1.499e-01) (25,6.645e-01) (50,2.167e+00) (75,4.625e+00) (100,7.923e+00)
};
\addlegendentry{$s=10$}
\addplot [black,domain=6:100, samples=100, dashed]{8.8e-04*x*x};
\addlegendentry{$\sim d^2$}
		\addplot[blue,mark=triangle,mark size=2.5pt,mark options={solid},only marks] coordinates {
(6,4.564e-01) (10,1.078e+00) (25,5.712e+00) (50,2.043e+01) (75,4.718e+01) (100,8.028e+01)
};
\addlegendentry{$s=25$}
\addplot [forget plot,blue,domain=6:100, samples=100, dashed]{0.0082*x*x};
		\end{axis}
		\end{tikzpicture}
}
\\
\subfloat[average iteration vs.\ spatial dimension $d$]{
\label{fig:fourier:d100:samples_runtime_iter_success_vs_d:iterations}
		\begin{tikzpicture}[baseline]
		\begin{axis}[
			font=\footnotesize,
			enlarge x limits=true,
			enlarge y limits=true,
			height=0.25\textwidth,
			grid=major,
			width=0.48\textwidth,
			xtick={6,10,25,50,100},
			ytick={1,2,3,4,5},
			ymin=1,ymax=4,
            xmode=log,
            log ticks with fixed point,
			xlabel={$d$},
			ylabel={avg. iteration},
			legend style={legend cell align=left}, legend pos=south east,
			legend columns = -1,
		]
		\addplot[black,mark=square,mark size=2.5pt,mark options={solid}] coordinates {
(6,3.210e+00) (10,3.320e+00) (25,3.430e+00) (50,3.370e+00) (75,3.420e+00) (100,3.480e+00)
};
\addlegendentry{$s=10$}
		\addplot[blue,mark=triangle,mark size=2.5pt,mark options={solid}] coordinates {
(6,3.540e+00) (10,3.620e+00) (25,3.670e+00) (50,3.500e+00) (75,3.670e+00) (100,3.600e+00)
};
\addlegendentry{$s=25$}
		\end{axis}
		\end{tikzpicture}
}
\hfill
\subfloat[success rate vs.\ spatial dimension~$d$]{
	\label{fig:fourier:d100:samples_runtime_iter_success_vs_d:success}
		\begin{tikzpicture}[baseline]
		\begin{axis}[
			font=\footnotesize,
			enlarge x limits=true,
			enlarge y limits=true,
			height=0.25\textwidth,
			grid=major,
			width=0.48\textwidth,
			xtick={6,10,25,50,100},
			ytick={0,0.5,0.8,1},
			ymin=0.0,ymax=1,
            xmode=log,
            log ticks with fixed point,
			xlabel={$d$},
			ylabel={success rate},
			legend style={legend cell align=left}, legend pos=south east,
			legend columns = -1,
		]
		\addplot[black,mark=square,mark size=2.5pt,mark options={solid}] coordinates {
(6,1.000) (10,1.000) (25,1.000) (50,1.000) (75,1.000) (100,1.000)
};
\addlegendentry{$s=10$}
		\addplot[blue,mark=triangle,mark size=2.5pt,mark options={solid}] coordinates {
(6,1.000) (10,1.000) (25,1.000) (50,1.000) (75,1.000) (100,1.000)
};
\addlegendentry{$s=25$}
		\end{axis}
		\end{tikzpicture}
}
\caption{Number of samples, runtime, number of iterations, success rate vs.\ spatial dimension $d=D\in\{6,10,25,50,100\}$ for Fourier bases, $N=200$, sparsity $s\in\{10,25\}$, $\mathrm{SNR}_{\mathrm{db}}=10$, $m_1=5s$, $m_2=s$.}
\label{fig:fourier:d100:samples_runtime_iter_success_vs_d}
\end{figure}
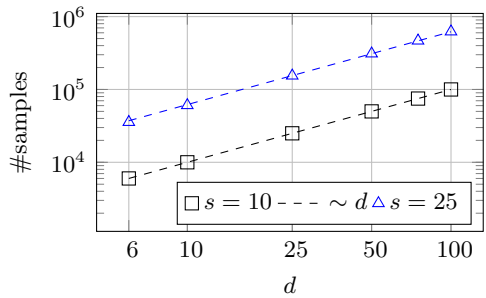
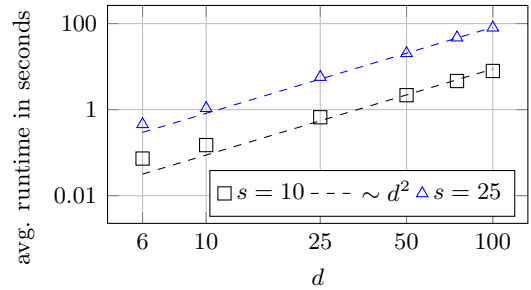
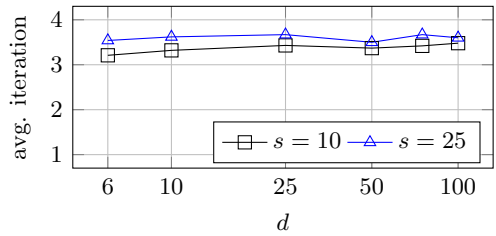
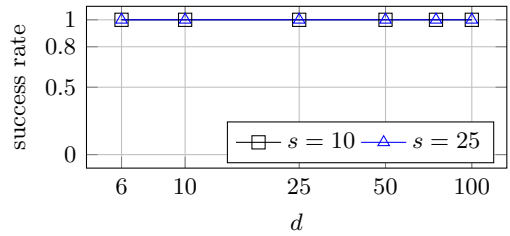

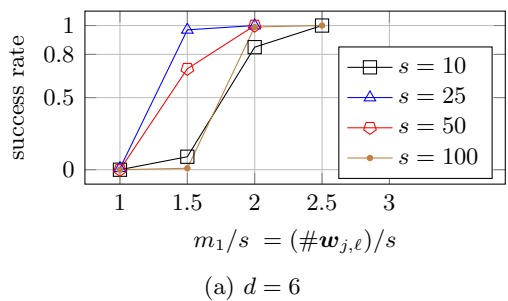
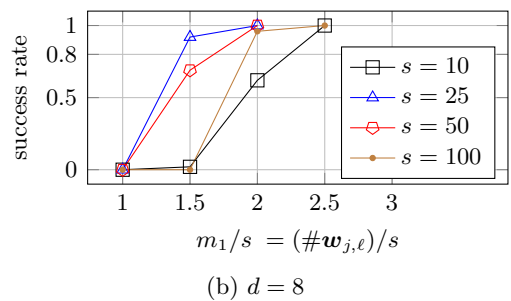
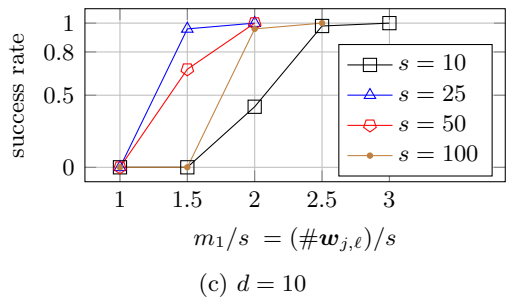
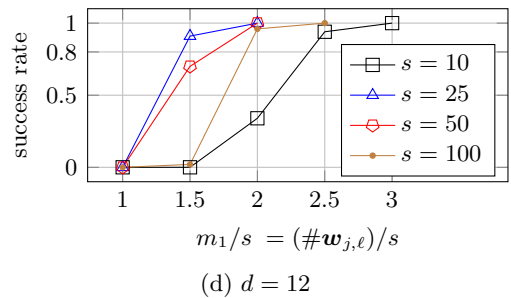
\begin{figure}[!h]
\subfloat[$d=6$]{
		\begin{tikzpicture}[baseline]
		\begin{axis}[
			font=\footnotesize,
			enlarge x limits=true,
			enlarge y limits=true,
			height=0.26\textwidth,
			grid=major,
			width=0.48\textwidth,
			xtick={1,1.5,2,2.5,3},
			ytick={0,0.5,0.8,1},
			xmax=3.6,
			ymin=0,ymax=1,
			xlabel={$m_1/s\;=\;$(\#$\w_{j,\ell})/s$},
			ylabel={success rate},
			legend style={legend cell align=left}, legend pos=south east,
			legend columns = 1,
		]
		\addplot[black,mark=square,mark size=2.5pt,mark options={solid}] coordinates {
(1.0,0.00) (1.5,0.09) (2.0,0.85) (2.5,1.00) 
};
\addlegendentry{$s=10$}
		\addplot[blue,mark=triangle,mark size=2.5pt,mark options={solid}] coordinates {
(1.0,0.01) (1.5,0.97) (2.0,1.00) 
};
\addlegendentry{$s=25$}
		\addplot[red,mark=pentagon,mark size=2.5pt,mark options={solid,rotate=180}] coordinates {
(1.0,0.00) (1.5,0.70) (2.0,1.00) 
};
\addlegendentry{$s=50$}
		\addplot[brown,mark=*,mark size=1pt,mark options={solid}] coordinates {
(1.0,0.00) (1.5,0.01) (2.0,0.99) (2.5,1.00) 
};
\addlegendentry{$s=100$}
		\end{axis}
		\end{tikzpicture}
}
\hfill
\subfloat[$d=8$]{
		\begin{tikzpicture}[baseline]
		\begin{axis}[
			font=\footnotesize,
			enlarge x limits=true,
			enlarge y limits=true,
			height=0.26\textwidth,
			grid=major,
			width=0.48\textwidth,
			xtick={1,1.5,2,2.5,3},
			ytick={0,0.5,0.8,1},
			xmax=3.6,
			ymin=0,ymax=1,
			xlabel={$m_1/s\;=\;$(\#$\w_{j,\ell})/s$},
			ylabel={success rate},
			legend style={legend cell align=left}, legend pos=south east,
			legend columns = 1,
		]
		\addplot[black,mark=square,mark size=2.5pt,mark options={solid}] coordinates {
(1.0,0.00) (1.5,0.02) (2.0,0.62) (2.5,1.00) 
};
\addlegendentry{$s=10$}
		\addplot[blue,mark=triangle,mark size=2.5pt,mark options={solid}] coordinates {
(1.0,0.00) (1.5,0.92) (2.0,1.00) 
};
\addlegendentry{$s=25$}
		\addplot[red,mark=pentagon,mark size=2.5pt,mark options={solid,rotate=180}] coordinates {
(1.0,0.00) (1.5,0.69) (2.0,1.00) 
};
\addlegendentry{$s=50$}
		\addplot[brown,mark=*,mark size=1pt,mark options={solid}] coordinates {
(1.0,0.00) (1.5,0.00) (2.0,0.96) (2.5,1.00) 
};
\addlegendentry{$s=100$}
		\end{axis}
		\end{tikzpicture}
}
\\
\subfloat[$d=10$]{
		\begin{tikzpicture}[baseline]
		\begin{axis}[
			font=\footnotesize,
			enlarge x limits=true,
			enlarge y limits=true,
			height=0.26\textwidth,
			grid=major,
			width=0.48\textwidth,
			xtick={1,1.5,2,2.5,3},
			ytick={0,0.5,0.8,1},
			xmax=3.6,
			ymin=0,ymax=1,
			xlabel={$m_1/s\;=\;$(\#$\w_{j,\ell})/s$},
			ylabel={success rate},
			legend style={legend cell align=left}, legend pos=south east,
			legend columns = 1,
		]
		\addplot[black,mark=square,mark size=2.5pt,mark options={solid}] coordinates {
(1.0,0.00) (1.5,0.00) (2.0,0.42) (2.5,0.98) (3.0,1.00)
};
\addlegendentry{$s=10$}
		\addplot[blue,mark=triangle,mark size=2.5pt,mark options={solid}] coordinates {
(1.0,0.00) (1.5,0.96) (2.0,1.00) 
};
\addlegendentry{$s=25$}
		\addplot[red,mark=pentagon,mark size=2.5pt,mark options={solid,rotate=180}] coordinates {
(1.0,0.00) (1.5,0.68) (2.0,1.00) 
};
\addlegendentry{$s=50$}
		\addplot[brown,mark=*,mark size=1pt,mark options={solid}] coordinates {
(1.0,0.00) (1.5,0.00) (2.0,0.96) (2.5,1.00) 
};
\addlegendentry{$s=100$}
		\end{axis}
		\end{tikzpicture}
}
\hfill
\subfloat[$d=12$]{
		\begin{tikzpicture}[baseline]
		\begin{axis}[
			font=\footnotesize,
			enlarge x limits=true,
			enlarge y limits=true,
			height=0.26\textwidth,
			grid=major,
			width=0.48\textwidth,
			xtick={1,1.5,2,2.5,3},
			ytick={0,0.5,0.8,1},
			ymin=0,ymax=1,
			xmax=3.6,
			xlabel={$m_1/s\;=\;$(\#$\w_{j,\ell})/s$},
			ylabel={success rate},
			legend style={legend cell align=left}, legend pos=south east,
			legend columns = 1,
		]
		\addplot[black,mark=square,mark size=2.5pt,mark options={solid}] coordinates {
(1.0,0.00) (1.5,0.00) (2.0,0.34) (2.5,0.94) (3.0,1.00)
};
\addlegendentry{$s=10$}
		\addplot[blue,mark=triangle,mark size=2.5pt,mark options={solid}] coordinates {
(1.0,0.00) (1.5,0.91) (2.0,1.00) 
};
\addlegendentry{$s=25$}
		\addplot[red,mark=pentagon,mark size=2.5pt,mark options={solid,rotate=180}] coordinates {
(1.0,0.00) (1.5,0.70) (2.0,1.00) 
};
\addlegendentry{$s=50$}
		\addplot[brown,mark=*,mark size=1pt,mark options={solid}] coordinates {
(1.0,0.00) (1.5,0.02) (2.0,0.96) (2.5,1.00) 
};
\addlegendentry{$s=100$}
		\end{axis}
		\end{tikzpicture}
}
\caption{Success rate vs.\ $m_1/s$ for Fourier bases, spatial dimension $d=D\in\{6,8,10,12\}$, $N=200$, sparsity $s\in\{10,25,50,100\}$, $\mathrm{SNR}_{\mathrm{db}}=10$, $m_2=s$.}
\label{fig:fourier:success_vs_m1s}
\end{figure}

In Figure~\ref{fig:fourier:success_vs_m1s}, we depict the success rate as a function of $m_1/s\in\{1,1.5,2,2.5,3\}$ for sparsities $s\in\{10,25,50,100\}$ and signal to noise ratio $\mathrm{SNR}_{\mathrm{db}}=10$ in spatial dimensions $d\in\{6,8,10,12\}$. We observe a very small dependence on the spatial dimension~$d$. For $m_1=3s$, the success rate is 100\% in each considered case. Furthermore, there is a rapid transition between full and zero success rate, i.e., the success rate is 0\% for $m_1=s$ each time.

\FloatBarrier

\subsubsection{Chebyshev bases}\label{sec:Numerics:exactly_sparse:chebyshev}

Next, we consider the tensor products of Chebyshev basis functions. Here we expect larger numbers of samples and runtimes compared to the Fourier case in Section~\ref{sec:Numerics:exactly_sparse:fourier} due to the BOS constant $K=\sqrt{2}^d$ for Chebyshev and $K=1$ for Fourier. In particular, for fixed sparsity~$s$ and fixed success rate, the numbers of samples and runtimes might grow for increasing spatial dimension~$d$.

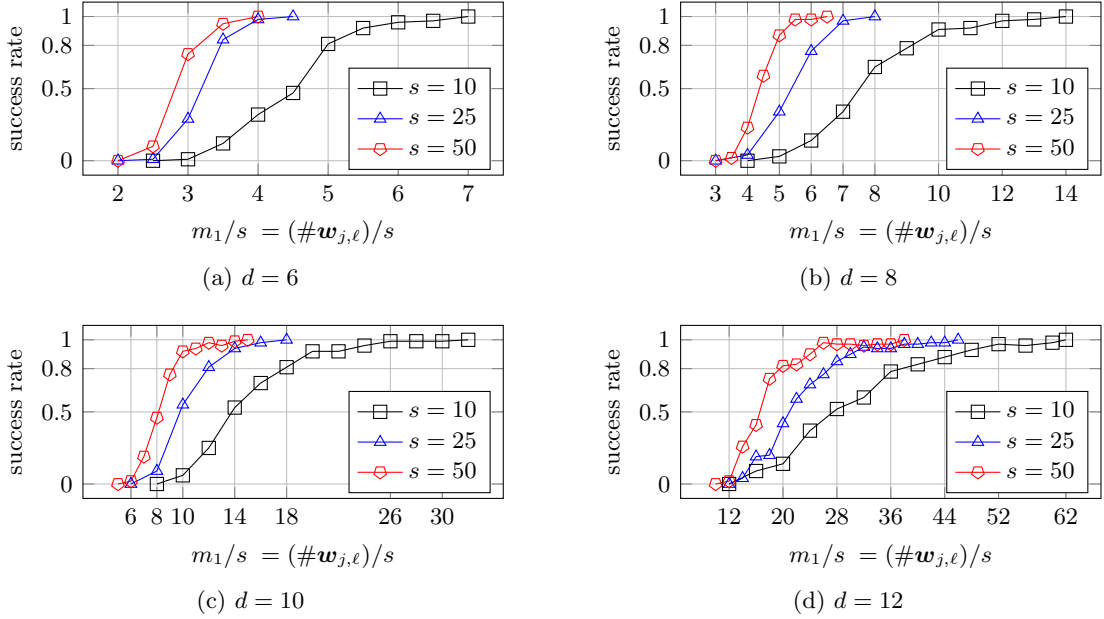
\begin{figure}[!h]
\subfloat[$d=6$]{
		\begin{tikzpicture}[baseline]
		\begin{axis}[
			font=\footnotesize,
			enlarge x limits=true,
			enlarge y limits=true,
			height=0.26\textwidth,
			grid=major,
			width=0.48\textwidth,
			xtick={2,3,4,5,6,7},
			ytick={0,0.5,0.8,1},
			ymin=0,ymax=1,
			xlabel={$m_1/s\;=\;$(\#$\w_{j,\ell})/s$},
			ylabel={success rate},
			legend style={legend cell align=left}, legend pos=south east,
			legend columns = 1,
		]
		\addplot[black,mark=square,mark size=2.5pt,mark options={solid}] coordinates {
(2.5,0.00) (3.0,0.01) (3.5,0.12) (4.0,0.32) (4.5,0.47) (5.0,0.81) (5.5,0.92) (6.0,0.96) (6.5,0.97) (7.0,1.00)
};
\addlegendentry{$s=10$}
		\addplot[blue,mark=triangle,mark size=2.5pt,mark options={solid}] coordinates {
(2.0,0.00) (2.5,0.01) (3.0,0.29) (3.5,0.84) (4.0,0.98) (4.5,1.00)
};
\addlegendentry{$s=25$}
		\addplot[red,mark=pentagon,mark size=2.5pt,mark options={solid,rotate=180}] coordinates {
(2.0,0.00) (2.5,0.10) (3.0,0.74) (3.5,0.95) (4.0,1.00)
};
\addlegendentry{$s=50$}
		\end{axis}
		\end{tikzpicture}
}
\hfill
\subfloat[$d=8$]{
		\begin{tikzpicture}[baseline]
		\begin{axis}[
			font=\footnotesize,
			enlarge x limits=true,
			enlarge y limits=true,
			height=0.26\textwidth,
			grid=major,
			width=0.48\textwidth,
			xtick={3,4,5,6,7,8,10,12,14},
			ytick={0,0.5,0.8,1},
			ymin=0,ymax=1,
			xlabel={$m_1/s\;=\;$(\#$\w_{j,\ell})/s$},
			ylabel={success rate},
			legend style={legend cell align=left}, legend pos=south east,
			legend columns = 1,
		]
		\addplot[black,mark=square,mark size=2.5pt,mark options={solid}] coordinates {
(4.0,0.00) (5.0,0.03) (6.0,0.14) (7.0,0.34) (8.0,0.65) (9.0,0.78) (10.0,0.91) (11.0,0.92) (12.0,0.97) (13.0,0.98) (14.0,1.00)
};
\addlegendentry{$s=10$}
		\addplot[blue,mark=triangle,mark size=2.5pt,mark options={solid}] coordinates {
(3.0,0.00) (4.0,0.04) (5.0,0.34) (6.0,0.76) (7.0,0.97) (8.0,1.00)
};
\addlegendentry{$s=25$}
		\addplot[red,mark=pentagon,mark size=2.5pt,mark options={solid,rotate=180}] coordinates {
(3.0,0.00) (3.5,0.02) (4.0,0.23) (4.5,0.59) (5.0,0.87) (5.5,0.98) (6.0,0.98) (6.5,1.00)
};
\addlegendentry{$s=50$}
		\end{axis}
		\end{tikzpicture}
}
\\
\subfloat[$d=10$]{
		\begin{tikzpicture}[baseline]
		\begin{axis}[
			font=\footnotesize,
			enlarge x limits=true,
			enlarge y limits=true,
			height=0.26\textwidth,
			grid=major,
			width=0.48\textwidth,
			xtick={6,8,10,14,18,26,30},
			ytick={0,0.5,0.8,1},
			ymin=0,ymax=1,
			xlabel={$m_1/s\;=\;$(\#$\w_{j,\ell})/s$},
			ylabel={success rate},
			legend style={legend cell align=left}, legend pos=south east,
			legend columns = 1,
		]
		\addplot[black,mark=square,mark size=2.5pt,mark options={solid}] coordinates {
(8.0,0.00) (10.0,0.06) (12.0,0.25) (14.0,0.53) (16.0,0.70) (18.0,0.81) (20.0,0.92) (22.0,0.92) (24.0,0.96) (26.0,0.99) (28.0,0.99) (30.0,0.99) (32.0,1.00)
};
\addlegendentry{$s=10$}
		\addplot[blue,mark=triangle,mark size=2.5pt,mark options={solid}] coordinates {
(6.0,0.00) (8.0,0.09) (10.0,0.55) (12.0,0.81) (14.0,0.94) (16.0,0.98) (18.0,1.00)
};
\addlegendentry{$s=25$}
		\addplot[red,mark=pentagon,mark size=2.5pt,mark options={solid,rotate=180}] coordinates {
(5.0,0.00) (6.0,0.02) (7.0,0.19) (8.0,0.46) (9.0,0.76) (10.0,0.92) (11.0,0.94) (12.0,0.98) (13.0,0.96) (14.0,0.99) (15.0,1.00)
};
\addlegendentry{$s=50$}
		\end{axis}
		\end{tikzpicture}
}
\hfill
\subfloat[$d=12$]{
		\begin{tikzpicture}[baseline]
		\begin{axis}[
			font=\footnotesize,
			enlarge x limits=true,
			enlarge y limits=true,
			height=0.26\textwidth,
			grid=major,
			width=0.48\textwidth,
			xtick={12,20,28,36,44,52,62},
			ytick={0,0.5,0.8,1},
			ymin=0,ymax=1,
			xlabel={$m_1/s\;=\;$(\#$\w_{j,\ell})/s$},
			ylabel={success rate},
			legend style={legend cell align=left}, legend pos=south east,
			legend columns = 1,
		]
		\addplot[black,mark=square,mark size=2.5pt,mark options={solid}] coordinates {
(12.0,0.00) (16.0,0.09) (20.0,0.14) (24.0,0.37) (28.0,0.52) (32.0,0.60) (36.0,0.78) (40.0,0.83) (44.0,0.88) (48.0,0.93) (52.0,0.97) (56.0,0.96) (60.0,0.98) (62.0,1.00)
};
\addlegendentry{$s=10$}
		\addplot[blue,mark=triangle,mark size=2.5pt,mark options={solid}] coordinates {
(12.0,0.00) (14.0,0.04) (16.0,0.19) (18.0,0.20) (20.0,0.42) (22.0,0.59) (24.0,0.69) (26.0,0.76) (28.0,0.85) (30.0,0.90) (32.0,0.95) (34.0,0.94) (36.0,0.94) (38.0,0.97) (40.0,0.97) (42.0,0.98) (44.0,0.98) (46.0,1.00) 
};
\addlegendentry{$s=25$}
		\addplot[red,mark=pentagon,mark size=2.5pt,mark options={solid,rotate=180}] coordinates {
(10.0,0.00) (12.0,0.02) (14.0,0.26) (16.0,0.41) (18.0,0.73) (20.0,0.82) (22.0,0.83) (24.0,0.90) (26.0,0.98) (28.0,0.97) (30.0,0.97) (32.0,0.96) (34.0,0.97) (36.0,0.97) (38.0,1.00) 
};
\addlegendentry{$s=50$}
		\end{axis}
		\end{tikzpicture}
}
\caption{Success rate vs.\ $m_1/s$ for Chebyshev bases, spatial dimension $d=D\in\{6,8,10,12\}$, $N=200$, sparsity $s\in\{10,25,50\}$, $\mathrm{SNR}_{\mathrm{db}}=10$, $m_2=4s$.}
\label{fig:chebyshev:success_vs_m1s}
\end{figure}

In Figure~\ref{fig:chebyshev:success_vs_m1s}, we depict the success rate as a function of $m_1/s$ for sparsities $s\in\{10,25,50\}$ and signal to noise ratio $\mathrm{SNR}_{\mathrm{db}}=10$ in spatial dimensions $d\in\{6,8,10,12\}$. As predicted, we observe that we have to increase $m_1=\#\w_{j,\ell}$ distinctly for growing spatial dimension~$d$ and fixed sparsity~$s$. For instance, for $s=25$, we observe a 98\% success rate for $m_1=4s$ and $d=6$, but obtain a success rate of only 4\% for $d=8$. For $d=12$, we had to choose $m_1=42s$ to achieve a success rate of 98\%.

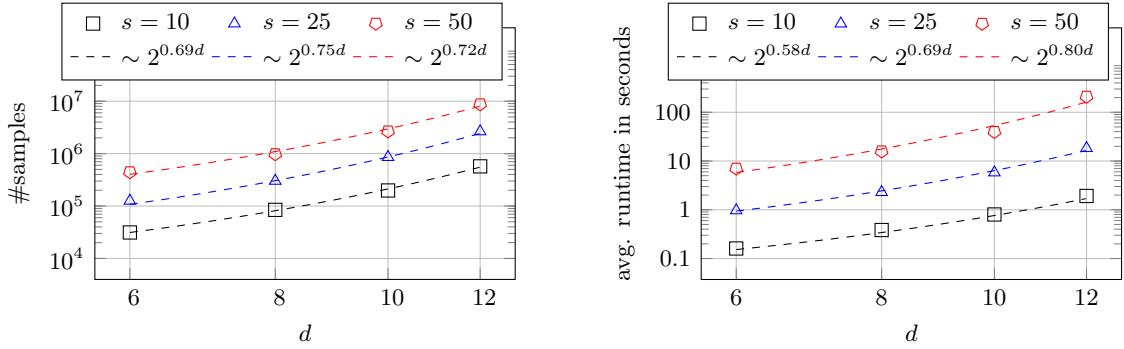
\begin{figure}[htbp]
		\begin{tikzpicture}[baseline]
		\begin{axis}[
			font=\footnotesize,
			enlarge x limits=true,
			enlarge y limits=true,
			height=0.33\textwidth,
			grid=major,
			width=0.48\textwidth,
			xtick={6,8,10,12},
			ytick={1e4,1e5,1e6,1e7},
            xmode=log,
            ymode=log,
      xticklabel={
        \pgfkeys{/pgf/fpu=true}
        \pgfmathparse{exp(\tick)}%
        \pgfmathprintnumber[fixed relative, precision=3]{\pgfmathresult}
        \pgfkeys{/pgf/fpu=false}
      },
			ymin=1e4,ymax=1e8,
			xlabel={$d$},
			ylabel={\#samples},
			legend style={legend cell align=left,at={(0.97,1.1)}},
			legend columns = 3,
		]
		\addplot[black,mark=square,mark size=2.5pt,mark options={solid},only marks] coordinates {
 (6,3.130e+04) (8,8.450e+04) (10,1.981e+05) (12,570900)
};
\addlegendentry{$s=10$}
		\addplot[blue,mark=triangle,mark size=2.5pt,mark options={solid},only marks] coordinates {
 (6,125550) (8,3.012e+05) (10,8.562e+05) (12,2646250)
};
\addlegendentry{$s=25$}
		\addplot[red,mark=pentagon,mark size=2.5pt,mark options={solid,rotate=180},only marks] coordinates {
 (6,4.425e+05) (8,977500) (10,2.662e+06) (12,8742500)
};
\addlegendentry{$s=50$}
\addplot [black,domain=6:12, samples=100, dashed]{2^10.7965 * 2^(0.69*x)};
\addlegendentry{$\sim 2^{0.69d}$}
\addplot [blue,domain=6:12, samples=100, dashed]{2^12.220807844665455 * 2^(0.75*x)};
\addlegendentry{$\sim 2^{0.75d}$}
\addplot [red,domain=6:12, samples=100, dashed]{2^14.3032 * 2^(0.72*x)};
\addlegendentry{$\sim 2^{0.72d}$}
		\end{axis}
		\end{tikzpicture}
\hfill
		\begin{tikzpicture}[baseline]
		\begin{axis}[
			font=\footnotesize,
			enlarge x limits=true,
			enlarge y limits=true,
			height=0.33\textwidth,
			grid=major,
			width=0.48\textwidth,
			ymode=log,
			xmode=log,
			log ticks with fixed point,
			xtick={4,6,8,10,12},
			ytick={0.1,1,10,100},
			ymin=1e-1,ymax=2e3,
			xlabel={$d$},
			ylabel={avg. runtime in seconds},
			legend style={legend cell align=left,at={(0.97,1.1)}},
			legend columns = 3,
		]
		\addplot[black,mark=square,mark size=2.5pt,mark options={solid},only marks] coordinates {
		 (6,0.1606) (8,0.3843) (10,0.8007) (12,1.9203)
};
\addlegendentry{$s=10$}
		\addplot[blue,mark=triangle,mark size=2.5pt,mark options={solid},only marks] coordinates {
		 (6,0.9689) (8,2.2925) (10,5.7969) (12,18.3088)
};
\addlegendentry{$s=25$}
		\addplot[red,mark=pentagon,mark size=2.5pt,mark options={solid,rotate=180},only marks] coordinates {
		 (6,7.0094) (8,15.8067) (10,39.7017) (12,209.2183)
};
\addlegendentry{$s=50$}
\addplot [black,domain=6:12, samples=100, dashed]{2^(-6.191117657342546) * 2^(0.58*x)};
\addlegendentry{$\sim 2^{0.58d}$}
\addplot [blue,domain=6:12, samples=100, dashed]{2^(-4.230602954903033) * 2^(0.69*x)};
\addlegendentry{$\sim 2^{0.69d}$}
\addplot [red,domain=6:12, samples=100, dashed]{2^(-2.2594) * 2^(0.80*x)};
\addlegendentry{$\sim 2^{0.80d}$}
		\end{axis}
		\end{tikzpicture}
\caption{Number of samples and average runtime vs.\ spatial dimension $d\in\{6,8,10,12\}$ for Chebyshev bases, $N=200$, sparsity $s\in\{10,25,50\}$, $\mathrm{SNR}_{\mathrm{db}}=10$, $m_2=4s$, success rate $\geq 99\%$.}
\label{fig:chebyshev:samples_runtime_vs_d}
\end{figure}

In Figure~\ref{fig:chebyshev:samples_runtime_vs_d}, we investigate the dependence of the spatial dimension~$d$ on the number of samples and average runtime for the case of $\geq 99\%$ success rate in more detail. For our test cases, we observe that the numbers of samples grow approximately like between $\sim 2^{0.69d}$ and $\sim 2^{0.75d}$ as well as the runtimes approximately like between $\sim 2^{0.58d}$ and $\sim 2^{0.80d}$. In each case, this is distinctly less than the worst case upper bounds in Theorem~\ref{thm:NewCoSaMP} suggest.

\FloatBarrier

\subsubsection{Preconditioned Legendre bases}\label{sec:Numerics:exactly_sparse:legendre}

Here, we consider the tensor products of preconditioned Legendre basis functions $Q_n$ with BOS constant $K=\sqrt{3}^d$, cf.\ Section~\ref{sec:Numerics:exactly_sparse:mixed1}.
In Figure~\ref{fig:legendre:success_vs_m1s}, we show the success rates as a function of $m_1/s$ for sparsities $s\in\{10,25,50\}$ and signal to noise ratio $\mathrm{SNR}_{\mathrm{db}}=10$ in spatial dimensions $d\in\{6,8,10,12\}$.
As in the case of Chebyshev bases, we observe that we have to increase $m_1$ distinctly for growing spatial dimension~$d$ and fixed sparsity~$s$. For instance, for $s=25$, we observe a 97\% success rate for $m_1=4s$ and $d=6$ as well as 100\% for $m_1=4.5s$ and $d=6$, but obtained a success rate of only 2\% for $m_1=4s$ and $d=8$. Moreover, we had to choose $m_1=18s$ to have a success rate of 100\% for $d=10$ and $m_1=56s$ for $d=12$.

When comparing the obtained results with the ones for the Chebyshev bases, we do not numerically observe the higher BOS constant $K=\sqrt{3}^d$ here for $d=6,8,10$. The plots in Figure~\ref{fig:legendre:success_vs_m1s} look very similar to the ones in Figure~\ref{fig:chebyshev:success_vs_m1s}. For $d=12$, the values of $m_1$ where a success rate of $\geq 99\%$ is reached are slightly larger than the ones in the Chebyshev case.

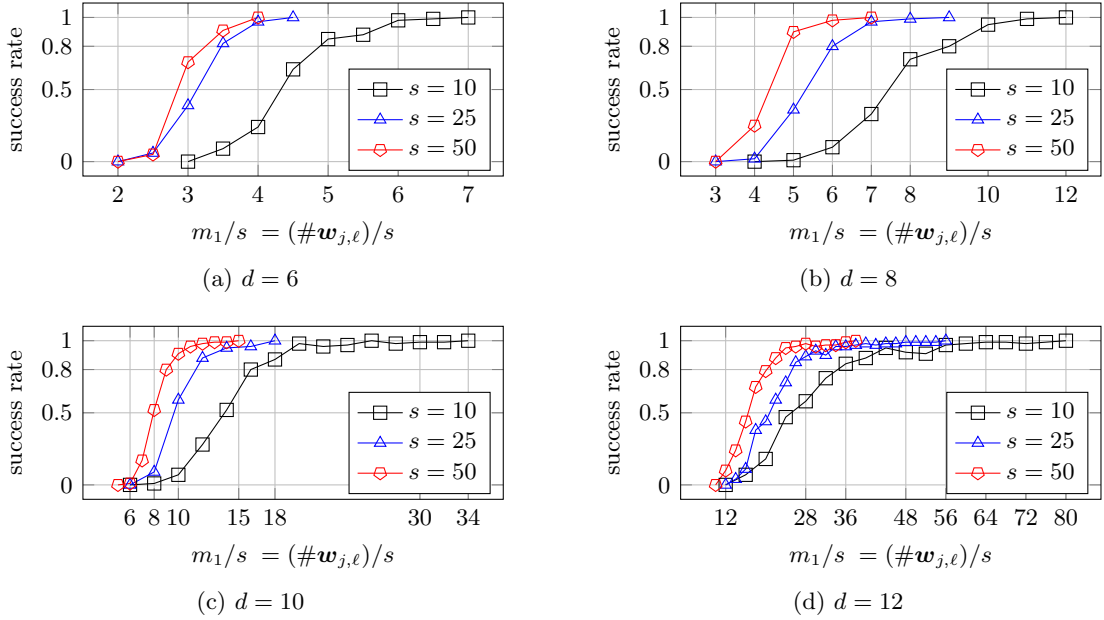
\begin{figure}[!h]
	\subfloat[$d=6$]{
		\begin{tikzpicture}[baseline]
		\begin{axis}[
		font=\footnotesize,
		enlarge x limits=true,
		enlarge y limits=true,
		height=0.26\textwidth,
		grid=major,
		width=0.48\textwidth,
		xtick={2,3,4,5,6,7},
		ytick={0,0.5,0.8,1},
		ymin=0,ymax=1,
		xlabel={$m_1/s\;=\;$(\#$\w_{j,\ell})/s$},
		ylabel={success rate},
		legend style={legend cell align=left}, legend pos=south east,
		legend columns = 1,
		]
		\addplot[black,mark=square,mark size=2.5pt,mark options={solid}] coordinates {
(3.0,0.00) (3.5,0.09) (4.0,0.24) (4.5,0.64) (5.0,0.85) (5.5,0.88) (6.0,0.98) (6.5,0.99) (7.0,1.00)
		};
		\addlegendentry{$s=10$}
		\addplot[blue,mark=triangle,mark size=2.5pt,mark options={solid}] coordinates {
(2.0,0.00) (2.5,0.06) (3.0,0.39) (3.5,0.82) (4.0,0.97) (4.5,1.00)
		};
		\addlegendentry{$s=25$}
		\addplot[red,mark=pentagon,mark size=2.5pt,mark options={solid,rotate=180}] coordinates {
(2.0,0.00) (2.5,0.05) (3.0,0.69) (3.5,0.91) (4.0,1.00)
		};
		\addlegendentry{$s=50$}
		\end{axis}
		\end{tikzpicture}
	}
	\hfill
	\subfloat[$d=8$]{
		\begin{tikzpicture}[baseline]
		\begin{axis}[
		font=\footnotesize,
		enlarge x limits=true,
		enlarge y limits=true,
		height=0.26\textwidth,
		grid=major,
		width=0.48\textwidth,
		xtick={3,4,5,6,7,8,10,12,14},
		ytick={0,0.5,0.8,1},
		ymin=0,ymax=1,
		xlabel={$m_1/s\;=\;$(\#$\w_{j,\ell})/s$},
		ylabel={success rate},
		legend style={legend cell align=left}, legend pos=south east,
		legend columns = 1,
		]
		\addplot[black,mark=square,mark size=2.5pt,mark options={solid}] coordinates {
(4.0,0.00) (5.0,0.01) (6.0,0.10) (7.0,0.33) (8.0,0.71) (9.0,0.80) (10.0,0.95) (11.0,0.99) (12.0,1.00)
		};
		\addlegendentry{$s=10$}
		\addplot[blue,mark=triangle,mark size=2.5pt,mark options={solid}] coordinates {
(3.0,0.00) (4.0,0.02) (5.0,0.36) (6.0,0.80) (7.0,0.97) (8.0,0.99) (9.0,1.00)
		};
		\addlegendentry{$s=25$}
		\addplot[red,mark=pentagon,mark size=2.5pt,mark options={solid,rotate=180}] coordinates {
(3.0,0.00) (4.0,0.25) (5.0,0.90) (6.0,0.98) (7.0,1.00)
		};
		\addlegendentry{$s=50$}
		\end{axis}
		\end{tikzpicture}
	}
	\\
	\subfloat[$d=10$]{
		\begin{tikzpicture}[baseline]
		\begin{axis}[
		font=\footnotesize,
		enlarge x limits=true,
		enlarge y limits=true,
		height=0.26\textwidth,
		grid=major,
		width=0.48\textwidth,
		xtick={6,8,10,15,18,30,34},
		ytick={0,0.5,0.8,1},
		ymin=0,ymax=1,
		xlabel={$m_1/s\;=\;$(\#$\w_{j,\ell})/s$},
		ylabel={success rate},
		legend style={legend cell align=left}, legend pos=south east,
		legend columns = 1,
		]
		\addplot[black,mark=square,mark size=2.5pt,mark options={solid}] coordinates {
(6.0,0.00) (8.0,0.01) (10.0,0.07) (12.0,0.28) (14.0,0.52) (16.0,0.80) (18.0,0.87) (20.0,0.98) (22.0,0.96) (24.0,0.97) (26.0,1.00) (28.0,0.98) (30.0,0.99) (32.0,0.99) (34.0,1.00)
		};
		\addlegendentry{$s=10$}
		\addplot[blue,mark=triangle,mark size=2.5pt,mark options={solid}] coordinates {
(6.0,0.00) (8.0,0.09) (10.0,0.59) (12.0,0.88) (14.0,0.95) (16.0,0.96) (18.0,1.00)
		};
		\addlegendentry{$s=25$}
		\addplot[red,mark=pentagon,mark size=2.5pt,mark options={solid,rotate=180}] coordinates {
(5.0,0.00) (6.0,0.01) (7.0,0.17) (8.0,0.52) (9.0,0.80) (10.0,0.91) (11.0,0.96) (12.0,0.98) (13.0,0.99) (14.0,0.99) (15.0,1.00)
		};
		\addlegendentry{$s=50$}
		\end{axis}
		\end{tikzpicture}
	}
	\hfill
	\subfloat[$d=12$]{
		\begin{tikzpicture}[baseline]
		\begin{axis}[
		font=\footnotesize,
		enlarge x limits=true,
		enlarge y limits=true,
		height=0.26\textwidth,
		grid=major,
		width=0.48\textwidth,
		xtick={12,28,36,48,56,64,72,80},
		ytick={0,0.5,0.8,1},
		ymin=0,ymax=1,
		xlabel={$m_1/s\;=\;$(\#$\w_{j,\ell})/s$},
		ylabel={success rate},
		legend style={legend cell align=left}, legend pos=south east,
		legend columns = 1,
		]
		\addplot[black,mark=square,mark size=2.5pt,mark options={solid}] coordinates {
(12.0,0.00) (16.0,0.07) (20.0,0.18) (24.0,0.47) (28.0,0.58) (32.0,0.74) (36.0,0.84) (40.0,0.88) (44.0,0.95) (48.0,0.92) (52.0,0.91) (56.0,0.97) (60.0,0.98) (64.0,0.99) (68.0,0.99) (72.0,0.98) (76.0,0.99) (80.0,1.00)
		};
		\addlegendentry{$s=10$}
		\addplot[blue,mark=triangle,mark size=2.5pt,mark options={solid}] coordinates {
(12.0,0.00) (14.0,0.04) (16.0,0.11) (18.0,0.38) (20.0,0.44) (22.0,0.59) (24.0,0.71) (26.0,0.85) (28.0,0.89) (30.0,0.93) (32.0,0.90) (34.0,0.96) (36.0,0.96) (38.0,0.97) (40.0,0.98) (42.0,0.97) (44.0,0.98) (46.0,0.98) (48.0,0.99) (50.0,0.99) (52.0,0.99) (54.0,0.99) (56.0,1.00)
		};
		\addlegendentry{$s=25$}
		\addplot[red,mark=pentagon,mark size=2.5pt,mark options={solid,rotate=180}] coordinates {
(10.0,0.00) (12.0,0.10) (14.0,0.24) (16.0,0.44) (18.0,0.68) (20.0,0.79) (22.0,0.88) (24.0,0.95) (26.0,0.96) (28.0,0.98) (30.0,0.96) (32.0,0.97) (34.0,0.97) (36.0,0.99) (38.0,1.00)
		};
		\addlegendentry{$s=50$}
		\end{axis}
		\end{tikzpicture}
	}
	\caption{Success rate vs.\ $m_1/s$ for preconditioned Legendre bases, spatial dimension $d=D\in\{6,8,10,12\}$, $N=200$, sparsity $s\in\{10,25,50\}$, $\mathrm{SNR}_{\mathrm{db}}=10$, $m_2=4s$.}
	\label{fig:legendre:success_vs_m1s}
\end{figure}

\FloatBarrier

\subsection{Approximately sparse case}
After considering exactly sparse test functions~$f$ in Section~\ref{sec:Numerics:exactly_sparse}, we continue with examples for the approximately sparse case, i.e., our test functions under consideration will have infinitely many non-zero basis coefficients~$c_\n$.

\subsubsection{Fourier type with $D=10$}
We use the 10-variate periodic test function $f\colon\mathbbm{T}^{10}\rightarrow\mathbbm{R}$,
\begin{equation} \label{equ:f:10}
f(\xib):=\prod_{t\in\{0,2,7\}}N_2(\xi_t) + \prod_{t\in\{1,4,5,9\}}N_4(\xi_t) + \prod_{t\in\{3,6,8\}}N_6(\xi_t),
\end{equation}
from \cite[Section~3.3]{potts2016sparse} and \cite[Section~5.3]{kammerer2017high} with infinitely many non-zero Fourier coefficients~$c_\n$,
where $\mathbbm{T}\simeq[0,1)$ is the torus and $N_m:\mathbbm{T}\rightarrow\mathbbm{R}$ is the B-Spline of order $m\in\mathbbm{N}$,
$$N_m(x) := C_m \sum_{n\in\mathbbm{Z}} \operatorname{sinc}\left(\frac{\pi}{m}n\right)^m (-1)^n \,\mathrm{e}^{2\pi\mathrm{i}nx},$$
with a constant $C_m>0$ such that $\Vert N_m \Vert_{L^2(\mathbbm{T})}=1$.
We remark that each B-Spline $N_m$ of order $m\in\mathbbm{N}$ is a piece-wise polynomial of degree $m-1$.
We approximate the function $f$ by multivariate trigonometric polynomials $a$ using Algorithm~\ref{alg:main}. 
The obtained basis index sets $\tilde{\Omega}$ should ``consist of'' the union of three lower dimensional manifolds,
a three-dimensional hyperbolic cross in the dimensions $1,3,8$;
a four-dimensional hyperbolic cross in the dimensions $2,5,6,10$;
and a three-dimensional hyperbolic cross in the dimensions $4,7,9$.
All tests are performed 10 times and the relative $L^2(\mathbbm{T}^{10})$  approximation error
$$
\frac{ \Vert f-a\Vert_{L^2(\mathbbm{T}^{10})} }{ \Vert f\Vert_{L^2(\mathbbm{T}^{10})} }
=
\frac{
\sqrt{\Vert f\Vert_{L^2(\mathbbm{T}^{10})}^2 - \sum_{\boldsymbol{n}\in \tilde{\Omega}}\vert\hat{f}_{\boldsymbol{n}}\vert^2 + \sum_{\boldsymbol{n}\in \tilde{\Omega}}\vert a_{\boldsymbol{n}}-\hat{f}_{\boldsymbol{n}}\vert^2} }{ \Vert f\Vert_{L^2(\mathbbm{T}^{10})} }
$$
is computed each time,
where the approximant $a:=\sum_{\boldsymbol{n}\in \tilde{\Omega}} a_{\boldsymbol{n}} \,\mathrm{e}^{2\pi\mathrm{i}\boldsymbol{n}\cdot\circ}$.

We set the parameters $N=64$, $d=D=10$, $m_2=s$, and we always use $m_{\rm CE}:=50\,s$ samples for the coefficient estimation where $s=|\S|$.
For our tests, we consider two different parameter combinations: $m_1=3s$ and $\kappa=20$, as well as $m_1=8s$ and $\kappa=10$. The obtained results, i.e., the numbers of samples, average runtimes, average iterations, and relative $L^2(\mathbbm{T}^{10})$ errors are plotted as a function of the sparsity~$s\in\{100,200,500,1000\}$ in Figure~\ref{fig:approx_sparse:fourier_d10:samples_runtime_iter_success_vs_s}. Due to the parameter choices for $m_1$ and $m_2$, we observe that the numbers of samples grow quadratically for increasing sparsity~$s$. The average runtimes grow approximately like $\sim s^3\min\{s,N\}$ and this means $\sim s^3$ for fixed $N$.
Moreover, the average numbers of iterations are much smaller than its imposed maximum $\kappa$ in most cases. The relative $L^2(\mathbbm{T}^{10})$ errors decrease for increasing sparsity $s$ having a value of approximately $10^{-2}$ for sparsity $s=1000$. Again, we emphasize the extremely high power of Algorithm~\ref{alg:main}, which is able to determine the $s=1000$ approximately largest bases coefficients and the corresponding indices for our test function out of $| \mathcal{I}_{N,d}| = N^d = 64^{10} \approx 10^{18}$ allowed indices.

\begin{figure}[!h]
	\subfloat[numbers of samples vs.\ sparsity $s$]{
		\begin{tikzpicture}[baseline]
		\begin{axis}[
		font=\footnotesize,
		enlarge x limits=true,
		enlarge y limits=true,
		height=0.3\textwidth,
		grid=major,
		width=0.47\textwidth,
		xtick={100,200,500,1000},
		xmode=log,
		ymode=log,
		xticklabel={
			\pgfkeys{/pgf/fpu=true}
			\pgfmathparse{exp(\tick)}%
			\pgfmathprintnumber[fixed relative, precision=3]{\pgfmathresult}
			\pgfkeys{/pgf/fpu=false}
		},
		xlabel={$s$},
		ylabel={\#samples},
		legend style={legend cell align=left}, legend pos=south east,
		legend columns = 1,
		]
		\addplot[black,mark=square,mark size=2.5pt,mark options={solid}] coordinates {
(100,5.750e+05) (200,2.290e+06) (500,1.428e+07) (1000,5.705e+07)
		};
	    \addlegendentry{$m_1=3s$}
		\addplot[blue,mark=triangle,mark size=2.5pt,mark options={solid}] coordinates {
(100,1.525e+06) (200,6.090e+06) (500,3.802e+07) (1000,1.520e+08)
        };
        \addlegendentry{$m_1=8s$}
		\end{axis}
		\end{tikzpicture}
	}
	\hfill
	\subfloat[average runtimes vs.\ sparsity $s$]{
		\begin{tikzpicture}[baseline]
		\begin{axis}[
		font=\footnotesize,
		enlarge x limits=true,
		enlarge y limits=true,
		height=0.3\textwidth,
		grid=major,
		width=0.48\textwidth,
		xtick={100,200,500,1000},
		ytick={1e2,1e3,1e4},
        ymax=8e4,
		xmode=log,
		ymode=log,
		ytick={85,430,5400,42000},
		log ticks with fixed point,
		xlabel={$s$},
		ylabel={avg. runtime in seconds},
		y label style={xshift=-0.8em},
		legend style={legend cell align=left,at={(1,1.3)}},
		legend columns = 2,
		]
		\addplot[black,mark=square,mark size=2.5pt,mark options={solid},only marks] coordinates {
(100,1.102e+02) (200,3.762e+02) (500,3.934e+03) (1000,3.260e+04)
		};
	    \addlegendentry{$m_1=3s$, $\kappa=20$}
		\addplot [black,domain=100:1000, samples=100, dashed]{29+3e-5*x*x*x};
		\addlegendentry{$3\cdot 10^{-5}s^3+29$}
		\addplot[blue,mark=triangle,mark size=2.5pt,mark options={solid},only marks] coordinates {
(100,1.139e+02) (200,3.930e+02) (500,5.946e+03) (1000,5.023e+04)
		};
	    \addlegendentry{$m_1=8s$, $\kappa=10$}
		\addplot [black,domain=100:1000, samples=100, dashed]{5e-5*x*x*x};
\addlegendentry{$5\cdot 10^{-5}s^3$}
		\end{axis}
		\end{tikzpicture}
	}
	\\
	\subfloat[average iteration vs.\ sparsity $s$]{
		\begin{tikzpicture}[baseline]
		\begin{axis}[
		font=\footnotesize,
		enlarge x limits=true,
		enlarge y limits=true,
		height=0.3\textwidth,
		grid=major,
		width=0.48\textwidth,
		xtick={100,200,500,1000},
		ytick={3,5,7,18},
		ymin=3,ymax=20,
		ytick={3,5,8,10,18},
		xmode=log,
		log ticks with fixed point,
		xlabel={$s$},
		ylabel={avg. iteration},
		legend style={legend cell align=left}, legend pos=north east,
		legend columns = 1,
		]
		\addplot[black,mark=square,mark size=2.5pt,mark options={solid}] coordinates {
(100,1.770e+01) (200,9.200e+00) (500,7.500e+00) (1000,8.200e+00)
		};
	    \addlegendentry{$m_1=3s$, $\kappa=20$}
		\addplot[blue,mark=triangle,mark size=2.5pt,mark options={solid}] coordinates {
(100,4.600e+00) (200,4.100e+00) (500,4.600e+00) (1000,5.000e+00)
		};
	    \addlegendentry{$m_1=8s$, $\kappa=10$}
		\end{axis}
		\end{tikzpicture}
	}
	\hfill
	\subfloat[$L^2(\mathbbm{T}^{10})$ error vs.\ sparsity $s$]{
		\begin{tikzpicture}[baseline]
		\begin{axis}[
		font=\footnotesize,
		enlarge x limits=true,
		enlarge y limits=true,
		height=0.3\textwidth,
		grid=major,
		width=0.48\textwidth,
		xtick={100,200,500,1000},
		ymin=1e-2,ymax=1,
		xmode=log,
		ymode=log,
		xticklabel={
	\pgfkeys{/pgf/fpu=true}
	\pgfmathparse{exp(\tick)}%
	\pgfmathprintnumber[fixed relative, precision=3]{\pgfmathresult}
	\pgfkeys{/pgf/fpu=false}
},
		xlabel={$s$},
		ylabel={relative error},
		legend style={legend cell align=left}, legend pos=south west,
		legend columns = 1,
		]
		\addplot[black,mark=square,mark size=2.5pt,mark options={solid}] coordinates {
(100,3.969e-01) (200,1.870e-01) (500,5.690e-02) (1000,1.240e-02)
		};
	    \addlegendentry{$m_1=3s$, $\kappa=20$}
		\addplot[blue,mark=triangle,mark size=2.5pt,mark options={solid}] coordinates {
(100,3.504e-01) (200,1.856e-01) (500,5.683e-02) (1000,1.238e-02)
		};
	    \addlegendentry{$m_1=8s$, $\kappa=10$}
		\end{axis}
		\end{tikzpicture}
	}
	\caption{Number of samples, runtime, number of iterations, $L^2(\mathbbm{T}^{10})$ error
		vs.\ sparsity $s\in\{100,200,500,1000\}$ for Fourier basis and test function~\eqref{equ:f:10}.}
	\label{fig:approx_sparse:fourier_d10:samples_runtime_iter_success_vs_s}
\end{figure}
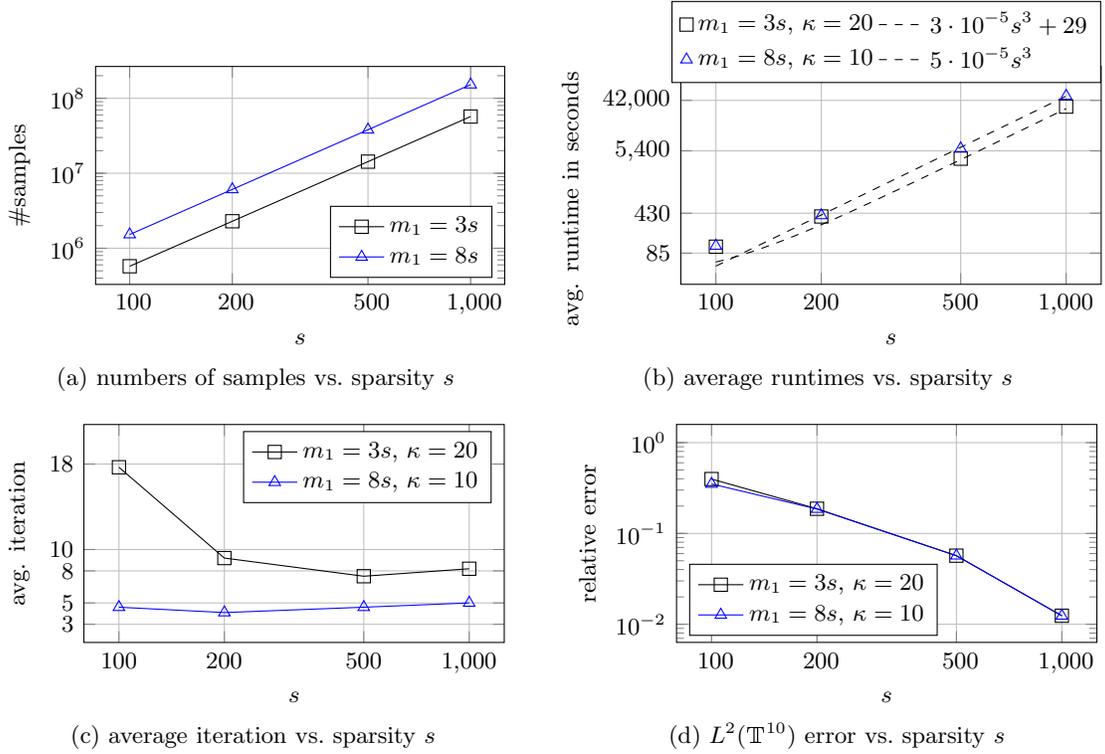

\FloatBarrier

\subsubsection{Chebyshev and Legendre type with $D=7$}

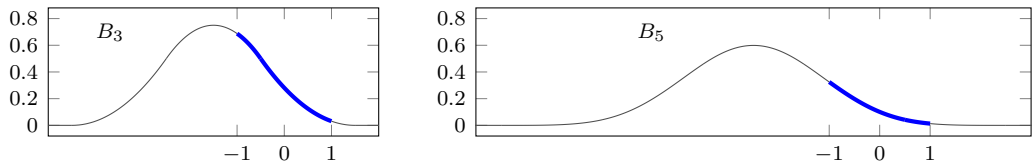
\begin{figure}[!h]
	\centering
	\begin{tikzpicture}[font=\scriptsize]
	\begin{axis}[width=0.4\textwidth,height=0.22\textwidth,enlarge x limits=false,enlarge y limits=true,xmin=-5,xmax=2,ymin=0,ymax=0.8,xtick={-1,0,1} 
	]
	\addplot[samples=100,black!70,domain=-5:-4.5] {0};
	\addplot[samples=100,black!70,domain=-4.5:-2.5] {x*x/8+9*x/8+81/32};
	\addplot[samples=100,black!70,domain=-2.5:-1] {-x*x/4-3*x/4+3/16};
	\addplot[samples=100,blue,domain=-1:-0.5,ultra thick] {-x*x/4-3*x/4+3/16};
	\addplot[samples=100,blue,domain=-0.5:1,ultra thick] {x*x/8-3*x/8+9/32};
	\addplot[samples=100,black!70,domain=1:1.5] {x*x/8-3*x/8+9/32};
	\addplot[samples=100,black!70,domain=1.5:10] {0};
	\node [right] at (axis cs:-4.2,0.7) {$B_3$};
	\end{axis}
	\end{tikzpicture}
	\hspace{1em}
	\begin{tikzpicture}[font=\scriptsize]
	\begin{axis}[width=0.6\textwidth,height=0.22\textwidth,enlarge x limits=false,enlarge y limits=true,xmin=-8,xmax=3,ymin=0,ymax=0.8,xtick={-1,0,1} 
	]
	\addplot[samples=100,black!70,domain=-10:-7.5] {0};
	\addplot[samples=100,black!70,domain=-7.5:-5.5] {(2*x+15)^4/6144};
	\addplot[samples=100,black!70,domain=-5.5:-3.5] {-5645/1536-205/48*x-95/64*x*x-5/24*x^3-1/96*x^4};
	\addplot[samples=100,black!70,domain=-3.5:-1.5] {715/3072+25/128*x+55/128*x*x+5/32*x^3+1/64*x^4};
	\addplot[samples=100,black!70,domain=-1.5:-1] {155/1536-5/32*x+5/64*x*x-1/96*x^4};
	\addplot[samples=100,blue,domain=-1:0.5,ultra thick] {155/1536-5/32*x+5/64*x*x-1/96*x^4};
	\addplot[samples=100,blue,domain=0.5:1,ultra thick] {(2*x-5)^4/6144};
	\addplot[samples=100,black!70,domain=1:2.5] {(2*x-5)^4/6144};
	\addplot[samples=100,black!70,domain=2.5:10] {0};
	\node [right] at (axis cs:-5,0.7) {$B_5$};
	\end{axis}
	\end{tikzpicture}
	\vspace{-0.6em}
	\caption{B-splines $B_3$ and $B_5$ considered in interval $[-1,1]$.}
	\label{fig:Bspline:nonper}
\end{figure}

Next, we apply Algorithm~\ref{alg:main} on the 7-variate test function $f\colon [-1,1]^{7}\rightarrow\mathbbm{R}$,
\begin{equation} \label{equ:f_cheb:7}
f(\xib):=\prod_{t\in\{0,2,5\}}B_3(\xi_t) + \prod_{t\in\{1,3,4,6\}}B_5(\xi_t)
\vspace{-0.5em}
\end{equation}
similar as in \cite{potts2017multivariate},
where
$B_3\colon\mathbbm{R}\rightarrow\mathbbm{R}$ is a shifted, scaled and dilated B-spline of order~3
and
$B_5\colon\mathbbm{R}\rightarrow\mathbbm{R}$ is a shifted, scaled and dilated B-spline of order~5, see Figure~\ref{fig:Bspline:nonper} for illustration.
We remark that the absolute values of the Chebyshev coefficients $c_n$, $n\in\mathbbm{N}_0$, 
of $B_3$ and $B_5$ decay like $\sim n^{-3}$ and $\sim n^{-5}$, respectively.
The obtained basis index sets $\tilde{\Omega}$ should ``consist of'' the union of two lower dimensional manifolds,
a three-dimensional hyperbolic cross in the dimensions $0,2,5$; and 
a four-dimensional hyperbolic cross in the dimensions $1,3,4,6$. 
All tests are performed 10 times and the relative $L^2([-1,1]^7,\mu_\mathrm{C})$ approximation error
$\Vert f-a\Vert_{L^2([-1,1]^7,\mu_\mathrm{C})} / \Vert f\Vert_{L^2([-1,1]^7,\mu_\mathrm{C})}$
is computed each time,
where the approximant $a:=\sum_{\n\in \tilde{\Omega}} a_{\n} \,T_\n$, $T_\n$ is the Chebyshev product basis, and $\mu_\mathrm{C}(\xib):=\pi^{-D} \prod_{j\in [D]} (1-\xi_j^2)^{-1/2}$ is the Chebyshev product measure.

We set the parameters $N=64$, $d=D=7$, $m_2=4s$, and we always use $m_{\rm CE}:=50\,s$ samples for the coefficient estimation where $s=|\S|$.
We consider two different parameter combinations: $m_1=4s$ and $\kappa=20$, as well as $m_1=8s$ and $\kappa=10$. The obtained results, i.e., the numbers of samples, average runtimes, average iterations, and relative $L^2([-1,1]^7,\mu_\mathrm{C})$ errors are plotted as a function of the sparsity~$s\in\{25,50,100,200,500\}$ in Figure~\ref{fig:approx_sparse:chebyshev_d7:samples_runtime_iter_success_vs_s}. Due to the parameter choices for $m_1$ and $m_2$, we observe that the numbers of samples grow quadratically for increasing sparsity~$s$. The average runtimes grow approximately like $\sim s^3\min\{s,N\}$ and this means $\sim s^3$ for fixed $N$.
Moreover, the average numbers of iterations are well below its imposed maximum $\kappa$ for $m_1=4s$ and $\kappa=20$ as well as close to $\kappa$ for $m_1=8s$ and $\kappa=10$. The relative $L^2([-1,1]^7,\mu_\mathrm{C})$ errors decrease for increasing sparsity $s$ having a value of approximately $2.3\cdot 10^{-4}$ for sparsity $s=500$.
We emphasize that Algorithm~\ref{alg:main} is able to easily determine the $s=500$ approximately largest basis coefficients and the corresponding basis indices for our test function out of $| \mathcal{I}_{N,d}| = N^d = 64^{7} \approx 4.4\cdot 10^{12}$ allowed indices.

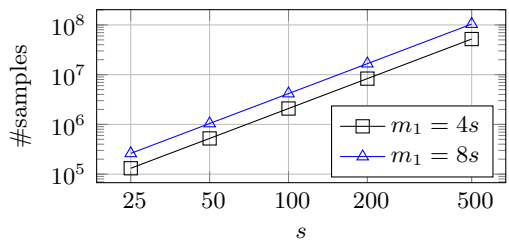
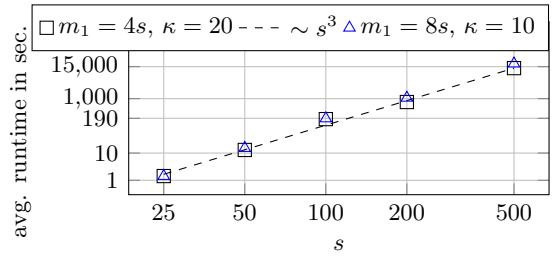
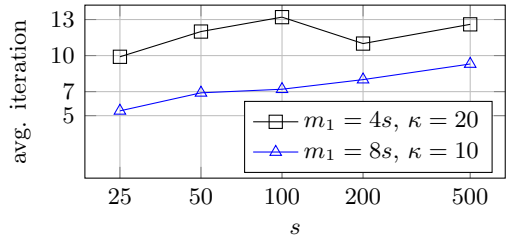
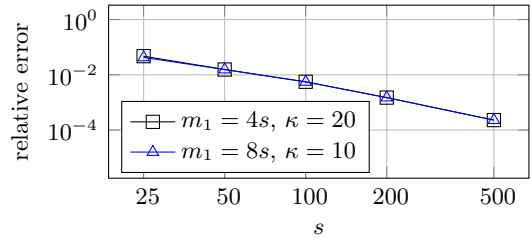
\begin{figure}[!h]
	\subfloat[number of samples vs.\ sparsity $s$]{
		\begin{tikzpicture}[baseline]
		\begin{axis}[
		font=\footnotesize,
		enlarge x limits=true,
		enlarge y limits=true,
		height=0.26\textwidth,
		grid=major,
		width=0.47\textwidth,
		xtick={25,50,100,200,500},
		xmode=log,
		ymode=log,
		xticklabel={
			\pgfkeys{/pgf/fpu=true}
			\pgfmathparse{exp(\tick)}%
			\pgfmathprintnumber[fixed relative, precision=3]{\pgfmathresult}
			\pgfkeys{/pgf/fpu=false}
		},
		xlabel={$s$},
		ylabel={\#samples},
		legend style={legend cell align=left}, legend pos=south east,
		legend columns = 1,
		]
		\addplot[black,mark=square,mark size=2.5pt,mark options={solid}] coordinates {
        	(25,131250) (50,522500) (100,2085000) (200,8330000) (500,52025000)
        };
        \addlegendentry{$m_1=4s$}
		\addplot[blue,mark=triangle,mark size=2.5pt,mark options={solid}] coordinates {
			(25,261250) (50,1042500) (100,4165000) (200,16650000) (500,104025000)
		};
		\addlegendentry{$m_1=8s$}
		\end{axis}
		\end{tikzpicture}
	}
	\hfill
	\subfloat[average runtime vs.\ sparsity $s$]{
		\begin{tikzpicture}[baseline]
		\begin{axis}[
		font=\footnotesize,
		enlarge x limits=true,
		enlarge y limits=true,
		height=0.26\textwidth,
		grid=major,
		width=0.48\textwidth,
		xtick={25,50,100,200,500},
		ytick={1,100,20000},
		ymin=1,
		ymax=2e5,
		ytick={1,10,190,1000,15000},
		xmode=log,
		ymode=log,
		log ticks with fixed point,
		xlabel={$s$},
		ylabel={avg. runtime in sec.},
		y label style={xshift=-0.8em},
		legend style={legend cell align=left,at={(1,1.1)}},
		legend columns = -1,
		]
		\addplot[black,mark=square,mark size=2.5pt,mark options={solid},only marks] coordinates {
(25,1.428e+00) (50,1.301e+01) (100,1.772e+02) (200,7.521e+02) (500,1.327e+04)
        };
        \addlegendentry{$m_1=4s$, $\kappa=20$}
		\addplot [black,domain=25:500, samples=100, dashed]{1.06e-4*x*x*x};
\addlegendentry{$\sim s^3$}
		\addplot[blue,mark=triangle,mark size=2.5pt,mark options={solid},only marks] coordinates {
(25,1.422e+00) (50,1.523e+01) (100,1.867e+02) (200,1.072e+03) (500,1.935e+04)
		};
		\addlegendentry{$m_1=8s$, $\kappa=10$}
		\end{axis}
		\end{tikzpicture}
	}
	\\
	\subfloat[average iteration vs.\ sparsity $s$]{
		\begin{tikzpicture}[baseline]
		\begin{axis}[
		font=\footnotesize,
		enlarge x limits=true,
		enlarge y limits=true,
		height=0.26\textwidth,
		grid=major,
		width=0.48\textwidth,
		xtick={25,50,100,200,500},
		ytick={5,7,10,13},
		ymin=1,ymax=13,
		xmode=log,
		log ticks with fixed point,
		xlabel={$s$},
		ylabel={avg. iteration},
		legend style={legend cell align=left}, legend pos=south east,
		legend columns = 1,
		]
		\addplot[black,mark=square,mark size=2.5pt,mark options={solid}] coordinates {
(25,9.900e+00) (50,1.200e+01) (100,1.320e+01) (200,1.100e+01) (500,1.260e+01)
        };
        \addlegendentry{$m_1=4s$, $\kappa=20$}
		\addplot[blue,mark=triangle,mark size=2.5pt,mark options={solid}] coordinates {
(25,5.400e+00) (50,6.900e+00) (100,7.200e+00) (200,8.000e+00) (500,9.300e+00)
		};
		\addlegendentry{$m_1=8s$, $\kappa=10$}
		\end{axis}
		\end{tikzpicture}
	}
	\hfill
	\subfloat[relative {$L^2([-1,1]^7,\mu_\mathrm{C})$} error vs.\ sparsity $s$]{
		\begin{tikzpicture}[baseline]
		\begin{axis}[
		font=\footnotesize,
		enlarge x limits=true,
		enlarge y limits=true,
		height=0.26\textwidth,
		grid=major,
		width=0.48\textwidth,
		xtick={25,50,100,200,500},
		ymin=6e-6,ymax=1,
		xmode=log,
		ymode=log,
		xticklabel={
			\pgfkeys{/pgf/fpu=true}
			\pgfmathparse{exp(\tick)}%
			\pgfmathprintnumber[fixed relative, precision=3]{\pgfmathresult}
			\pgfkeys{/pgf/fpu=false}
		},
		xlabel={$s$},
		ylabel={relative error},
		legend style={legend cell align=left}, legend pos=south west,
		legend columns = 1,
		]
		\addplot[black,mark=square,mark size=2.5pt,mark options={solid}] coordinates {
(25,4.726e-02) (50,1.556e-02) (100,5.592e-03) (200,1.495e-03) (500,2.287e-04)
        };
        \addlegendentry{$m_1=4s$, $\kappa=20$}
		\addplot[blue,mark=triangle,mark size=2.5pt,mark options={solid}] coordinates {
(25,4.242e-02) (50,1.549e-02) (100,5.514e-03) (200,1.493e-03) (500,2.279e-04)
		};
		\addlegendentry{$m_1=8s$, $\kappa=10$}
		\end{axis}
		\end{tikzpicture}
	}
	\caption{Number of samples, runtime, number of iterations, relative $L^2([-1,1]^7,\mu_\mathrm{C})$ error
		vs.\ sparsity $s\in\{25,50,100,200,500\}$ for Chebyshev basis and test function~\eqref{equ:f_cheb:7}.}
	\label{fig:approx_sparse:chebyshev_d7:samples_runtime_iter_success_vs_s}
\end{figure}

In addition, we use the preconditioned Legendre polynomials $Q_n$ from Section~\ref{sec:Numerics:exactly_sparse:mixed1} as basis functions, i.e.\ $T_\n$ is now the Legendre product basis in the approximant $a:=\sum_{\n\in \tilde{\Omega}} a_{\n} \,T_\n$. Besides that, we keep all parameters identical but determine the relative $L^2([-1,1]^7,\mu_\mathrm{L})$ approximation error
$\Vert f-a\Vert_{L^2([-1,1]^7,\mu_\mathrm{L})} / \Vert f\Vert_{L^2([-1,1]^7,\mu_\mathrm{L})}$, which corresponds to the Legendre basis and uses the probability measure $\mu_\mathrm{L} \equiv 2^{-D}$. The results are shown in Figure~\ref{fig:approx_sparse:legendre_d7:samples_runtime_iter_success_vs_s}. Here, we observe that the numbers of iterations are higher by up to $\approx$ 50\% compared to the Chebyshev case in Figure~\ref{fig:approx_sparse:chebyshev_d7:samples_runtime_iter_success_vs_s}, and that they reach the imposed maximum of $\kappa:=20$ for $m_1=4s$ and $\kappa:=10$ for $m_2=8s$ in several cases. Correspondingly, the runtimes are also higher by up to $\approx$ 50\%. The obtained relative $L^2([-1,1]^7,\mu_\mathrm{L})$ errors are similar, but we also remark that we cannot compare these errors directly to the relative $L^2([-1,1]^7,\mu_\mathrm{C})$ errors of the Chebyshev basis since they are measured with respect to different probability measures, $\mu_\mathrm{C}(\xib):= \pi^{-D} \prod_{j\in [D]} (1-\xi_j^2)^{-1/2}$ for Chebyshev and $\mu_\mathrm{L} \equiv 2^{-D}$ for Legendre.

\begin{figure}[!h]
	\subfloat[number of samples vs.\ sparsity $s$]{
		\begin{tikzpicture}[baseline]
		\begin{axis}[
		font=\footnotesize,
		enlarge x limits=true,
		enlarge y limits=true,
		height=0.26\textwidth,
		grid=major,
		width=0.47\textwidth,
		xtick={25,50,100,200,500},
		xmode=log,
		ymode=log,
		xticklabel={
			\pgfkeys{/pgf/fpu=true}
			\pgfmathparse{exp(\tick)}%
			\pgfmathprintnumber[fixed relative, precision=3]{\pgfmathresult}
			\pgfkeys{/pgf/fpu=false}
		},
		xlabel={$s$},
		ylabel={\#samples},
		legend style={legend cell align=left}, legend pos=south east,
		legend columns = 1,
		]
		\addplot[black,mark=square,mark size=2.5pt,mark options={solid}] coordinates {
	(25,131250) (50,522500) (100,2085000) (200,8330000) (500,52025000)
};
\addlegendentry{$m_1=4s$}
\addplot[blue,mark=triangle,mark size=2.5pt,mark options={solid}] coordinates {
	(25,261250) (50,1042500) (100,4165000) (200,16650000) (500,104025000)
};
\addlegendentry{$m_1=8s$}
		\end{axis}
		\end{tikzpicture}
	}
	\hfill
	\subfloat[average runtime vs.\ sparsity $s$]{
		\begin{tikzpicture}[baseline]
		\begin{axis}[
		font=\footnotesize,
		enlarge x limits=true,
		enlarge y limits=true,
		height=0.26\textwidth,
		grid=major,
		width=0.48\textwidth,
		xtick={25,50,100,200,500},
		ytick={10,1000,30000},
		ymax=2e5,
		xmode=log,
		ymode=log,
		log ticks with fixed point,
		xlabel={$s$},
		ylabel={avg. runtime in sec.},
		y label style={xshift=-0.8em},
		legend style={legend cell align=left}, legend pos=north west,
		legend columns = 1,
		]
		\addplot[black,mark=square,mark size=2.5pt,mark options={solid}] coordinates {
(25,7.210e+00) (50,5.395e+01) (100,3.555e+02) (200,1.931e+03) (500,2.612e+04)
		};
		\addlegendentry{$m_1=4s$, $\kappa=20$}
		\addplot[blue,mark=triangle,mark size=2.5pt,mark options={solid}] coordinates {
(25,7.130e+00) (50,5.524e+01) (100,3.171e+02) (200,1.971e+03) (500,2.727e+04)
		};
		\addlegendentry{$m_1=8s$, $\kappa=10$}
		\end{axis}
		\end{tikzpicture}
	}
	\\
	\subfloat[average iteration vs.\ sparsity $s$]{
		\begin{tikzpicture}[baseline]
		\begin{axis}[
		font=\footnotesize,
		enlarge x limits=true,
		enlarge y limits=true,
		height=0.26\textwidth,
		grid=major,
		width=0.48\textwidth,
		xtick={25,50,100,200,500},
		ytick={5,10,15,20},
		ymin=-1,ymax=20,
		xmode=log,
		log ticks with fixed point,
		xlabel={$s$},
		ylabel={avg. iteration},
		legend style={legend cell align=left}, legend pos=south east,
		legend columns = 1,
		]
		\addplot[black,mark=square,mark size=2.5pt,mark options={solid}] coordinates {
(25,1.020e+01) (50,1.610e+01) (100,1.840e+01) (200,1.820e+01) (500,1.940e+01)
		};
		\addlegendentry{$m_1=4s$, $\kappa=20$}
		\addplot[blue,mark=triangle,mark size=2.5pt,mark options={solid}] coordinates {
(25,8.300e+00) (50,9.300e+00) (100,9.400e+00) (200,9.200e+00) (500,1.000e+01)
		};
		\addlegendentry{$m_1=8s$, $\kappa=10$}
		\end{axis}
		\end{tikzpicture}
	}
	\hfill
	\subfloat[relative {$L^2([-1,1]^7,\mu_\mathrm{L})$} error vs.\ sparsity $s$]{
		\begin{tikzpicture}[baseline]
		\begin{axis}[
		font=\footnotesize,
		enlarge x limits=true,
		enlarge y limits=true,
		height=0.26\textwidth,
		grid=major,
		width=0.48\textwidth,
		xtick={25,50,100,200,500},
		ytick={1,1e-2,1e-4},
		ymin=6e-6,ymax=1,
		xmode=log,
		ymode=log,
		xticklabel={
			\pgfkeys{/pgf/fpu=true}
			\pgfmathparse{exp(\tick)}%
			\pgfmathprintnumber[fixed relative, precision=3]{\pgfmathresult}
			\pgfkeys{/pgf/fpu=false}
		},
		xlabel={$s$},
		ylabel={relative error},
		legend style={legend cell align=left}, legend pos=south west,
		legend columns = 1,
		]
		\addplot[black,mark=square,mark size=2.5pt,mark options={solid}] coordinates {
(25,7.240e-02) (50,2.124e-02) (100,5.603e-03) (200,1.460e-03) (500,2.178e-04)
		};
		\addlegendentry{$m_1=4s$, $\kappa=20$}
		\addplot[blue,mark=triangle,mark size=2.5pt,mark options={solid}] coordinates {
(25,3.757e-02) (50,1.394e-02) (100,4.879e-03) (200,1.421e-03) (500,2.109e-04)
		};
		\addlegendentry{$m_1=8s$, $\kappa=10$}
		\end{axis}
		\end{tikzpicture}
	}
	\caption{
		Number of samples, runtime, number of iterations, relative $L^2([-1,1]^7,\mu_\mathrm{L})$ error
		vs.\ sparsity $s\in\{25,50,100,200,500\}$ for Legendre basis and test function~\eqref{equ:f_cheb:7}.
	}
	\label{fig:approx_sparse:legendre_d7:samples_runtime_iter_success_vs_s}
\end{figure}
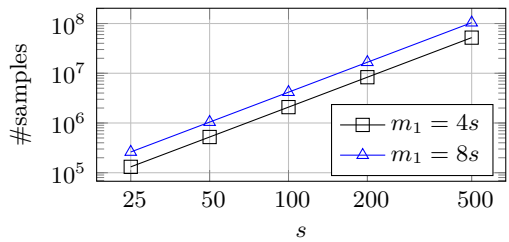
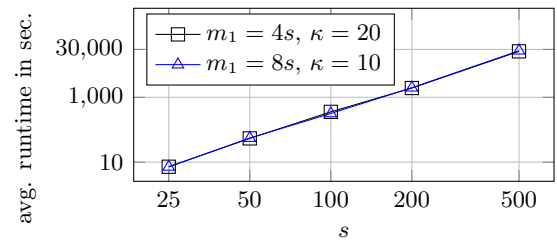
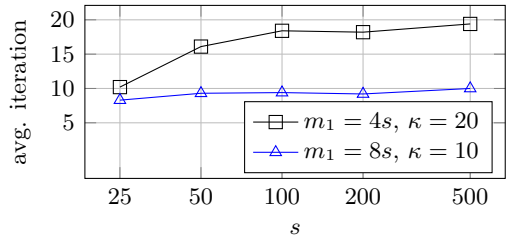
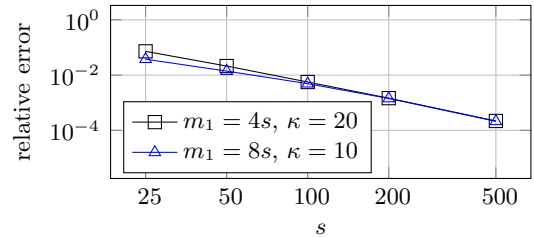

\FloatBarrier

\subsubsection{Mixed type with $D=10$}

\begin{sloppypar}
Finally, we combine parts of the test functions from the previous two subsections. We consider the 10-variate test function
$f\colon \tilde{\D}\rightarrow\mathbbm{R}$, $\tilde{\D}:=[-1,1]\times\mathbbm{T}\times[-1,1]^3\times\mathbbm{T}^2\times[-1,1]\times\mathbbm{T}^2$,
\begin{equation} \label{equ:f_mixed:10}
f(\xib):=B_3(\xi_0) B_3(\xi_2) N_4(\xi_8) + B_5(\xi_3) B_5(\xi_4) N_2(\xi_1) N_2(\xi_6) + B_3(\xi_7) N_2(\xi_5) N_2(\xi_9)
\vspace{-0.5em}
\end{equation}
In spatial dimensions $j=0,2,3,4,7$, we use Chebyshev basis functions as well as Fourier basis functions in the remaining spatial dimensions $j=1,5,6,8,9$.
All tests are performed 10 times and the relative $L^2(\tilde{\D},\mu_\mathrm{F,C})$ approximation error
$\Vert f-a\Vert_{L^2(\tilde{\D},\mu_\mathrm{F,C})} / \Vert f\Vert_{L^2(\tilde{\D},\mu_\mathrm{F,C})}$
is computed each time, 
where the approximant $a:=\sum_{\n\in \tilde{\Omega}} a_{\n} \,T_\n$, $$T_\n(\xib):=\left(\prod_{j\in\{0,2,3,4,7\}} \cos(n_j\arccos{\xi_j})\right)\,\left(\prod_{j\in\{1,5,6,8,9\}} \mathrm{e}^{2\pi\mathrm{i}n_j\xi_j}\right)$$ is the mixed product basis, and $\mu_\mathrm{F,C}(\xib):= \left(\frac{2}{\pi}\right)^5 \prod_{j\in \{0,2,3,4,7\}} (1-\xi_j^2)^{-1/2}$ is the corresponding probability measure.
\end{sloppypar}

Here we set the parameters $N=64$, $d=D=10$, $m_2=4s$, and we always use $m_{\rm CE}:=50\,s$ samples for the coefficient estimation where $s=|\S|$. 
We consider two different parameter combinations from the previous subsection: $m_1=4s$ and $\kappa=20$, as well as $m_1=8s$ and $\kappa=10$. The obtained results, i.e., the numbers of samples, average runtimes, average iterations, and relative $L^2(\tilde{\D},\mu_\mathrm{F,C})$ errors are plotted as a function of the sparsity~$s\in\{25,50,100,200\}$ in Figure~\ref{fig:approx_sparse:mixed_d10:samples_runtime_iter_success_vs_s}. As before, the numbers of samples grow quadratically for increasing sparsity~$s$. The average runtimes grow approximately like $\sim s^3\min\{s,N\}$ and this means $\sim s^3$ for fixed $N$.
Moreover, the average numbers of iterations are well below its imposed maximum $\kappa$. The relative $L^2(\tilde{\D},\mu_\mathrm{F,C})$ errors decrease for increasing sparsity $s$ having a value of approximately $4.9\cdot 10^{-3}$ for sparsity $s=500$.
We emphasize that Algorithm~\ref{alg:main} is able to easily determine the $s=500$ approximately largest basis coefficients and the corresponding indices for our test function out of $| \mathcal{I}_{N,d}| = N^d = 64^{10} \approx 10^{18}$ possible indices.

\begin{figure}[!h]
	\subfloat[number of samples vs.\ sparsity $s$]{
		\begin{tikzpicture}[baseline]
		\begin{axis}[
		font=\footnotesize,
		enlarge x limits=true,
		enlarge y limits=true,
		height=0.26\textwidth,
		grid=major,
		width=0.47\textwidth,
		xtick={25,50,100,200,500},
		xmode=log,
		ymode=log,
		xticklabel={
			\pgfkeys{/pgf/fpu=true}
			\pgfmathparse{exp(\tick)}%
			\pgfmathprintnumber[fixed relative, precision=3]{\pgfmathresult}
			\pgfkeys{/pgf/fpu=false}
		},
		xlabel={$s$},
		ylabel={\#samples},
		legend style={legend cell align=left}, legend pos=south east,
		legend columns = 1,
		]
		\addplot[black,mark=square,mark size=2.5pt,mark options={solid}] coordinates {
(25,1.912e+05) (50,7.625e+05) (100,3.045e+06) (200,1.217e+07) (500,7.602e+07)
		};
		\addlegendentry{$m_1=4s$}
		\addplot[blue,mark=triangle,mark size=2.5pt,mark options={solid}] coordinates {
(25,3.812e+05) (50,1.522e+06) (100,6.085e+06) (200,2.433e+07) (500,1.520e+08)
		};
		\addlegendentry{$m_1=8s$}
		\end{axis}
		\end{tikzpicture}
	}
	\hfill
	\subfloat[average runtime vs.\ sparsity $s$]{
		\begin{tikzpicture}[baseline]
		\begin{axis}[
		font=\footnotesize,
		enlarge x limits=true,
		enlarge y limits=true,
		height=0.26\textwidth,
		grid=major,
		width=0.48\textwidth,
		xtick={25,50,100,200,500},
		ytick={1,100,10000},
		ymin=1,
		ymax=9e5,
		ytick={2,25,400,2000,25000},
		xmode=log,
		ymode=log,
		log ticks with fixed point,
		xlabel={$s$},
		ylabel={avg. runtime in sec.},
		y label style={xshift=-0.8em},
		legend style={legend cell align=left}, legend pos=north west,
		legend columns = 1,
		]
		\addplot[black,mark=square,mark size=2.5pt,mark options={solid}] coordinates {
(25,2.132e+00) (50,2.788e+01) (100,3.105e+02) (200,1.449e+03) (500,2.258e+04)
		};
		\addlegendentry{$m_1=4s$, $\kappa=20$}
		\addplot[blue,mark=triangle,mark size=2.5pt,mark options={solid}] coordinates {
(25,3.173e+00) (50,4.525e+01) (100,4.079e+02) (200,1.874e+03) (500,3.065e+04)
		};
		\addlegendentry{$m_1=8s$, $\kappa=10$}
		\end{axis}
		\end{tikzpicture}
	}
	\\
	\subfloat[average iteration vs.\ sparsity $s$]{
		\begin{tikzpicture}[baseline]
		\begin{axis}[
		font=\footnotesize,
		enlarge x limits=true,
		enlarge y limits=true,
		height=0.26\textwidth,
		grid=major,
		width=0.48\textwidth,
		xtick={25,50,100,200,500},
		ytick={3,5,7,9},
		ymin=2,ymax=9,
		xmode=log,
		log ticks with fixed point,
		xlabel={$s$},
		ylabel={avg. iteration},
		legend style={legend cell align=left}, legend pos=south east,
		legend columns = 1,
		]
		\addplot[black,mark=square,mark size=2.5pt,mark options={solid}] coordinates {
(25,5.800e+00) (50,8.100e+00) (100,8.200e+00) (200,8.400e+00) (500,9.200e+00)
		};
		\addlegendentry{$m_1=4s$, $\kappa=20$}
		\addplot[blue,mark=triangle,mark size=2.5pt,mark options={solid}] coordinates {
(25,4.200e+00) (50,5.200e+00) (100,5.600e+00) (200,5.600e+00) (500,6.200e+00)
		};
		\addlegendentry{$m_1=8s$, $\kappa=10$}
		\end{axis}
		\end{tikzpicture}
	}
	\hfill
	\subfloat[relative $L^2(\tilde{\D},\mu_\mathrm{F,C})$ error vs.\ sparsity $s$]{
		\begin{tikzpicture}[baseline]
		\begin{axis}[
		font=\footnotesize,
		enlarge x limits=true,
		enlarge y limits=true,
		height=0.26\textwidth,
		grid=major,
		width=0.48\textwidth,
		xtick={25,50,100,200,500},
		ymin=5e-4,ymax=1,
		xmode=log,
		ymode=log,
		xticklabel={
			\pgfkeys{/pgf/fpu=true}
			\pgfmathparse{exp(\tick)}%
			\pgfmathprintnumber[fixed relative, precision=3]{\pgfmathresult}
			\pgfkeys{/pgf/fpu=false}
		},
		xlabel={$s$},
		ylabel={relative error},
		legend style={legend cell align=left}, legend pos=south west,
		legend columns = 1,
		]
		\addplot[black,mark=square,mark size=2.5pt,mark options={solid}] coordinates {
(25,1.871e-01) (50,8.152e-02) (100,4.288e-02) (200,1.689e-02) (500,4.923e-03)
		};
		\addlegendentry{$m_1=4s$, $\kappa=20$}
		\addplot[blue,mark=triangle,mark size=2.5pt,mark options={solid}] coordinates {
(25,1.750e-01) (50,8.123e-02) (100,4.275e-02) (200,1.683e-02) (500,4.914e-03)
		};
		\addlegendentry{$m_1=8s$, $\kappa=10$}
		\end{axis}
		\end{tikzpicture}
	}
	\caption{Number of samples, runtime, number of iterations, relative $L^2(\tilde{\D},\mu_\mathrm{F,C})$ error
		vs.\ sparsity $s\in\{25,50,100,200\}$ for mixed Fourier+Chebyshev basis and test function~\eqref{equ:f_mixed:10}.}
	\label{fig:approx_sparse:mixed_d10:samples_runtime_iter_success_vs_s}
\end{figure}
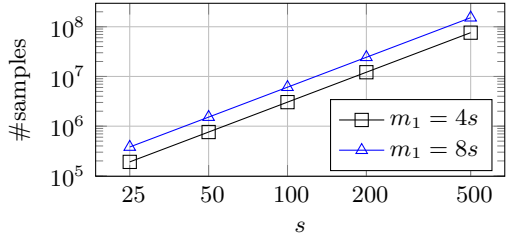
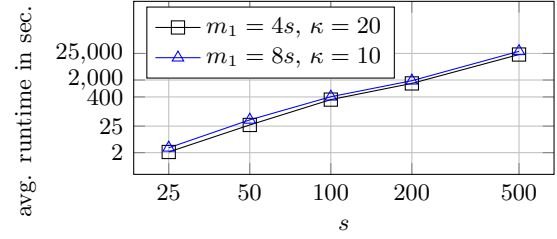
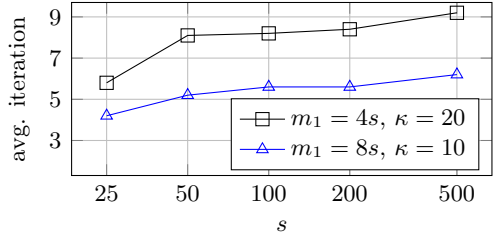
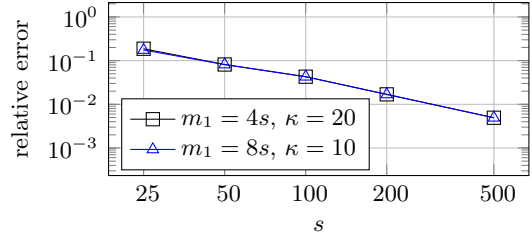

Additionally, we use the preconditioned Legendre polynomials $Q_n$ from Section~\ref{sec:Numerics:exactly_sparse:mixed1} as basis functions in the spatial dimensions $j=0,2,3,4,7$ instead of the Chebyshev polynomials. Besides that, we keep all parameters identical but determine now the relative $L^2(\tilde{\D},\mu_\mathrm{F,L})$ approximation error with respect to the probability measure $\mu_{F,L}(\xib)\equiv 2^{-5}$ which corresponds to the current choice of bases. The results are presented in Figure~\ref{fig:approx_sparse:mixed_l_d10:samples_runtime_iter_success_vs_s}. As before, we observe that the numbers of iterations are higher, now by up to $\approx$ 100\% compared to using Chebyshev polynomials in Figure~\ref{fig:approx_sparse:mixed_l_d10:samples_runtime_iter_success_vs_s}. Correspondingly, the runtimes also double in some cases. The obtained relative errors are similar, but we again remark that we cannot compare these errors directly since they are measured with respect to different probability measures.

\begin{figure}[!h]
	\subfloat[number of samples vs.\ sparsity $s$]{
		\begin{tikzpicture}[baseline]
		\begin{axis}[
		font=\footnotesize,
		enlarge x limits=true,
		enlarge y limits=true,
		height=0.26\textwidth,
		grid=major,
		width=0.47\textwidth,
		xtick={25,50,100,200,500},
		xmode=log,
		ymode=log,
		xticklabel={
			\pgfkeys{/pgf/fpu=true}
			\pgfmathparse{exp(\tick)}%
			\pgfmathprintnumber[fixed relative, precision=3]{\pgfmathresult}
			\pgfkeys{/pgf/fpu=false}
		},
		xlabel={$s$},
		ylabel={\#samples},
		legend style={legend cell align=left}, legend pos=south east,
		legend columns = 1,
		]
		\addplot[black,mark=square,mark size=2.5pt,mark options={solid}] coordinates {
(25,1.912e+05) (50,7.625e+05) (100,3.045e+06) (200,1.217e+07) (500,7.602e+07)
		};
		\addlegendentry{$m_1=4s$}
		\addplot[blue,mark=triangle,mark size=2.5pt,mark options={solid}] coordinates {
(25,3.812e+05) (50,1.522e+06) (100,6.085e+06) (200,2.433e+07) (500,1.520e+08)
		};
		\addlegendentry{$m_1=8s$}
		\end{axis}
		\end{tikzpicture}
	}
	\hfill
	\subfloat[average runtime vs.\ sparsity $s$]{
		\begin{tikzpicture}[baseline]
		\begin{axis}[
		font=\footnotesize,
		enlarge x limits=true,
		enlarge y limits=true,
		height=0.26\textwidth,
		grid=major,
		width=0.48\textwidth,
		xtick={25,50,100,200,500},
		xmax=500,
		ymin=4,
		ymax=8e5,
		ytick={4,40,500,2500,4e4},
		xmode=log,
		ymode=log,
		log ticks with fixed point,
		xlabel={$s$},
		ylabel={avg. runtime in sec.},
		y label style={xshift=-0.8em},
		legend style={legend cell align=left}, legend pos=north west,
		legend columns = 1,
		]
		\addplot[black,mark=square,mark size=2.5pt,mark options={solid}] coordinates {
(25,4.505e+00) (50,3.624e+01) (100,5.236e+02) (200,2.404e+03) (500,3.490e+04)
		};
		\addlegendentry{$m_1=4s$, $\kappa=20$}
		\addplot[blue,mark=triangle,mark size=2.5pt,mark options={solid}] coordinates {
(25,4.725e+00) (50,5.789e+01) (100,5.122e+02) (200,2.905e+03) (500,4.061e+04)
		};
		\addlegendentry{$m_1=8s$, $\kappa=10$}
		\end{axis}
		\end{tikzpicture}
	}
	\\
	\subfloat[average iteration vs.\ sparsity $s$]{
		\begin{tikzpicture}[baseline]
		\begin{axis}[
		font=\footnotesize,
		enlarge x limits=true,
		enlarge y limits=true,
		height=0.26\textwidth,
		grid=major,
		width=0.48\textwidth,
		xtick={25,50,100,200,500},
		ytick={5,10,15,20},
		ymin=-1,ymax=15,
		xmode=log,
		log ticks with fixed point,
		xlabel={$s$},
		ylabel={avg. iteration},
		legend style={legend cell align=left}, legend pos=south east,
		legend columns = 1,
		]
		\addplot[black,mark=square,mark size=2.5pt,mark options={solid}] coordinates {
(25,8.700e+00) (50,1.010e+01) (100,1.450e+01) (200,1.320e+01) (500,1.380e+01)
		};
		\addlegendentry{$m_1=4s$, $\kappa=20$}
        \addplot[blue,mark=triangle,mark size=2.5pt,mark options={solid}] coordinates {
(25,5.300e+00) (50,6.700e+00) (100,7.400e+00) (200,8.200e+00) (500,8.100e+00)
        };
        \addlegendentry{$m_1=8s$, $\kappa=10$}
		\end{axis}
		\end{tikzpicture}
	}
	\hfill
	\subfloat[relative $L^2(\tilde{\D},\mu_\mathrm{F,L})$ error vs.\ sparsity $s$]{
		\begin{tikzpicture}[baseline]
		\begin{axis}[
		font=\footnotesize,
		enlarge x limits=true,
		enlarge y limits=true,
		height=0.26\textwidth,
		grid=major,
		width=0.48\textwidth,
		xtick={25,50,100,200,500},
		ymin=5e-4,ymax=1,
		xmode=log,
		ymode=log,
		xticklabel={
			\pgfkeys{/pgf/fpu=true}
			\pgfmathparse{exp(\tick)}%
			\pgfmathprintnumber[fixed relative, precision=3]{\pgfmathresult}
			\pgfkeys{/pgf/fpu=false}
		},
		xlabel={$s$},
		ylabel={relative error},
		legend style={legend cell align=left}, legend pos=south west,
		legend columns = 1,
		]
		\addplot[black,mark=square,mark size=2.5pt,mark options={solid}] coordinates {
(25,1.639e-01) (50,8.973e-02) (100,4.164e-02) (200,1.619e-02) (500,4.675e-03)
		};
		\addlegendentry{$m_1=4s$, $\kappa=20$}
		\addplot[blue,mark=triangle,mark size=2.5pt,mark options={solid}] coordinates {
(25,1.608e-01) (50,7.551e-02) (100,3.873e-02) (200,1.599e-02) (500,4.613e-03)
		};
		\addlegendentry{$m_1=8s$, $\kappa=10$}
		\end{axis}
		\end{tikzpicture}
	}
	\caption{Number of samples, runtime, number of iterations, relative $L^2(\tilde{\D},\mu_\mathrm{F,L})$ error
		vs.\ sparsity $s\in\{25,50,100,200\}$ for mixed Fourier+Legendre basis and test function~\eqref{equ:f_mixed:10}.}
	\label{fig:approx_sparse:mixed_l_d10:samples_runtime_iter_success_vs_s}
\end{figure}
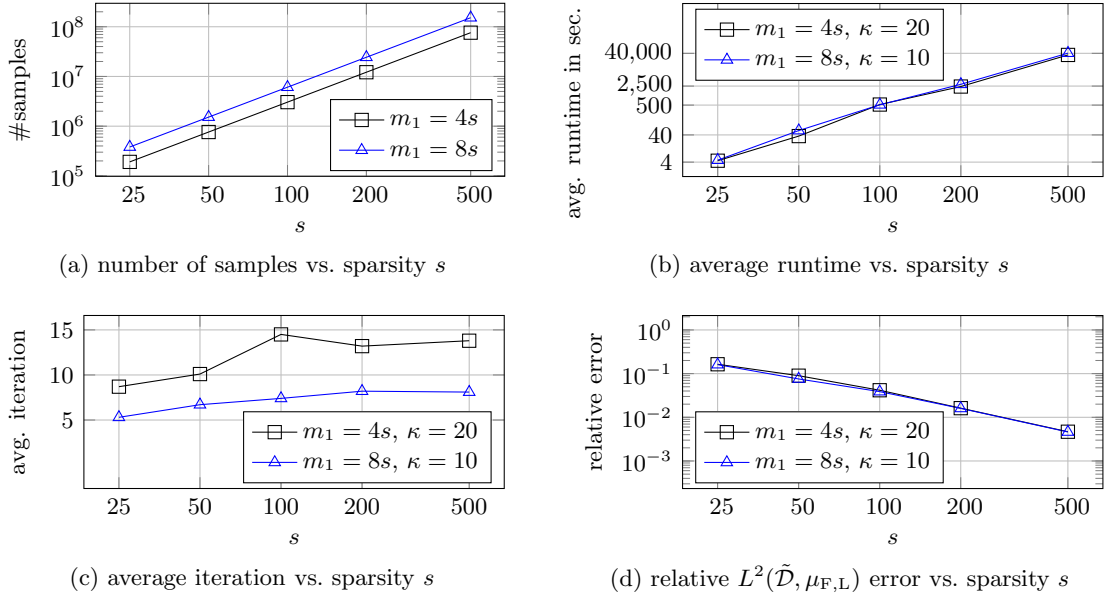

\FloatBarrier

\section*{Acknowledgements}
Mark Iwen was supported in part by NSF DMS-1912706, and would like to dedicate this paper to his ever bright, hard working, and spirited wife Tsveta, and to the prosperity of their newborn daughter Evgenia.  Evgenia -- I am anxious to know you are healthy, eager to see you are happy, and already sad at the distant prospect your moving out.  May you be more like your mother than like me for your own sake!

\bibliographystyle{abbrv}
\bibliography{BosuEDs}

\begin{thebibliography}{10}

\bibitem{adcock2017infinite}
B.~Adcock.
\newblock Infinite-dimensional $\ell^{1}$ minimization and function
  approximation from pointwise data.
\newblock {\em Constructive Approximation}, 45(3):345--390, 2017.

\bibitem{Adcock2017}
B.~Adcock, S.~Brugiapaglia, and C.~G. Webster.
\newblock Compressed sensing approaches for polynomial approximation of
  high-dimensional functions.
\newblock In {\em Compressed Sensing and its Applications}, pages 93--124.
  Springer International Publishing, 2017.

\bibitem{bailey2012design}
J.~Bailey, M.~A. Iwen, and C.~V. Spencer.
\newblock On the design of deterministic matrices for fast recovery of
  {Fourier} compressible functions.
\newblock {\em SIAM Journal on Matrix Analysis and Applications},
  33(1):263--289, 2012.

\bibitem{bittens2018sparse}
S.~Bittens and G.~Plonka.
\newblock Sparse fast {DCT} for vectors with one-block support.
\newblock {\em Numerical Algorithms}, pages 1--35, 2018.

\bibitem{bittens2019deterministic}
S.~Bittens, R.~Zhang, and M.~A. Iwen.
\newblock A deterministic sparse {FFT} for functions with structured {Fourier}
  sparsity.
\newblock {\em Advances in Computational Mathematics}, 45:519--561, 2019.

\bibitem{doi:10.1137/1.9781611971484}
A.~Bj{\"o}rck.
\newblock {\em Numerical Methods for Least Squares Problems}.
\newblock Society for Industrial and Applied Mathematics (SIAM), 1996.

\bibitem{2017arXiv170101671B}
J.-L. {Bouchot}, H.~{Rauhut}, and C.~{Schwab}.
\newblock {Multi-level Compressed Sensing Petrov-Galerkin discretization of
  high-dimensional parametric PDEs}.
\newblock {\em ArXiv e-prints}, 2017.

\bibitem{bungartz_griebel_2004}
H.-J. Bungartz and M.~Griebel.
\newblock Sparse grids.
\newblock {\em Acta Numerica}, 13:147--269, 2004.

\bibitem{chkifa2016polynomial}
A.~Chkifa, N.~Dexter, H.~Tran, and C.~Webster.
\newblock Polynomial approximation via compressed sensing of high-dimensional
  functions on lower sets.
\newblock {\em Mathematics of Computation}, 87(311):1415--1450, 2018.

\bibitem{choi2016multi}
B.~Choi, A.~Christlieb, and Y.~Wang.
\newblock Multi-dimensional sublinear sparse {Fourier} algorithm.
\newblock {\em arXiv preprint arXiv:1606.07407}, 2016.

\bibitem{2019arXiv190703692C}
B.~{Choi}, A.~{Christlieb}, and Y.~{Wang}.
\newblock {Multiscale High-Dimensional Sparse Fourier Algorithms for Noisy
  Data}.
\newblock {\em arXiv e-prints}, page arXiv:1907.03692, 2019.

\bibitem{choi2018sparse}
B.~Choi, M.~Iwen, and F.~Krahmer.
\newblock Sparse harmonic transforms: {A} new class of sublinear-time
  algorithms for learning functions of many variables.
\newblock {\em arXiv preprint arXiv:1808.04932}, 2018.

\bibitem{cohen2009compressed}
A.~Cohen, W.~Dahmen, and R.~DeVore.
\newblock Compressed sensing and best $k$-term approximation.
\newblock {\em Journal of the American Mathematical Society}, 22(1):211--231,
  2009.

\bibitem{Dahlquist:2008:NMS:1383510}
G.~Dahlquist and A.~Bj{\"o}rck.
\newblock {\em Numerical Methods in Scientific Computing, Volume I}.
\newblock Society for Industrial and Applied Mathematics (SIAM), Philadelphia,
  PA, USA, 2008.

\bibitem{duarte2012kronecker}
M.~F. Duarte and R.~G. Baraniuk.
\newblock {Kronecker} compressive sensing.
\newblock {\em IEEE Transactions on Image Processing}, 21(2):494--504, 2012.

\bibitem{dung2016hyperbolic}
D.~D{\~u}ng, V.~N. Temlyakov, and T.~Ullrich.
\newblock Hyperbolic cross approximation.
\newblock {\em arXiv preprint arXiv:1601.03978}, 2016.

\bibitem{foucart2013mathematical}
S.~Foucart and H.~Rauhut.
\newblock {\em A Mathematical Introduction to Compressive Sensing}.
\newblock Springer New York, 2013.

\bibitem{gilbert2019sparse}
A.~Gilbert, A.~Gu, C.~Re, A.~Rudra, and M.~Wootters.
\newblock {Sparse Recovery for Orthogonal Polynomial Transforms}.
\newblock {\em arXiv preprint arXiv:1907.08362}, 2019.

\bibitem{iwen2007empirical}
A.~Gilbert, M.~Iwen, and M.~Strauss.
\newblock Empirical evaluation of a sub-linear time sparse {DFT} algorithm.
\newblock {\em Communications in Mathematical Sciences}, 5(4):981--998, 2007.

\bibitem{gilbert2014recent}
A.~C. Gilbert, P.~Indyk, M.~A. Iwen, and L.~Schmidt.
\newblock Recent developments in the sparse {Fourier} transform: {A} compressed
  {Fourier} transform for big data.
\newblock {\em IEEE Signal Processing Magazine}, 31(5):91--100, 2014.

\bibitem{gilbert2005improved}
A.~C. Gilbert, S.~Muthukrishnan, and M.~Strauss.
\newblock Improved time bounds for near-optimal sparse {Fourier}
  representations.
\newblock In {\em Proceedings of SPIE}, volume 5914, page 59141A, 2005.

\bibitem{hassanieh2012simple}
H.~Hassanieh, P.~Indyk, D.~Katabi, and E.~Price.
\newblock Simple and practical algorithm for sparse {Fourier} transform.
\newblock In {\em Proceedings of the Twenty-Third Annual ACM-SIAM Symposium on
  Discrete Algorithms}, pages 1183--1194. Society for Industrial and Applied
  Mathematics (SIAM), 2012.

\bibitem{hu2017rapidly}
X.~Hu, M.~Iwen, and H.~Kim.
\newblock Rapidly computing sparse {Legendre} expansions via sparse {Fourier}
  transforms.
\newblock {\em Numerical Algorithms}, 74(4):1029--1059, 2017.

\bibitem{iwen2008deterministic}
M.~A. Iwen.
\newblock A deterministic sub-linear time sparse {Fourier} algorithm via
  non-adaptive compressed sensing methods.
\newblock In {\em Proceedings of the nineteenth annual ACM-SIAM symposium on
  Discrete algorithms}, pages 20--29. Society for Industrial and Applied
  Mathematics (SIAM), 2008.

\bibitem{iwen2010combinatorial}
M.~A. Iwen.
\newblock Combinatorial sublinear-time {Fourier} algorithms.
\newblock {\em Foundations of Computational Mathematics}, 10(3):303--338, 2010.

\bibitem{iwen2013improved}
M.~A. Iwen.
\newblock Improved approximation guarantees for sublinear-time {Fourier}
  algorithms.
\newblock {\em Applied and Computational Harmonic Analysis}, 34(1):57--82,
  2013.

\bibitem{kammerer2017high}
L.~K{\"a}mmerer, D.~Potts, and T.~Volkmer.
\newblock High-dimensional sparse {FFT} based on sampling along multiple rank-1
  lattices.
\newblock {\em arXiv preprint arXiv:1711.05152}, 2017.

\bibitem{kapralov2016sparse}
M.~Kapralov.
\newblock Sparse {Fourier} transform in any constant dimension with
  nearly-optimal sample complexity in sublinear time.
\newblock In {\em Proceedings of the forty-eighth annual ACM symposium on
  Theory of Computing}, pages 264--277. ACM Press, 2016.

\bibitem{2019arXiv190210633K}
M.~{Kapralov}, A.~{Velingker}, and A.~{Zandieh}.
\newblock {Dimension-independent Sparse Fourier Transform}.
\newblock {\em arXiv e-prints}, page arXiv:1902.10633, 2019.

\bibitem{Mansour:1992:RIA:646246.684842}
Y.~Mansour.
\newblock Randomized interpolation and approximation of sparse polynomials.
\newblock In {\em Proceedings of the 19th International Colloquium on Automata,
  Languages and Programming}, ICALP '92, pages 261--272, London, UK, 1992.
  Springer-Verlag.

\bibitem{merhi2017new}
S.~Merhi, R.~Zhang, M.~A. Iwen, and A.~Christlieb.
\newblock A new class of fully discrete sparse {Fourier} transforms: {Faster}
  stable implementations with guarantees.
\newblock {\em Journal of Fourier Analysis and Applications}, 25(3):751--784,
  2019.

\bibitem{morotti2017explicit}
L.~Morotti.
\newblock Explicit universal sampling sets in finite vector spaces.
\newblock {\em Applied and Computational Harmonic Analysis}, 43(2):354--369,
  2017.

\bibitem{needell2009cosamp}
D.~Needell and J.~A. Tropp.
\newblock {CoSaMP}: Iterative signal recovery from incomplete and inaccurate
  samples.
\newblock {\em Applied and Computational Harmonic Analysis}, 26(3):301--321,
  2009.

\bibitem{potts2016sparse}
D.~Potts and T.~Volkmer.
\newblock Sparse high-dimensional {FFT} based on rank-1 lattice sampling.
\newblock {\em Applied and Computational Harmonic Analysis}, 41(3):713--748,
  2016.

\bibitem{potts2017multivariate}
D.~Potts and T.~Volkmer.
\newblock Multivariate sparse {FFT} based on rank-1 {Chebyshev} lattice
  sampling.
\newblock In {\em 2017 International Conference on Sampling Theory and
  Applications (SampTA)}, pages 504--508. IEEE, 2017.

\bibitem{R07}
H.~Rauhut.
\newblock Random sampling of sparse trigonometric polynomials.
\newblock {\em Applied and Computational Harmonic Analysis}, 22(1):16--42,
  2007.

\bibitem{R12}
H.~Rauhut and R.~Ward.
\newblock Sparse {Legendre} expansions via $\ell_1$-minimization.
\newblock {\em Journal of Approximation Theory}, 164(5):517--533, 2012.

\bibitem{schwab2006karhunen}
C.~Schwab and R.~A. Todor.
\newblock Karhunen--{Lo\`eve} approximation of random fields by generalized
  fast multipole methods.
\newblock {\em Journal of Computational Physics}, 217(1):100--122, 2006.

\bibitem{segal2013improved}
B.~Segal and M.~A. Iwen.
\newblock Improved sparse {Fourier} approximation results: {Faster}
  implementations and stronger guarantees.
\newblock {\em Numerical Algorithms}, 63(2):239--263, 2013.

\bibitem{shen2010sparse}
J.~Shen and L.-L. Wang.
\newblock Sparse spectral approximations of high-dimensional problems based on
  hyperbolic cross.
\newblock {\em SIAM Journal on Numerical Analysis}, 48(3):1087--1109, 2010.

\end{thebibliography}
  
\end{document}